\documentclass[11pt]{article}
\usepackage[dvips, letterpaper, margin = 1.0in]{geometry}
\usepackage{sty}

\date{}
\title{Low-degree estimation thresholds in planted hypergraphs \\
and tensor PCA}

\usepackage{authblk}
\author[1]{Daniel Fu \thanks{Email: \textit{daniel\textunderscore fu1@brown.edu}.}}
\author[1]{Youngtak Sohn \thanks{Email: \textit{youngtak\textunderscore sohn@brown.edu}.}}
\affil[1]{Division of Applied Mathematics, Brown University}

\begin{document}
\maketitle
\begin{abstract}
A central question in high-dimensional statistics is to understand \textit{statistical--computational gaps}: regimes in which recovering a hidden signal is information-theoretically possible but conjectured to be computationally intractable. The low-degree framework offers a concrete way to study this gap by restricting attention to estimators that are polynomials of degree at most $D$ in the observed data. In this paper, we study low-degree estimation in planted dense subhypergraph, sparse tensor PCA, and tensor PCA with a general prior.

For the planted dense subhypergraph model on $n$ vertices, we identify two regimes depending on whether the planted set is larger or smaller than $\sqrt{n}$. Above this scale, we identify a sharp threshold for low-degree estimation. Below this scale, we establish hardness in the regimes predicted by prior work, thereby resolving a question of Schramm and Wein (2022) and Sohn and Wein (2025). For sparse tensor PCA, we identify an analogous sharp phase transition. For tensor PCA with a general prior, we prove a low-degree estimation lower bound at the critical signal scale, matching the degree--signal tradeoff suggested by prior work.

Our lower bounds apply to degree $D=n^{\delta}$, where $n$ is the dimension and $\delta>0$ is a constant, and we complement them with corresponding low-degree upper bounds. In addition, for planted dense subhypergraph and sparse tensor PCA above the $\sqrt{n}$ scale, we convert our upper bounds into polynomial-time algorithms that achieve almost exact recovery above the sharp threshold, yielding polynomial-time algorithms succeeding up to this threshold. Our proofs extend the framework of Sohn and Wein (2025) through a conditional variant that yields the correct signal-to-noise ratio in settings where the unconditional approach is insufficient.

\end{abstract}

\section{Introduction}\label{sec:intro}
A fundamental task in high-dimensional statistics is to recover a hidden signal buried in a large, noisy dataset. Alongside the \emph{information-theoretic} question of determining the weakest signal strength at which recovery is possible by \emph{any} estimator, the high dimensionality of these problems introduces a \emph{computational} challenge: whether recovery can be achieved by an efficient algorithm of practical runtime. When there is a gap between the signal strength needed for information-theoretic recovery and that needed for computationally efficient recovery, we say there is a \emph{statistical--computational gap}. Understanding the nature of such gaps is one of the central goals of high-dimensional statistics and average-case complexity.

In this work, we study three closely-related models in which a signal is buried in a random tensor or random hypergraph, each exhibiting a statistical--computational gap: 

\begin{itemize}
     \item \textbf{Planted dense subhypergraph}: For a sparsity parameter $\rho\in [0,1]$ and edge probabilities $0\leq q_0\leq q_1\leq 1$, we observe an $r$-uniform random hypergraph whose adjacency tensor $Y$ is generated as follows. Draw $\theta\in \{0,1\}^n$ with i.i.d.\ $\Ber(\rho)$ entries. Conditional on $\theta$, independently for each hyperedge $e=\{i_1,\ldots,i_r\}$ of distinct vertices, draw 
     \[
     Y_{e}\sim \Ber\bigg(q_0+(q_1-q_0)\prod_{j=1}^{r}\theta_{i_j}\bigg).
     \]
     Thus, hyperedges whose vertices all lie in the planted set $\{1\leq i\leq n:\theta_i=1\}$ appear with probability $q_1$, while all other hyperedges appear with background probability $q_0$.
     \item \textbf{Sparse tensor PCA}: For a sparsity parameter $\rho\in [0,1]$ and signal-to-noise parameter $\la\geq 0$, we observe the $r$-th order random tensor
    \[
      Y= \lambda\theta^{\otimes r}+W\,,
    \]
    where the planted signal $\theta\in \{0,1\}^n$ has i.i.d.\ $\Ber(\rho)$ entries and $W\in (\R^n)^{\otimes r}$ is a Gaussian noise tensor with i.i.d.\ entries $W_{i_1,\ldots,i_r}\sim \cN(0,1)$.
     \item \textbf{Tensor PCA with general prior}: This is the standard model of a low-rank tensor corrupted by Gaussian noise. We again observe $Y= \lambda\theta^{\otimes r}+W$, but now the planted signal $\theta$ has i.i.d.\ entries drawn from a general prior $\pi$ with mean~$0$ and variance~$1$.
\end{itemize}

In each model, we consider the asymptotic regime $n\to\infty$. The parameters $\rho$, $\la$, $q_0$, and $q_1$ may scale with $n$, while the tensor order $r$ and the prior $\pi$ are held fixed. All model parameters are assumed known to the statistician.

All three models have been studied extensively in the literature. In particular, statistical--computational gaps in \emph{hypothesis testing} (a.k.a. \emph{detection}) task---testing whether observation $Y$ is drawn from the planted distribution as defined above or from a pure noise distribution---has been studied for all three models (we refer to~\cite{hopkins-thesis, ld-notes, ld-survey} for extensive literature on detection) and is by now relatively well-understood.

The present work instead focuses on the \emph{estimation} task: given the observed data $Y$, recover the latent signal $\theta$.\footnote{For tensor PCA with general prior, when $r$ is even and $\pi$ is symmetric, the model is invariant under $\theta \leftrightarrow -\theta$, so direct estimation of $\theta$ is impossible. In that case, one instead estimates $\theta\theta^{\top}$ or $\theta^{\otimes r}$. Our lower bound in Theorem~\ref{thm:general:PCA} rules out both forms of recovery.} Estimation has also received significant attention---we defer a thorough literature review to the discussion of each model below. In particular, for the matrix case $r=2$,~\cite{arxiv-version} pinned down sharp computational thresholds for the corresponding matrix versions of all three models. However, for tensor/hypergraph models $r\geq 3$, the picture is much less complete, and clarifying it motivates the present work. A central question we address is
\begin{quote}
   \emph{What are the limits of computationally efficient estimation in these tensor and hypergraph models, and can the exact boundary be pinned down?}
\end{quote}
For tensor PCA with a fixed prior, it is widely believed that there is a smooth tradeoff between signal and runtime~\cite{sos-detect, ld-notes, kunisky24tensor}. For sparse tensor PCA and planted dense subhypergraph, by contrast, the form of the threshold has been unclear. To study this question, we employ the \emph{low-degree polynomial framework} described in the next subsection. Perhaps surprisingly, we show that a \emph{sharp} phase transition occurs in these two models for low-degree polynomial estimation when $\rho \gg n^{-1/2}$, with a critical signal-to-noise ratio that we determine explicitly. In the complementary regime $\rho\ll n^{-1/2}$ of planted dense subhypergraph, we establish the estimation hardness predicted by previous work, resolving an open question of~\cite{SW-22, arxiv-version}; this result is new even in the graph case $r=2$.

\subsection{Low-degree estimation framework}
\label{subsec:intro:low-degree}
Each of the three models described above is an \emph{average-case} problem: the goal is to find an estimator that runs in polynomial time and recovers the signal with high probability when the input $Y$ is sampled from the model, rather than an estimator that succeeds uniformly over all inputs. Proving unconditional computational lower bounds for such problems---for instance, showing that no polynomial-time algorithm can estimate $\theta$ with probability $1-o(1)$ as $n\to\infty$---appears beyond the reach of current complexity theory, so one seeks evidence for hardness through a variety of restricted computational frameworks and structural barriers. Prominent examples include the sum-of-squares hierarchy~\cite{sos-survey}, the statistical query model~\cite{sq-ld}, reductions from conjecturally hard problems~\cite{bresler-2020}, and geometric obstructions based on the overlap gap property~\cite{ogp-survey}. Among these, the \emph{low-degree polynomial framework} has been successful, with a strong track record of correctly predicting the computational thresholds in average-case inference problems.
 
The low-degree framework was first developed for detection~\cite{HS-bayesian,sos-detect}, where the central object is the low-degree likelihood ratio~\cite{ld-notes}. Its estimation analog was introduced by Schramm and Wein~\cite{SW-22}, who defined the \emph{low-degree MMSE} as the smallest mean-squared error achievable by a polynomial estimator of degree at most $D$. Suppose that the goal is to estimate a scalar $x$ from the observed data $Y$: in planted dense subhypergraph and sparse tensor PCA, $x$ will be the first coordinate of the signal $\theta_1$. The degree-$D$ minimum-mean-squared-error (MMSE) is defined by
\begin{equation}\label{eq:def:low:MMSE}
  \MMSE_{\leq D} := \inf_{\substack{f \in \RR[Y] \\ \deg(f) \le D}} \EE\!\big[(f(Y)-x)^2\big]\,,
\end{equation}
where $\RR[Y]$ denotes the space of multivariate polynomials in the entries of $Y$, i.e. $(Y_{i_1,\ldots, i_r})_{i_1,\ldots, i_r\in [n]}$.\footnote{There is also an equivalent vector analog, in which one minimizes $\EE[\|f(Y)-\theta\|^2]$ over vector-valued polynomial estimators $f$ of degree at most $D$ (see e.g.~\cite[Section 1]{SW-22} and~\cite[Section 2.2]{arxiv-version}).}
A useful benchmark is the degree-$0$ case, where $\MMSE_{\leq 0}$ is the error of the trivial estimator $f(Y)=\E[x]$, so that $\MMSE_{\leq 0}=\Var(x)$. Accordingly, $\MMSE_{\leq D}\geq (1-o(1))\MMSE_{\leq 0}$ says that degree-$D$ polynomials have, asymptotically, no better performance than trivial estimation. This is the form of hardness that we establish, with degrees as large as $D=n^{\delta}$ for some constant $\delta>0$.

For the types of average-case problems considered here, degree-$O(\log n)$ polynomials are often already rich enough to match the performance of the best currently known polynomial-time algorithms. Consequently, when one can show that polynomials of super-logarithmic degree fail, this is typically viewed as evidence that no polynomial-time algorithm should succeed. In this sense, the term ``low-degree'' usually refers to degree-$O(\log n)$. More generally, the guiding heuristic is that degree-$D$ polynomials correspond to algorithms with runtime $n^{\widetilde{O}(D)}$ where $\widetilde{O}(\cdot)$ suppresses poly-logarithmic factors in $n$. This heuristic was formalized in work of Hopkins~\cite[Hypothesis 2.1.5, Conjecture 2.2.4]{hopkins-thesis} for detection problems, but it is now known not to hold in full generality~\cite{buhai25false}. Determining the correct general formulation remains an active open problem; see the recent survey~\cite[Section 6]{ld-survey} for a broader discussion of how low-degree lower bounds should be interpreted. Nonetheless, the framework's predictions have consistently matched the conjectured thresholds for natural high-dimensional models such as planted clique~\cite{SW-22} and community detection~\cite{arxiv-version}, and we accordingly view our lower bounds as strong evidence for computational hardness.

\subsection{Our contributions}
\subsubsection{Planted dense subhypergraphs}
We briefly describe our main results for planted dense subhypergraph; the formal statements appear in Section~\ref{sec:results}. Consider the $r$-uniform planted dense subhypergraph model described above, with sparsity $\rho$, background edge probability $q_0$, and planted edge probability $q_1$. The behavior of low-degree estimation differs markedly depending on whether the planted set is \emph{large}, of size $\gg\sqrt{n}$, or \emph{small}, of size $\ll\sqrt{n}$, and we address these two regimes separately. For ease of exposition, we restrict to the polynomial scaling 
\[
\rho=n^{\xi-1}\,,\qquad q_1=n^{-a}\,,\qquad q_0=n^{-b}
\]
with $0<a\leq b$ and $\xi,a,b$ fixed. Our first main result identifies a sharp phase transition when $\xi>1/2$, in terms of the signal-to-noise ratio
    \[
    \SNR:=\frac{e }{(r - 2)!}\frac{(n\rho^2)^{r-1}(q_{1} - q_{0})^{2}}{q_{0}(1 - q_{0})}\,.
    \]
\begin{theorem}[Large planted set, informal; see Theorem~\ref{thm:planted:hypergraph}]\label{thm:intro:large-hypergraph}
Fix $\xi \in (1/2,1)$ and $\eps>0$. There exists a constant $C\equiv C(\eps, r)>0$ such that the following holds for all $n$ sufficiently large.
\begin{enumerate}
        \item[(a)] If $\eps\in (0, 1)$ and $\SNR \leq 1-\eps$ and $D \leq n^{2\xi-1}/C$, then degree-$D$ polynomials cannot improve on trivial estimation: 
\[ \MMSE_{\le D} := \inf_{\substack{f \in \RR[Y] \\ \deg(f) \le D}} \EE[(f(Y)-\theta_1)^2] \ge \rho - C \rho^2\,. \]
   \item[(b)] If $\SNR \geq 1+\eps$ and $b<r-1$, then $\MMSE_{\leq C\log n}=o(\rho)$.
\end{enumerate}
\end{theorem}

Observe that the trivial estimator $f(Y) = \E[\theta_1] = \rho$ has error $\MMSE_{\le 0} = \rho - \rho^2$. Part~(a) says that when $\SNR \le 1 - \eps$, estimating $\theta_1$ with polynomials in $Y$ of degree-$n^{2\xi - 1}$ is asymptotically no better than trivial estimation. Part~(b) says that when $\SNR \ge 1 + \eps$, a degree-$O(\log n)$ polynomial has error $o(\rho)$. Since $\E\theta_1=\rho$, an error of $o(\rho)$ means that the planted set is recovered almost perfectly. In Theorem~\ref{thm:algorithm} below, we turn this polynomial into a polynomial-time algorithm using the color coding trick~\cite{color-coding}. Thus, the sharp transition at $\SNR = 1$ from hard to easy may be viewed as a computational version of the all-or-nothing phenomenon~\cite{macris2020all, niles2020all, mossel2023sharp}. Our theorem holds for all $r \ge 2$ with critical constant $e/(r-2)!$, recovering the result of Sohn and Wein~\cite{arxiv-version} in the graph case $r = 2$. For all $r \geq 2$, the polynomial-time algorithm achieving this threshold appears to be new. See Section~\ref{subsec:res:planted:hypergraph} for further discussion and comparison with prior work.
 
Theorem~\ref{thm:intro:large-hypergraph}-(b) also holds for $\rho \le n^{-1/2}$ (see Theorem~\ref{thm:planted:hypergraph}), but in this regime the following result gives the correct computational scale.
 
\begin{theorem}[Small planted set, informal; see Theorem~\ref{thm:small:planted:hypergraph}]
\label{thm:intro:small-hypergraph}
Fix $\xi\in (0,1/2]$ and $0<a<b<r-1$. There exist constants $C,\delta>0$ depending only on $r,a,b,\xi$ such that the following holds for all $n$ sufficiently large.
\begin{enumerate}
    \item[(a)] If $a>b\xi$, then for $D=n^{\delta}$, $\MMSE_{\leq D}=(1-o(1))\rho$.
    \item[(b)] If $a<b\xi$, then $\MMSE_{\leq C}=o(\rho)$.
    \end{enumerate}
\end{theorem}
Reading parts~$(a)$ and~$(b)$ as before, Theorem~\ref{thm:intro:small-hypergraph} identifies the computational boundary for estimation as the curve $a=b\xi$, resolving a question (for $r=2$) raised by Schramm and Wein~\cite{SW-22} and reiterated by Sohn and Wein~\cite{arxiv-version}. For detection, the analogous low-degree hardness was established by Dhawan, Mao, and Wein~\cite{DMW-23}. The proof relies on a new conditioning argument tailored to low-degree estimation, which we expect to be of independent interest. See Section~\ref{subsec:res:planted:hypergraph} for further discussion and Section~\ref{subsec:proof:conditioning} for an overview of the proof.

\subsubsection{Other models}
For sparse tensor PCA, we prove an analogous sharp phase transition for low-degree estimation in the regime $\rho\gg n^{-1/2}$. The relevant signal-to-noise ratio is $\widetilde{\SNR}=\frac{e(n\rho^2)^{r-1}\la^2}{(r-2)!}$. In particular, when $n^{-1/2}\ll \rho\ll 1$, if $\widetilde{\SNR}<1$, then degree-$n^{\delta}$ polynomials cannot asymptotically improve upon the trivial estimator $f(Y)=\rho$. On the other hand, if $\widetilde{\SNR}>1$, then there exists a polynomial of degree-$O(\log n)$ that achieves near-perfect recovery. As in the planted dense subhypergraph model, this estimator can be converted into a polynomial-time algorithm. For $r=2$, the same low-degree phase transition was established by Sohn and Wein~\cite{arxiv-version}, and matching polynomial-time recovery at the same threshold was achieved by the AMP algorithm of Hajek, Wu, and Xu~\cite{submatrix-amp}. For $r \ge 3$, the threshold was predicted heuristically, via an AMP analysis, by Corinzia et al.~\cite{corinzia22statistical}; our contribution gives the first rigorous evidence for $r\geq 3$, establishing both low-degree hardness below the threshold and a polynomial-time recovery algorithm above it.

For tensor PCA with general prior ($r\geq 3$), the prediction is qualitatively different: it is expected that there is a smooth tradeoff between the degree $D$ and the signal strength for weak recovery, rather than a sharp threshold. Earlier work~\cite{sos-detect, hopkins-thesis, ld-notes, kunisky24tensor} predicts that detection requires $\lambda\asymp n^{-r/4}D^{-(r-2)/4}$. We rigorously establish the corresponding low-degree estimation lower bound for a broad class of priors $\pi$, which includes all sub-exponential distributions: if $\lambda\lesssim n^{-r/4}D^{-(r-2)/4}$, then degree-$n^{\delta}$ polynomials cannot asymptotically improve upon trivial estimation. For the Rademacher prior, the analogous detection lower bound was proved by Kunisky, Wein, and Bandeira~\cite{ld-notes}, and for the spherical prior, the corresponding estimation lower bound was proved by Kunisky, Moore, and Wein~\cite{kunisky24tensor} (among other results). Our work treats a general class of independent priors under mild moment assumptions, via a different proof method based on cumulant bound of Schramm and Wein~\cite{SW-22}.

\section{Main results}\label{sec:results}
In all the models we consider, the observation is an order-$r$ tensor $Y \in (\R^n)^{\otimes r}$ and the estimand is a scalar $x \in \R$; the specific choice of $x$ will be given for each model in the subsections that follow. Recall from~\eqref{eq:def:low:MMSE} that $\MMSE_{\leq D}$ is the smallest mean-squared error achievable by polynomials in $Y$ of degree at most $D$. It is convenient to work with the \emph{low-degree correlation}. Let $\R_D[Y]$ denote the space of all (multivariate) polynomials in $Y$ with degree at most $D$, i.e.  $f \in \R[Y]$ with $\deg(f) \leq D$, and define
\begin{equation}\label{eq:def:Corr}
	\Corr_{\leq D} := \sup_{f \in \R_D[Y]} \frac{\E[f(Y) \cdot x]}{\sqrt{\E[f(Y)^2]\E[x^2]}} \in [0, 1]\,.
\end{equation}
The two quantities are related by the following identity.

\begin{fact}[{\cite[Fact~1.1]{SW-22}}]\label{fact:corr:mmse}
$\MMSE_{\leq D} = (1 - \Corr_{\leq D}^2)\E[x^2]$.
\end{fact}

For our lower bounds, we rule out \emph{weak recovery}. We say that weak recovery is achievable by degree-$D$ polynomials if $\MMSE_{\leq D} \leq (1 - \Omega(1))\MMSE_{\leq 0}$. Since $\MMSE_{\leq 0} = \Var(x)$, this means that there exists a polynomial of degree at most $D$ that improves on the trivial estimator $f(Y) = \E[x]$ by a constant fraction. By Fact~\ref{fact:corr:mmse}, weak recovery implies $\Corr_{\leq D} = \Omega(1)$, so $\Corr_{\leq D} = o(1)$ rules out weak recovery. Our hardness results establish the latter for degrees as large as $D = n^{\delta}$ for some constant $\delta > 0$.

For our upper bounds, we establish the stronger property of \emph{strong recovery}: there exists a polynomial $f \in \R_D[Y]$ with $\E[(f(Y) - x)^2] = o(\Var(x))$, equivalently $\MMSE_{\leq D} = o(\MMSE_{\leq 0})$, or by Fact~\ref{fact:corr:mmse}, $\Corr_{\leq D} = 1 - o(1)$. We achieve this with polynomials of degree $D = O(\log n)$.
\subsection{Planted dense subhypergraph}\label{subsec:res:planted:hypergraph}
We briefly recall the planted dense subhypergraph model introduced in Section~\ref{sec:intro}. Given parameters $\rho\in (0,1)$ and $0<q_0< q_1\leq 1$\footnote{The assumption $q_0 <q_1$ is without loss of generality, since the complement hypergraph $Y' = \mathbf{1} - Y$ is planted dense subhypergraph  with $(q_0, q_1)$ replaced by $(1-q_0, 1-q_1)$.} the latent signal $\theta = (\theta_i)_{i \in [n]}$ has i.i.d. $\Ber(\rho)$ entries. Conditioned on $\theta$, independently for each hyperedge $e=\{i_1,\ldots,i_r\}$ of distinct vertices, we draw $  Y_{e}\sim \Ber(q_0+(q_1-q_0)\prod_{j=1}^{r}\theta_{i_j})$. The goal is to estimate $x = \theta_1$. Define
\[
	\lambda := \frac{q_1 - q_0}{\sqrt{q_0(1 - q_0)}}\,.
\]
\begin{theorem}\label{thm:planted:hypergraph}
Consider the $r$-uniform planted dense subhypergraph model. For any $\eps>0$ there is a constant $C = C(\eps, r) > 0$ such that the following holds.
\begin{itemize}
    \item[(a)] If $\eps \in (0, 1)$ and
    \[
        \frac{en^{r - 1} \rho^{2r - 2} \lambda^{2}}{(r - 2)!(1 - \rho)^{r - 1}} \leq 1 - \eps\,, \quad D^{r - 1} \leq \frac{1}{C\lambda^{2}}\,,
    \]
    then
    \[
        \Corr_{\leq D} \leq C\sqrt{\frac{\rho}{1-\rho}}\,.
    \]
    \item[(b)] If
    \[
        \rho = \omega\left(n^{-1} (\log n)^{6 + 3/(r-1)}\right), \quad \rho = o\left((\log{n})^{-6(r-1) - 3}\right), \quad q_{0} = \omega\left(n^{-(r-1)}(\log{n})^{12(r-1)+6}\right)\,,
    \]
    and
    \[
        \frac{en^{r - 1} \rho^{2r - 2} \lambda^{2}}{(r - 2)!} \geq 1 + \eps\,,
    \]
    then
    \[
        \Corr_{\leq C\log n} = 1 - o(1) \quad \text{as $n \rightarrow \infty$}.
    \]
\end{itemize}
\end{theorem}
The upper bound in Theorem~\ref{thm:planted:hypergraph}-(b) can be converted into a polynomial-time recovery algorithm via the color coding trick~\cite{color-coding, HS-bayesian, mao2024testing}; see Section~\ref{sec:polytime:algorithm} for the explicit algorithm and its recovery guarantee.

\begin{theorem}\label{thm:small:planted:hypergraph}
Consider the $r$-uniform planted dense subhypergraph model with sparsity $\rho=n^{\xi-1}$ for $0<\xi\leq 1/2$ and $q_1=n^{-a}$ and $q_0=n^{-b}$ where $0<a<b$. There exist constants $C, \delta_1,\delta_2 > 0$ that depend only on $r, a,b,\xi$ such that, as $n\to\infty$,
\begin{enumerate}
    \item[(a)] If $a>b\xi$, then for $D\leq n^{\delta_1}$ it holds that $\Corr_{\leq D} \leq C n^{-\delta_2}$. 
    \item[(b)] If $a<b\xi$ and $b<r-1$, then $\Corr_{\leq C}=1-o(1)$.
    \end{enumerate}
\end{theorem}

\paragraph{Discussion}
The planted dense subhypergraph problem has extensive literature. For $r=2$, it reduces to planted dense subgraph model, which has been studied from both statistical and computational perspectives. Statistical limits, both the information-theoretic threshold and performance of the maximum likelihood estimator, are established in~\cite{AV-info, VA-info, CX-info, hajek2017recovery}. On the computational side, hardness has been studied through low-degree polynomial lower bounds for detection~\cite{DMW-23} and for estimation~\cite{SW-22, arxiv-version}, and through average-case reductions~\cite{bresler-2020, bresler-2023}. There are also positive algorithmic results~\cite{log-density,ames-convex,CX-info,one-community-sparse}. The special case $r=2, q_1=1$ is the planted clique problem~\cite{jerrum, alon-clique, clique-e, sq-clique, sos-clique, GZ-clique}, a canonical benchmark for average-case hardness and a base for establishing reduction-based hardness in many inference problems~\cite{BR-reduction, BBH-reduction, bresler-2020}. For $r \geq 3$, relatively less is known. The special case $q_1 = 1$ is the hypergraphic planted clique problem; we refer to~\cite{luo20a} for a discussion of open problems related to its average-case hardness. For general $q_1 < 1$, statistical limits for weak recovery were determined by Yuan and Shang~\cite{yuan-2021}.

On the computational side, Luo and Zhang~\cite{LZ-tensor} studied the regime $\rho\geq n^{-1/2}$ from both directions: on the
hardness side~\cite[Theorem~20]{LZ-tensor} rules out low-degree
estimation under the condition $\SNR \leq n^{-\delta}$ for a
fixed $\delta > 0$ as $n\to\infty$, and on the algorithmic side~\cite[Proposition~1]{LZ-tensor} proves that an aggregated-SVD recovers the planted set
$\{i \in [n] : \theta_i = 1\}$ \emph{exactly} once $\SNR \geq c$ for some
constant $c > 0$.\footnote{Condition~(23) of \cite[Proposition~1]{LZ-tensor}
reads $\limsup_{n\to\infty}\log_n \SNR \geq 0$ in our normalization, but should
be read as $\SNR \geq c$ for a constant $c > 0$ based on its proof.}
Theorem~\ref{thm:planted:hypergraph} improves their result by determining the threshold at constant-level:
estimation is low-degree hard for $\SNR < 1$, with a matching upper bound for
$\SNR > 1$, locating the critical point with constant
$e/(r-2)!$. While preparing this paper, we became aware that the same value was predicted by the heuristic
AMP analysis of Corinzia, Penna, Szpankowski, and
Buhmann~\cite{corinzia22statistical} (see the discussion after Theorem~\ref{thm:sparse:PCA} for sparse tensor PCA). Theorem~\ref{thm:algorithm} in Section~\ref{subsec:res:efficient:algorithm} lifts the upper bound in part~(b) to a
polynomial-time algorithm via a color-coding trick~\cite{color-coding,
HS-bayesian, mao2024testing}, achieving almost-exact recovery, which appears to be the first to attain this sharp threshold for all
$r \geq 2$. Whether its guarantee can be strengthened to exact recovery is
an interesting question.

We expect the degree condition $D^{r-1} \leq 1/(C\lambda^2)$ in Theorem~\ref{thm:planted:hypergraph}-(a) to be optimal. A brute-force search over $\ell$-element subsets of $[n]$ with $\ell =\widetilde{O}(\lambda^{-2/(r-1)})$ yields a recovery algorithm of $\theta_1$ with runtime $\exp(\widetilde{O}(\lambda^{-2/(r-1)}))$ (where $\widetilde{O}$ hides the polylogarithmic factors), which is an adaptation of the subexponential time algorithms for the spiked Wigner model~\cite{subexp-sparse} and sparse tensor PCA~\cite{choo2021complexity}. This matches our lower bound under the degree-runtime heuristic (Section~\ref{subsec:intro:low-degree}). See the discussion following Theorem~\ref{thm:sparse:PCA} for the analogous matching in sparse tensor PCA.

In the regime $\rho \leq n^{-1/2}$, the best known algorithm for the planted
dense subgraph ($r = 2$) is the caterpillar-counting algorithm of Bhaskara et
al.~\cite{log-density}, which succeeds for $a < b\xi$. The matching hardness,
however, was only partially understood: Schramm and Wein~\cite{SW-22}
established low-degree estimation hardness in the region $a > b/2$,
leaving open the gap $b/2>a >b \xi$. Whether estimation hardness extends
throughout $a > b\xi$, or whether a better algorithm could succeed up to
$a < b/2$, was raised as an open question in~\cite{SW-22} and reiterated by Sohn
and Wein~\cite{arxiv-version}. For the corresponding detection task, Dhawan,
Mao, and Wein~\cite{DMW-23} had established hardness throughout $a > b\xi$; this
does not settle the estimation question, however, as low-degree estimation
hardness does not formally follow from its detection counterpart and is
generally more delicate to establish. In fact, recent work by Tang, Han, and Zhang~\cite{tang2026detection} shows that in a certain tensor cumulant inference problem, there are regimes in which computationally efficient estimation is possible while the detection task is low-degree hard.

Theorem~\ref{thm:small:planted:hypergraph} resolves this question, identifying
the estimation boundary as $a = b\xi$: part~(a) shows low-degree estimation
hardness for $a > b\xi$, and part~(b) establishes the existence of a matching
low-degree estimator for $a < b\xi$. The proof of part~(a) relies on a
conditioning argument adapted to low-degree estimation. Conditioning arguments
were previously developed for detection~\cite{fp, grp-testing, DMW-23}; we develop the analogous technique for estimation hardness, which may be
useful in other models where rare events inflate the unconditional moments. We remark that our proof can track the precise dependence on $\rho, q_0, q_1$, giving a
quantitative hardness statement
(see Remark~\ref{rem:small:hypergraph:quantitative}).

We also note that the computational boundary $a = b\xi$ in Theorem~\ref{thm:small:planted:hypergraph} admits an interpretation in terms of the modified Kesten--Stigum threshold introduced by Chin, Mossel, Sohn, and Wein~\cite{CMSW25} for the sparse stochastic block model (SBM)~\cite{decelle}; see \cite{abbe-survey-sbm} for a survey of the model. While the planted dense subgraph model has a single planted community of size roughly $n\rho = n^\xi$, replacing this with $q = 1/\rho = n^{1-\xi}$ symmetric communities of the same size, where each within-community (resp. between-community) edge is drawn with probability $q_1$ (resp. $q_0$), yields the analogous sparse SBM. In the regime $q\to\infty$, the two relevant parameters are the average degree of the entire graph $d \equiv n q_0 = n^{1-b}$ and the average degree of each community $d\lambda_1 \equiv n^\xi q_1 = n^{\xi - a}$ (up to a $1+o(1)$ factor). The classical Kesten--Stigum threshold $d\lambda_1^2 = 1$~\cite{KestenStigum:66} governs efficient weak recovery in the SBM with a constant number of communities~\cite{AS-acyclic, HM-tree, opt-bot, arxiv-version}. In the sparse regime $q \gg \sqrt{n}$, \cite{CMSW25} showed that efficient weak recovery is possible above a modified threshold in which the exponent $2$ is replaced by $1/\chi$, where $\chi$ is defined by $q = n^\chi$, so $\chi = 1 - \xi$. Namely, this threshold is given by $d\lambda_1^{1/\chi}=1$ (hiding $\log d$ factors), or equivalently $ d\lambda_1 = d^\xi$. Substituting our parameters yields $n^{\xi - a} = n^{\xi(1 - b)}$, i.e. $a = b\xi$, recovering the threshold curve of Theorem~\ref{thm:small:planted:hypergraph}. The corresponding low-degree estimation lower bound for the SBM was established by Carpentier, Giancola, Giraud, and Verzelen using an almost-orthogonal basis approach~\cite{carpentier2025low} (see also~\cite{carpentier2025phase}), which is a different method from the conditioning argument used here. This correspondence was one of the motivations for our study of the $\rho\leq n^{-1/2}$ regime.

\subsection{Sparse tensor PCA}
We first introduce a symmetric version of the sparse tensor PCA model defined in Section~\ref{sec:intro}, in which the Gaussian noise tensor is replaced by its symmetrization.

\begin{definition}[Symmetric sparse tensor PCA]\label{def:symmetric:sparse:PCA}
For parameters $\lambda \geq 0$ and $\rho \in [0, 1]$ the symmetric sparse tensor PCA model is as follows. Draw $\theta \in \{0, 1\}^n$ with $\theta_{i} \stackrel{iid}{\sim} \Ber(\rho)$ and $W \in (\R^{n})^{\otimes r}$ with $W_{i_1, \hdots, i_r} \stackrel{iid}{\sim} \cN(0, 1)$. Set $X=\lambda\theta^{\otimes r}$ and observe the symmetric $r$-th order random tensor
    \[
       Y= X + \Wsy \,,\quad\textnormal{where $\Wsy$ has entries $(\Wsy)_{i_1,\hdots,i_r}=\frac{1}{\sqrt{r!}}\sum_{\pi\in S_r} W_{i_{\pi(1)},\hdots, i_{\pi(r)}}$}\,,
    \]
    and $S_r$ is the symmetric group on $[r]$. The estimand is $x=\theta_1$.
\end{definition}
The asymmetric model in Section~\ref{sec:intro} is equivalent to the symmetric model from the estimation perspective: Lemma~\ref{lem:equivalence:of:noise:models} in the appendix shows that the degree-$D$ MMSE under the asymmetric model of Section~\ref{sec:intro} matches that under Definition~\ref{def:symmetric:sparse:PCA} after the substitution $\lambda \mapsto \sqrt{r!}\lambda$. In particular, when $\rho = o(1)$, the sharp threshold of Theorem~\ref{thm:sparse:PCA} below corresponds to $er(r-1)n^{r-1}\rho^{2r-2}\lambda^2 = 1$ in the asymmetric model.

\begin{theorem}\label{thm:sparse:PCA}
Consider the symmetric sparse tensor PCA model. For any $\eps>0$ there is a constant $C = C(\eps, r) > 0$ such that the following holds.
\begin{itemize}
    \item[(a)] If $\eps \in (0, 1)$ and
    \[
        \frac{en^{r - 1} \rho^{2r - 2} \lambda^{2}}{(r - 2)!(1 - \rho)^{r - 1}} \leq 1 - \eps\,, \quad D^{r - 1} \leq \frac{1}{C\lambda^{2}}\,,
    \]
    then
    \[
        \Corr_{\leq D} \leq C\sqrt{\frac{\rho}{1-\rho}}\,.
    \]
    \item[(b)] If
    \[
        \frac{en^{r - 1} \rho^{2r - 2} \lambda^{2}}{(r - 2)!} \geq 1 + \eps, \quad \rho = \omega\left(n^{-1} (\log n)^{6 + 3/(r-1)}\right), \quad \rho = o\left((\log n)^{-6(r-1) - 3}\right)\,,
    \]
    then
    \[
        \Corr_{\leq C\log n} = 1 - o(1) \quad \text{as $n \rightarrow \infty$}.
    \]
\end{itemize}
\end{theorem}

\paragraph{Discussion}
When specialized to $r = 2$, sparse tensor PCA is known as the \textit{planted submatrix model}, and estimation has been studied from both statistical~\cite{BI-info, kolar-info, BIS-info} and computational~\cite{submatrix-ogp, SW-22, arxiv-version} perspectives. In particular, Theorem~\ref{thm:sparse:PCA} specialized to $r=2$ recovers~\cite[Theorem 2.2]{arxiv-version}.

For $r \geq 3$, the statistical limits of estimation are well-understood~\cite{LZ-tensor, niles2020all, perry-2020}. In particular, Niles-Weed and Zadik~\cite{niles2020all} establish (with a slightly different prior where exactly $k=n\rho$ entries in $\theta$ are non-zero) an \textit{all-or-nothing phenomenon}: for $\rho=o(1)$, there exists a threshold $ \lambda_{\mathrm{stat}} = \sqrt{\frac{2\log(1/\rho)}{n^{r-1}\rho^{r-1}}}$ such that if $\lambda \leq (1-\eps)\lambda_{\mathrm{stat}}$ then the estimation of $\theta$ is information-theoretically impossible, whereas if $\lambda \geq (1+\eps)\lambda_{\mathrm{stat}}$, then the Bayes optimal estimator achieves asymptotically perfect estimation.

By contrast, much less is known about the limits of computationally efficient estimation for $r\geq 3$. To simplify the exposition, we specialize to the case where $\rho = n^{\xi - 1}$ for a fixed $\xi \in [1/2, 1)$. Luo and Zhang~\cite[Theorem~8]{LZ-tensor} showed that an efficient aggregated-SVD procedure succeeds for $\lambda > \lambda_{\mathrm{alg}} = \Theta(n^{(1/2-\xi)(r - 1)})$, and~\cite[Theorem~16]{LZ-tensor} gives a matching exponent-level lower bound conditional on the planted dense subhypergraph recovery conjecture. Together, these results pin down the threshold for computationally efficient estimation at the exponent level, but do not determine its sharp value at the level of constants. Nonetheless, these results reveal a substantial statistical--computational gap: $\lambda_{\mathrm{alg}}/\lambda_{\mathrm{stat}} = \widetilde{\Theta}(\rho^{-(r-1)/2})$.

A variety of algorithmic classes have been studied in connection with sparse tensor PCA. Low-temperature local reversible MCMC algorithms have been shown to fail to reach $\lambda_{\mathrm{alg}}$~\cite{chen-2024}, only succeeding above $ \lambda_{\mathrm{MCMC}} = \widetilde{\Theta}(n^{(1-3\xi/2)(r-1)})$. A broader class of local search algorithms has subsequently been shown to reach $\lambda_{\mathrm{alg}}$ up to polylogarithmic factors~\cite{lovig-2025}. A recent related work by Tsirkas, Wang, and Zadik~\cite{tsirkas-2026} considers a variant of sparse tensor PCA with a symmetric Rademacher prior on $\theta$, establishing low-degree hardness for the full-tensor estimand via a different approach based on the Franz--Parisi potential from statistical physics~\cite{FranzParisi95, FranzParisi97}. Since their result is more closely
related to tensor PCA with a general prior, we defer a detailed discussion of their
techniques to the paragraph following Theorem~\ref{thm:general:PCA}.

Prior to our result, the sharp threshold was conjectured in the approximate message passing (AMP) literature. Based on a heuristic argument analyzing the fixed points of the state evolution of the AMP algorithm, Corinzia, Penna, Szpankowski, and Buhmann~\cite[Claim 1]{corinzia22statistical} predicted that the estimation threshold is given by $\frac{en^{r - 1} \rho^{2r - 2} \lambda^{2}}{(r - 2)!}=1$. Their analysis is stated for a fixed-$k$ prior, with $k$ corresponding to our $n\rho$: in their normalized signal-to-noise ratio $\SNR := \lambda\sqrt{\binom{k}{r}/(2k\log n)}$, their conjectured threshold reads $\SNR_{\mathrm{AMP}} = \sqrt{\frac{1}{2e}\lPa\frac{n}{k}\rPa^{r-1}\frac{1}{r(r-1)\log n}}$, which converts back to the condition above up to a $1+o(1)$ factor.

In the regime $n^{-1/2} \ll \rho \ll 1$, Theorem~\ref{thm:sparse:PCA} establishes a computational all-or-nothing phenomenon~\cite{macris2020all}: if $\frac{en^{r-1}\rho^{2r-2}\lambda^2}{(r-2)!} \leq 1-\eps$, no polynomial of degree-$n^{C}$ (for some $C=C(r,\eps)>0$) achieves correlation $\omega(\sqrt{\rho})$ with $\theta_1$; once $\frac{en^{r-1}\rho^{2r-2}\lambda^2}{(r-2)!} \geq 1+\eps$, a polynomial of degree-$O_{r,\eps}(\log n)$ achieves near-perfect recovery of $\theta_1$. By contrast, in the highly sparse regime $\rho \ll n^{-1/2}$, we do not expect the same phenomenon: exact recovery is already achieved in polynomial time at $\lambda \gg 1$ (where $\gg$ hides a $\polylog(n)$ factor) by thresholding the $n\rho$ largest values among the diagonal entries $(Y_{i,i,\ldots,i})_{i \in [n]}$ (see~\cite[Remark~6]{choo2021complexity}), which is a tensor analog of diagonal thresholding for planted submatrix $r=2$~\cite{arash-2008}. More generally, Choo and d'Orsi~\cite[Theorem~21]{choo2021complexity} construct a family of estimators parameterized by an integer $t \geq 1$, achieving almost-exact recovery with runtime $O_r(n^{r+t})$ whenever $\lambda \geq \widetilde{\Omega}(t^{-(r-1)/2})$. Equivalently, this requires $t \geq \widetilde{\Omega}(\lambda^{-2/(r-1)})$, giving runtime $\exp\{\widetilde{O}(\lambda^{-2/(r-1)})\}$. Recalling the heuristic correspondence between degree and runtime discussed in the introduction, our condition $D^{r-1}\lambda^2 \lesssim 1$---equivalently, $D \lesssim \lambda^{-2/(r-1)}$---coincides with this runtime. Therefore, our lower bound suggests that below the sharp threshold, runtime $\exp\{\widetilde{\Omega}_{r}(\lambda^{-2/(r-1)})\}$ is required. This shows optimality of the algorithms of~\cite{choo2021complexity} by pinning down the precise degree-$\Omega_r(\lambda^{-2/(r-1)})$.

\subsection{Tensor PCA with general prior}
We now consider tensor PCA with a general prior $\pi$ on the entries of $\theta$. As in Definition~\ref{def:symmetric:sparse:PCA} for sparse tensor PCA, we work with the symmetric variant: the observation is $Y = \lambda\,\theta^{\otimes r} + \Wsy$ with $\theta \in \R^n$ having i.i.d.\ entries $\theta_i \sim \pi$ and $\Wsy$ the symmetrized standard Gaussian tensor. The asymmetric model defined in Section~\ref{sec:intro} is equivalent to this symmetric model from an estimation perspective. Indeed, the degree-$D$ MMSE under the asymmetric model matches that under the symmetric model after the substitution $\lambda \mapsto \sqrt{r!}\lambda$ (see Lemma~\ref{lem:equivalence:of:noise:models} in the appendix). The goal is to estimate $x = \prod_{i=1}^{m} \theta_i$ for a fixed $m \geq 2$.

\begin{theorem}\label{thm:general:PCA}
Consider the general tensor PCA model with estimand $x = \prod_{i = 1}^{m} \theta_{i}$. Assume that the prior satisfies $\E[\pi]=0, \E[\pi^2]=1$, and the moment condition $\E[|\pi|^{t}] \leq (Kt)^{\nu t}$ for all $t \geq 1$ for some constants $K, \nu> 0$. There is a constant $C = C(r, m,K,\nu)$ such that if
\begin{equation}
	\lambda n^{r/4} D^{(r-2)/4} \leq 1/C \qquad \textnormal{and} \qquad \lambda D^{(6\nu+3)r/2} \leq 1/C\,,
\label{eq:general:PCA:condition}
\end{equation}
then 
\[
\Corr_{\leq D}\equiv \sup_{f \in \R_D[Y]} \frac{\E[f(Y) \cdot \prod_{i=1}^m \theta_i]}{\sqrt{\E[f(Y)^2]}}  \leq Cn^{-m/4}\,.
\]
\end{theorem}

\paragraph{Discussion}
For $r=2$, tensor PCA with general prior is called the \textit{spiked Wigner model} and is known to exhibit a sharp threshold at the BBP phase transition $\lambda=1$~\cite{BBP} for low-degree estimation~\cite{arxiv-version}. The spiked Wigner model has been extensively studied from both statistical~\cite{proof-replica, LM-wigner,fund-limits-wigner} and computational~\cite{BM-amp,FR-amp,MV-amp} perspectives; see~\cite{miolane-survey} for a survey.

For $r \geq 3$, the statistical thresholds for weak recovery and detection have been established for various priors~\cite{perry-2020, lesieur2017statistical, chen2019phase, jagannath-2020}. In particular, for odd $r$ and i.i.d. prior $\pi$ with bounded support, Lesieur, Miolane, Lelarge, Krzakala, and Zdeborov\'a~\cite{lesieur2017statistical} establish a sharp threshold for weak recovery at $\lambda_{\mathrm{stat}} = \Theta_r(n^{-(r-1)/2})$.

On the other hand, the best known algorithms require $\lambda \gtrsim n^{-r/4}$: tensor unfolding succeeds for $\lambda > (1+o(1)) n^{-r/4}$~\cite{RM-tensor-pca,feldman-2023}; degree-$4$ sum-of-squares for $\lambda \gg n^{-3/4}$ when $r=3$~\cite{tensor-pca-sos}; spectral methods based on Kikuchi Hessian~\cite{kikuchi} and Langevin dynamics~\cite{BGJ-tensor} at $\lambda \gg n^{-r/4}$. These match the leading order of our lower bound $\lambda \lesssim n^{-r/4} D^{-(r-2)/4}$ at constant $D$. Recently, for the Rademacher prior and for all sufficiently large $m$, Li~\cite[Theorem 2.11]{li-2025} constructed a family of low-degree polynomial estimators with degree $D=m\log n$ achieving weak recovery at $\lambda \gtrsim n^{-r/4}m^{-(r-2)/4}$. In particular, the threshold at which their estimators succeed exhibits the same dependence on $n$ and $D$ as our lower bound, up to polylogarithmic factors.

In the low-degree polynomial framework, the optimal degree--signal tradeoff $\lambda \asymp n^{-r/4} D^{-(r-2)/4}$ was established for detection (with Rademacher prior) by Kunisky, Wein, and Bandeira~\cite{ld-notes}. Prior to our work, the corresponding lower bound for weak recovery had been established in the spherical prior case, where $\theta$ is drawn uniformly from $\{\theta \in \R^n : \|\theta\| = \sqrt{n}\}$, by Kunisky, Moore, and Wein~\cite{kunisky24tensor}. They proved that, for odd $r \geq 3$, no polynomial of degree $D = O_r(n^{1/2})$ achieves weak recovery whenever $\lambda = O_r(n^{-r/4}D^{-(r-2)/4})$, exploiting the orthogonal invariance of the prior. To do so, \cite{kunisky24tensor} introduced a new notion of \emph{tensor cumulants}, expecting that the framework of~\cite{SW-22} alone does not yield the precise dependence of $\lambda$ on $D$. Somewhat surprisingly, Theorem~\ref{thm:general:PCA} obtains the conjecturally correct dependence $\lambda = O_r(n^{-r/4}D^{-(r-2)/4})$ directly via the cumulant approach of Schramm and Wein~\cite{SW-22} for the broad class of i.i.d. priors with bounded moments. Our result, however, holds only over a more restricted range of $D$ than $O_r(n^{1/2})$ due to the technical condition $\lambda D^{(6\nu+3)r/2} \leq 1/C$ in Theorem~\ref{thm:general:PCA}, which is likely a proof artifact. 

In recent and concurrent work, Tsirkas, Wang, and Zadik~\cite{tsirkas-2026} develop a rigorous connection between the Franz--Parisi (FP) potential from statistical physics~\cite{FranzParisi95, FranzParisi97} and low-degree MMSE for a broad class of Gaussian additive models. Their main result establishes that, in this class, the monotonicity of the annealed FP potential is equivalent to low-degree MMSE hardness; as one application, they obtain low-degree MMSE hardness for estimating the rank-one tensor $\theta^{\otimes r}$ in tensor PCA with i.i.d. Gaussian priors~\cite[Theorem~3.1]{tsirkas-2026}.\footnote{Although their bound is on the vectorized MMSE for estimating the tensor $\theta^{\otimes r}$, by symmetry of the i.i.d. prior this is equivalent up to a $1+o(1)$ factor to the scalar MMSE with estimand $x = \prod_{i=1}^r \theta_i$, since the dominant contribution in the vectorized MMSE comes from the entries of the tensor with all distinct indices.} For the Gaussian prior $\pi = \cN(0, 1)$, the moment condition in Theorem~\ref{thm:general:PCA} is satisfied with $\nu = 1/2$. Thus, our result implies hardness for all polynomials of degree $D = O(n^{r/(11r+2)})$ at $\lambda =c n^{-r/4} D^{-(r-2)/4}$, where $c=c_r>0$ is a small enough constant, while~\cite[Theorem~3.1]{tsirkas-2026} gives hardness for $D = \widetilde{o}(n^{1/2})$ at $\lambda = \widetilde{O}_r(n^{-r/4} D^{-r/4})$. Their bound covers a larger degree range but is loose by $\textnormal{polylog}(n)$ factors and a factor $D^{1/2}$ in $\lambda$. It would be interesting to determine the largest exponent $\delta$ for which our hardness bound extends to degree $D = n^{\delta}$.

\subsection{Efficient algorithm for estimation}\label{subsec:res:efficient:algorithm}
Although Theorem~\ref{thm:planted:hypergraph}-(b) and
Theorem~\ref{thm:sparse:PCA}-(b) establish the existence of
degree-$O_{r,\eps}(\log n)$ polynomials achieving near-perfect correlation
with the estimand $x=\theta_1$ above the sharp thresholds, a naive
evaluation of these polynomials would take $n^{O_{r,\eps}(\log n)}$ time. The polynomials in
our proofs are, however, tree-shaped: each is indexed by a hypertree
(see Definition~\ref{def:hypertree}) on $k = O_{r,\eps}(\log n)$ vertices.
For such tree-shaped polynomials, the color-coding trick of Alon, Yuster,
and Zwick~\cite{color-coding} gives a randomized polynomial-time
implementation. Namely, we randomly color the $n$ vertices using $k$ colors and restrict
the sum to colorful hypertrees, i.e., hypertrees whose $k$ vertices receive
distinct colors. For a fixed hypertree, this occurs with probability $q=k!/k^k=n^{-O(1)}$. After rescaling by $1/q$, the colorful count is an unbiased estimator
of the original tree-shaped polynomial. Moreover, the total colorful weighted count
can be computed in polynomial time by dynamic programming over the tree
structure. Averaging over $n^{\Omega(1)}$ independent colorings reduces the additional variance from the coloring
randomness, yielding a randomized polynomial-time estimator with the same
asymptotic correlation. These techniques are by now standard; see, e.g.,
\cite{color-coding,HS-bayesian,mao2024testing}, and similar color-coded
approximations were used by Li~\cite{li-2025} to obtain polynomial-time
algorithms for tensor PCA with the Rademacher prior.

Through this color coding trick, we prove the following theorem in Appendix~\ref{sec:polytime:algorithm} by turning our low-degree upper bound into an efficient algorithm.

\begin{theorem}\label{thm:algorithm}
Consider the planted dense subhypergraph model and assume that the assumptions of Theorem~\ref{thm:planted:hypergraph}-(b) hold for $\eps>0$. Let $\ell=\lceil \frac{4}{\eps} \log(1/\rho) \rceil$. Given the observation $Y$ and the parameters $\rho,q_0,q_1$, the (randomized) Algorithm~\ref{alg:almost:exact:recovery} has runtime $n^{r+o(1)}e^{O_r(\ell)}$ and outputs a set of vertices $\widehat{S}$ satisfying $|\widehat{S} \triangle S|=o(n\rho)$ with probability $1 - o(1)$ as $n\rightarrow\infty$. 

Similarly, for the symmetric sparse tensor PCA model, if the assumptions of Theorem~\ref{thm:sparse:PCA}-(b) hold for $\eps>0$, then given the observation $Y$ and parameters $\lambda, \rho, \eps$ there is a randomized algorithm with runtime $n^{r+o(1)}e^{O_r(\ell)}$ that outputs $\widehat{S}$ with $|\widehat{S} \triangle S| =o(n \rho)$ with probability $1 - o(1)$ as $n\rightarrow\infty$.
\end{theorem}

To describe the algorithm, we first introduce the notion of a \textit{hypertree}.

\begin{definition}\label{def:hypertree}
An $r$-uniform hypergraph $H$ on $[k]$ is connected if, for every pair of distinct vertices $u,v\in V(H)$, there exists a sequence of hyperedges $e_1,\ldots, e_m\in E(H)$ such that $u\in e_1$, $v\in e_m$, and $e_i\cap e_{i+1}\neq \emptyset$ for all $1\leq i\leq m-1$. An $r$-uniform hypergraph $T$ on $[k]$ is called a hypertree if there is a sequence of $r$-uniform hypergraphs $H_1,\hdots,H_{|E(T)|}$ such that $H_1=\{e_1\}$ is a single hyperedge, $H_{|E(T)|} = T$, and for each $2 \leq i \leq |E(T)|$, $H_i$ is successively obtained from $H_{i-1}$ by appending a hyperedge $e_i$ with $|V(e_i) \cap V(H_{i-1})|=1$. Equivalently, a connected $r$-uniform hypergraph $T$ with $k$ vertices is a hypertree if and only if $|E(T)|=(k-1)/(r-1)$, the minimum number of edges among all connected $r$-uniform hypergraphs on $[k]$.
See, e.g.~\cite[Definition 1]{BHP-19}.
\end{definition}

The polynomial used to establish the upper bounds in Theorem~\ref{thm:planted:hypergraph}-(b) and Theorem~\ref{thm:sparse:PCA}-(b) is an average of tree-shaped polynomials indexed by the following class of hypertrees.
\begin{definition}\label{def:special:trees}
For $\ell\geq 1$, let $\sT_{\ell}$ be the set of all rooted $r$-uniform hypertrees $\alpha$ with root vertex~$1$ satisfying the following: the degree of the root is two, and in each of the two root-incident edges, exactly one of the $r-1$ non-root vertices has degree $\geq 2$ in $\alpha$. Furthermore, deleting the root vertex~$1$ together with the $2(r-2)$ vertices contained in the two root-incident edges decomposes $\alpha$ into  two disjoint hypertrees each having $\ell$ edges. See Figure~\ref{fig:hypertree} for an example.

Let $\sH_\ell$ be a
choice of one representative from each root-preserving isomorphism class in
$\sT_\ell$. For $H\in \sH_\ell$, we write $r_H$ for its root and also write $\Aut(H)$ for the group of root-preserving automorphisms of
$H$; that is, bijections $\sigma:V(H)\to V(H)$ such that $\sigma(r_H)=r_H$ and $\sigma$ preserves the edge set of $H$.
\end{definition}
\begin{figure}
    \centering
    \includegraphics[width=0.55\linewidth, frame]{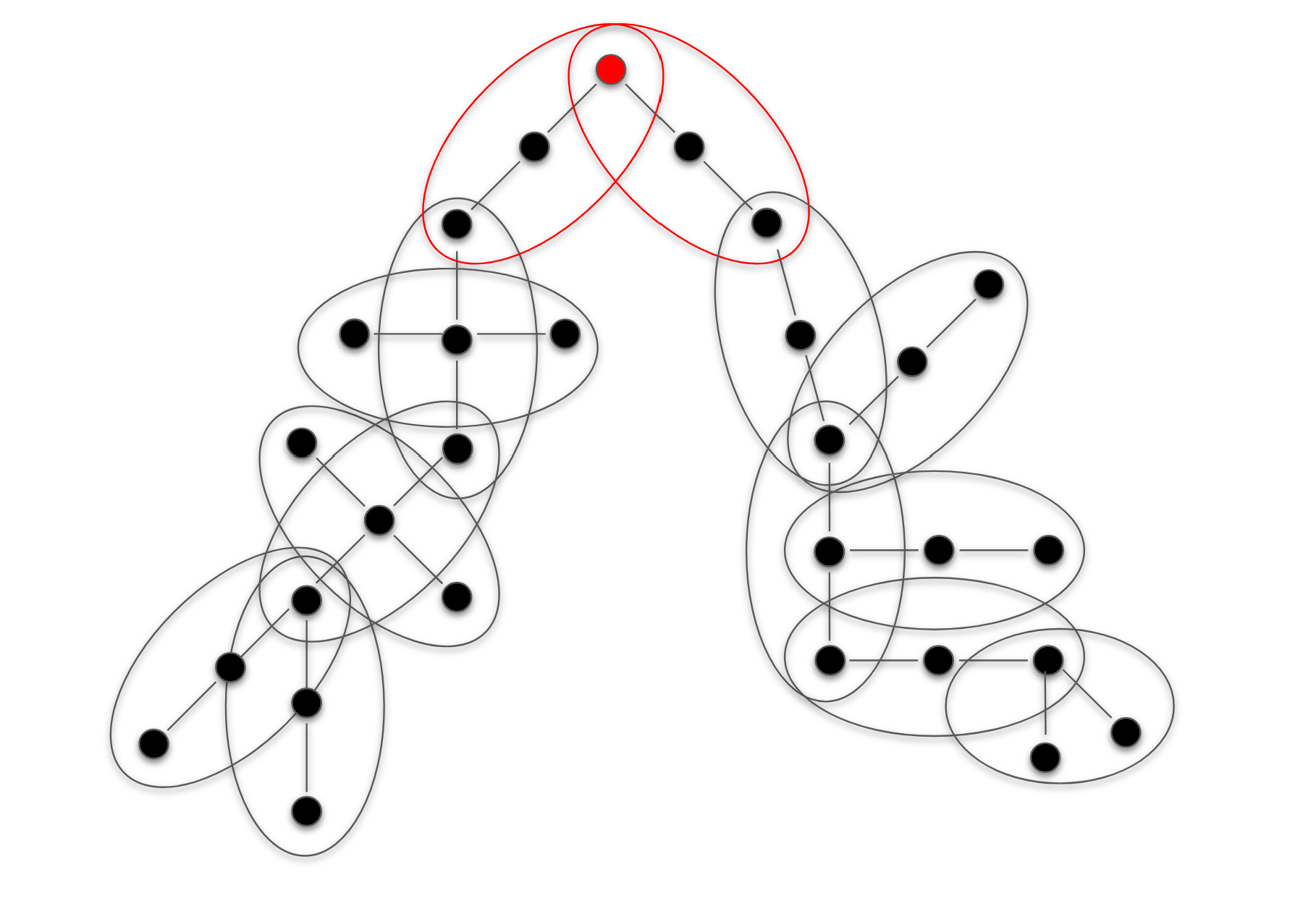}
    \caption{Schematic of a hypertree in $\sT_6$ for $r=3$ and $\ell=6$. Vertices are represented by filled dots and edges by enclosed circles. The root vertex and its root-incident edges are highlighted in red.}
    \label{fig:hypertree}
    \vspace{-10pt}
\end{figure}
Equivalently, $\alpha \in \sT_\ell$ is a rooted hypertree (with root vertex $1$) consisting of two disjoint rooted hypertrees of $\ell$ edges each, each attached to vertex $1$ by a root-incident edge.

\begin{algorithm}[H]
\caption{Almost-exact recovery via color-coding for planted dense subhypergraph.}
\label{alg:almost:exact:recovery}

\begin{algorithmic}[1]
\State \textbf{Input:} Adjacency tensor $Y$, parameters $\rho,q_0,q_1,\eps$, and degree parameter $\ell$.
\State \textbf{Preprocessing:} Set $\lambda = \frac{q_1-q_0}{\sqrt{q_0(1-q_0)}}$, $\lambda_\star = \sqrt{\frac{(r-2)!(1+\eps)}{en^{r-1}\rho^{2r-2}}}$, and $p=\frac{\lambda_\star}{\lambda}$. Independently, for every hyperedge $e = \{i_1,\hdots,i_r\}$, keep $Y_e$ with probability $p$ and, with probability $1-p$, resample $Y_e$ from $\Ber(q_0)$. Denote the resulting hypergraph by $Y_\star$ and define $\widetilde{Y} = \frac{Y_\star-q_0}{\sqrt{q_0(1-q_0)}}$.
\State Set $k = (r-1)(2\ell+2) + 1$ and $t = \lceil \frac{k^{k}}{k!}\log{n} \rceil$.
\For{$s=1,\hdots,t$}
    \State Sample a fresh independent uniformly random coloring $c_{s}: [n] \to [k]$.
    \For{$H\in\sH_{\ell}$}
        \State Compute $A_H(s)=(A_{r_H}(i,[k],c_{s}))_{i \in [n]}$ using Algorithm~\ref{alg:color:coded:score} below.
    \EndFor
\EndFor
\State Compute $z=\sum_{s=1}^{t} \sum_{H \in \sH_\ell} A_H(s)/|\Aut(H)|$.
\State Let $\widehat{S} \subset [n]$ be the set of $\lfloor n\rho \rfloor$ indices with the largest entries in $z=(z_1,\ldots,z_n)$, breaking ties by largest index.
\State \textbf{Output:} $\widehat{S}$.
\end{algorithmic}
\end{algorithm}
We remark that the \textbf{Preprocessing} step in Algorithm~\ref{alg:almost:exact:recovery} is only needed for our proof: it replaces each edge of $Y$ independently with probability $1-p$ by an independent $\Ber(q_0)$ draw, reducing the effective signal-to-noise ratio to $\lambda_\star = p\lambda$. As larger $\lambda$ only makes estimation easier, this is most likely a proof artifact, and is not needed. This step is included because the proof of Theorem~\ref{thm:planted:hypergraph}-(b) assumes $\lambda = \lambda_\star$, so Algorithm~\ref{alg:almost:exact:recovery} is modified to reflect this change. Since preprocessing all edges naively takes $\binom{n}{r} \leq n^r$ time, this step does not affect the final algorithmic runtime guarantee.

The score computation in Algorithm~\ref{alg:color:coded:score} below performs a dynamic program on each $H \in \sH_\ell$. For any rooted hypertree $H$ with root vertex $r_H$, the associated \emph{bipartite incidence tree} $I$ has node set $V(H) \sqcup E(H)$, with each vertex-node $v$ adjacent to every edge-node $e$ containing $v$, and is rooted at $r_H$; the hypertree property of $H$ ensures that $I$ is itself a rooted tree. For each vertex-node $v \in V(I)$, let $\ch(v)$ denote its child edge-nodes, and for each edge-node $e \in E(I)$, let $\ch(e) = (v_1, \ldots, v_{r-1})$ denote its child vertex-nodes in an arbitrary fixed ordering. For each node $a$ of $I$, let $H_a$ denote the rooted sub-incidence tree below $a$, and write $\vtx(H_a)$ for the set of vertex-nodes of $H_a$; when $a$ is an edge-node, this excludes the parent vertex-node of $a$.
\begin{algorithm}[H]
\caption{Computation of color-coded score.}
\label{alg:color:coded:score}

\begin{algorithmic}[1]
\State \textbf{Input:} Tensor $\widetilde{Y} \in (\R^n)^{\otimes r}$, coloring $c: [n] \to [k]$, and rooted hypertree representative $H \in \sH_{\ell}$ with root $r_{H}$.
\State View $H$ as a rooted bipartite incidence tree $I$ rooted at $r_{H}$. 

\State For each vertex-node $v$, every $i \in [n]$ and $Q \subseteq [k]$ with $|Q| = |\vtx(H_v)|$, compute recursively
\Statex
\[
    A_v(i,Q,c) = \begin{cases}
        \In_{Q = \{c(i)\}} & \text{if $v$ is a leaf},\\[0.8em]
        \In_{c(i) \in Q} \displaystyle\sum_{\substack{(Q_e)_{e \in \ch(v)} \\ |Q_e| = |\vtx(H_e)| \\ \bigsqcup_{e \in \ch(v)}Q_{e} = Q \setminus \{c(i)\}}} \prod_{e \in \ch(v)} B_{e}(i, Q_e, c) & \text{otherwise}.
    \end{cases}
\]
\State For each edge-node $e$, every $i \in [n]$ and $Q \subseteq [k]$ with $|Q| = |\vtx(H_e)|$, compute recursively
\Statex
\[
    B_e(i, Q, c) = \In_{c(i) \notin Q} \sum_{\substack{y_1,\hdots,y_{r-1} \in [n] \setminus \{i\}\\ \text{distinct}}} \widetilde{Y}_{\{i,y_1,\hdots,y_{r-1}\}} \sum_{\substack{(Q_j)_{j=1}^{r-1} \\ |Q_j| = |\vtx(H_{v_j})| \\ \bigsqcup_{j=1}^{r-1} Q_j = Q}} \prod_{j=1}^{r-1} A_{v_j}(y_j, Q_j, c)
\]
\State \textbf{Output:} $(A_{r_H}(i,[k],c))_{i \in [n]}$.
\end{algorithmic}
\end{algorithm}

\section{Proof techniques}
We now outline the techniques used to derive our low-degree lower bounds. For tensor PCA with a general prior, we use the cumulant expansion approach of~\cite{SW-22}. For planted dense subhypergraph and sparse tensor PCA with $\rho\geq n^{-1/2}$, we use the \textit{orthogonal expansion} approach of~\cite{arxiv-version}. For planted dense subhypergraph with $\rho\leq n^{-1/2}$, we introduce a conditional version of this method, namely \textit{orthogonal expansion with conditioning}. To explain it, we first revisit the approach of~\cite{arxiv-version}, and then describe the modifications needed to incorporate conditioning.

\subsection{Orthogonal expansion approach by~\cite{arxiv-version}}
\label{subsec:orthogonal:expansion}
 
Recall that our lower bounds aim to establish $\Corr_{\leq D}=o(1)$ for some $D=n^{\Omega(1)}$. The orthogonal expansion approach by~\cite{arxiv-version} provides a systematic way to upper bound $\Corr_{\leq D}$ via the following three steps.

\medskip
\noindent\textbf{Step 1: Choose a basis for $\R_D[Y]$.}
Select a collection $(\phi_\alpha(Y))_{\alpha\in \GGG}$ that spans $\R_D[Y]$, where $\GGG$ is a suitable index set---in our applications, this will be (multi)-hypergraphs with at most $D$ edges. This basis need not be orthogonal; it is used to expand the numerator $\E[f(Y)\cdot x]$. Writing $f(Y)=\sum_{\alpha\in \sG}\hat{f}_\alpha\,\phi_\alpha(Y)$, we have
\[
\E[f(Y)\cdot x]=\sum_{\alpha\in \GGG}c_{\alpha}\hat{f}_{\al}=\langle c\,, \hat{f}\rangle ,\quad\textnormal{where}\quad c=(c_{\al})_{\alpha\in \GGG}\equiv (\E[\phi_{\al}x])_{\al\in \GGG}\,.
\]
For instance, in the planted dense subhypergraph model, we identify the set of all $r$-uniform hyperedges on $n$ vertices with $[N]$, where $N:=\binom{n}{r}$, and take $\phi_\alpha(Y)=\prod_{e=1}^N(Y_e-q_0)^{\alpha_e}$ for $\al=(\al_e)_{e\in [N]}\in \{0,1\}^N$. Note that  this is the centered monomial basis under the ``null distribution'' where $q_1=q_0$.

\medskip
\noindent\textbf{Step 2: Introduce an orthonormal family to control $\E[f(Y)^2]$.}
The key difficulty in bounding $\Corr_{\leq D}$ is the denominator $\E[f(Y)^2]$. Since the entries of $Y$ are not independent under the planted distribution, there is no obvious orthogonal basis in the space of low-degree polynomials of $Y$ with respect to which one could expand $\E[f(Y)^2]$ using Parseval's identity.

The idea of~\cite{arxiv-version} is to instead work in a larger space: introduce an orthonormal family $(\psi_{\beta\gamma}(W,\theta))_{(\beta,\gamma)\in \PPP}$ that depends on both the ``noise'' variable $W$ and the signal $\theta$, where $W$ denotes the source of randomness in the observation (e.g., the independent edge variables in the hypergraph model, or the Gaussian noise tensor in tensor PCA). In our applications, since $W$ and $\theta$ are independent with independent entries, and it is straightforward to construct such $\psi_{\beta\gamma}$ by taking the separable form $\psi_{\beta\gamma}=\psi_{1,\beta}(W)\cdot \psi_{2,\gamma}(\theta)$. 

Assuming $Y$ is measurable with respect to $(W,\theta)$, the orthonormality condition $\E[\psi_{\beta\gamma}\,\psi_{\beta'\gamma'}]=\In\{(\beta,\gamma)=(\beta',\gamma')\}$ allows one to lower-bound $\E[f(Y)^2]$ via Bessel's inequality
\begin{equation}\label{eq:Bessel}
\begin{split}
    &\E[f(Y)^{2}]
    \geq \sum_{(\beta,\gamma) \in \PPP} \Big(\E\big[f(Y)\psi_{\beta\gamma}(W,\theta)\big]\Big)^{2} 
               = \big\|M \hat{f}\big\|^{2}\,,\quad\textnormal{where} \\
            &M=(M_{\be\ga,\al})_{(\be,\ga)\in \PPP, \al\in \GGG}\equiv (\E[\phi_{\al}(Y)\psi_{\be\ga}(W,\theta)])_{(\be,\ga)\in \PPP, \al\in \GGG}\,.
\end{split}
\end{equation}
Combining~\eqref{eq:Bessel} with the Cauchy Schwarz inequality yields the following dual bound.
\begin{proposition}[Proposition 1.3 of~\cite{arxiv-version}]\label{prop:duality}
Suppose that $Y$ is $(W,\theta)$-measurable. Fix a basis $(\phi_{\al}(Y))_{\al\in \GGG}$ in $\R_D[Y]$ and an orthonormal family $(\psi_{\be\ga}(W,\theta))_{(\be,\ga)\in \PPP}$. Letting $c$ and $M$ be defined as above accordingly. Then,
\begin{equation}\label{eq:corr:bound}
\Corr_{\leq D}\leq \frac{1}{\sqrt{\E[x^2]}}\;\inf_{u:M^{\top}u=c}\|u\|\,, 
\end{equation}
where we take the convention that the infimum is $\infty$ if none of such $u$ exists, and the infimum is taken w.r.t. $u=(u_{\be\ga})_{(\be,\ga)\in \PPP}$ such that $M^{\top}u=c$. That is,
\begin{equation}\label{eq:u:condition}
\sum_{(\be,\ga)\in \PPP}M_{\be\ga,\al}u_{\be\ga}=c_{\al}\,.
\end{equation}
Moreover, if $\R_D[Y]\subseteq \textnormal{span}(\psi_{\be\ga})$, then the above inequality is an equality.
\end{proposition}
\begin{proof}
We refer to~\cite[Proposition 1.3]{arxiv-version} for the proof of \eqref{eq:corr:bound}, which follows from~\eqref{eq:Bessel} and a Cauchy Schwarz inequality. To prove the final claim, assume that $\R_D[Y]\subseteq \textnormal{span}(\psi_{\be\ga})$. Let $\mathsf{P}_D$ denote the projection operator onto $\R_{D}[Y]$, and set $f_\star=\mathsf{P}_D x$.
Since $f_\star\in \R_D[Y]\subseteq \textnormal{span}(\psi_{\be\ga})$, we have $f_\star=\sum_{\be\ga}u^\star_{\be\ga}\psi_{\be\ga}$ for some $u^\star=(u^\star_{\be\ga})_{\be\ga\in \PPP}$. Since $x-f_\star $ must be orthogonal to $\phi_{\al}(Y)$ for any $\al\in \GGG$, we must have $M^\top u^\star=c$. Thus,
\[
   \inf_{M^\top u=c}\|u\|\leq \|u^\star\|=\sqrt{\E[f_\star^2]}=\sqrt{\E[x^2]}\cdot \Corr_{\leq D}\,,
\]
where the last equality holds since $x-f_\star$ is orthogonal to $\R_D[Y]$, so the supremum defining $\Corr_{\leq D}$ in~\eqref{eq:def:Corr} is attained at $f=f_\star$. Combining with~\eqref{eq:corr:bound} concludes the proof.
\end{proof}
\begin{remark}\label{rem:cumulant}
In the Gaussian additive model $Y=X+W$, where $W$ is a matrix with $\mathcal{N}(0,1)$ i.i.d. entries and $X$ is independent of $W$, the cumulant bound~\cite[Theorem 2.2]{SW-22} may be viewed as a special case of Proposition~\ref{prop:duality}. Indeed, let $H_{\al}$ be the normalized multivariate Hermite polynomial for $\al\in \GGG$, where $\GGG$ is the set of all multigraphs $\al$ with $|\al| \leq D$. Set $\sP = \GGG$. Taking $\phi_{\al}(Y)=H_{\al}(Y)$, $\theta=X$, and $\psi_{\beta\ga}(W,X)=H_{\be}(W)$, it is not difficult to see that there exists a unique $u$ such that $M^{\top}u=c$, and the bound in Proposition~\ref{prop:duality} reduces to the cumulant bound~\cite[Theorem 2.2]{SW-22}.
\end{remark}
The power of Proposition~\ref{prop:duality} is that \emph{any} $u$ satisfying $M^\top u =c$ may be viewed as a dual certificate for upper bounding $\Corr_{\leq D}$. By having a rich enough family $(\psi_{\be\ga})$, we have freedom to construct $u$. Such an approach is particularly effective for proving very accurate bounds on $\Corr_{\leq D}$ leading to a sharp phase transition for many models, as seen in Theorem~\ref{thm:planted:hypergraph} and~\ref{thm:sparse:PCA}.

\medskip
\noindent\textbf{Step 3:  Reduce to ``good'' graphs and construct $u$.} The final step is to find $u$ satisfying $M^\top u=c$ with small enough
norm by eliminating ``uninformative'' terms. To illustrate, consider the planted dense subhypergraph model with $(\phi_{\al})_{\al\in \GGG}$ and $(\psi_{\be\ga})_{(\be,\ga)\in \PPP}$ constructed above. For any $u$ such that $M^\top u =c$ with small enough norm, we expect that the function $\sum_{\be\ga}u_{\be\ga}\psi_{\be\ga}$ must be close to $f_\star$, the projection of $x=\theta_1$ onto $\R_D[Y]$. Thus, intuitively, for $\be$ which doesn't contain vertex $1$ or is not connected, the value $u_{\be\ga}$ must be small. The following lemma guarantees that such terms can actually be set to $0$ by exploiting linearity in the system $M^\top u =c$.

\begin{lemma}[Lemma 1.4 of~\cite{arxiv-version}]\label{lem:reduce}
Suppose that there exist subsets $\GGG_\star \subset \GGG$ and $\PPP_\star \subset \PPP$ such that for each $\al\in \GGG$, there exists $\alpha_\star \in \PPP_\star$ and $\mu \in \RR$ such that 
\begin{equation}\label{eq:reduce}
c_\alpha = \mu c_{\alpha_\star}\quad\textnormal{and}\quad M_{\beta\gamma,\alpha} = \mu M_{\beta\gamma,\alpha_\star}\quad\textnormal{for all}\quad (\beta,\gamma)\in \PPP_\star\,.
\end{equation}
Then, for any $(u_{\be\ga})_{\be\ga\in \PPP}$ such that $u_{\be\ga}=0$ for all $(\be,\ga)\notin \PPP_\star$, if~\eqref{eq:u:condition} holds for all $\alpha \in \GGG_\star$, then~\eqref{eq:u:condition} holds for all $\alpha \in \GGG$ as well.
\end{lemma}
For planted dense subhypergraphs, we will set $\GGG_\star$ to be the all connected hypergraphs $\al$ containing the vertex $1$ (cf. Definition~\ref{def:good:hypergraph}), and $\PPP_\star$ to be all collection $(\be,\ga)$ such that $\be\in \GGG_\star$ and $\ga$ such that its support is included in vertices of $\be$. Lemma~\ref{lem:good:reduction:hypergraph} verifies the condition \eqref{eq:reduce} holds.

Lemma~\ref{lem:reduce} is a crucial step in the orthogonal expansion approach by~\cite{arxiv-version} since in both the construction of $u$ such that $M^\top u =c$ and the bounding step of $\|u\|$ it allows us to consider only a much smaller set $\PPP_\star$. For example, in the planted dense subhypergraphs above, the size of $|\PPP_\star|$ is at least $n^{-1}$ factor smaller than $|\PPP|$ due to the condition that vertex $1$ must be included.

Finally, we remark that in applying this orthogonal expansion approach to the planted dense subhypergraph with $\rho\geq n^{-1/2}$ and sparse tensor PCA, the $r\geq 3$ case requires new combinatorial estimates when bounding $\|u\|$. For $r=2$, a connected graph always contains a spanning tree, which was one of the main facts in the $r=2$ case~\cite{arxiv-version}. However, the analog fails for $r\geq 3$: a connected hypergraph need not contain a hypertree (cf. Definition~\ref{def:hypertree}). A spanning hypertree on $k$ vertices exists only if $k\equiv 1\pmod{r-1}$, and even this is not
sufficient---whether one exists depends on how the edges overlap, not merely on
the vertex and edge counts. To pin down the sharp threshold in
Theorem~\ref{thm:planted:hypergraph}, we therefore derive a sufficiently tight upper bound on the number of connected hypergraphs with a given number of
vertices and hyperedges, which appears to be new.

\subsection{Orthogonal expansion approach with conditioning}
\label{subsec:proof:conditioning}

For the planted dense subhypergraph model with $\rho \leq n^{-1/2}$ (i.e. $\xi\leq 1/2$), the orthogonal expansion approach in Section~\ref{subsec:orthogonal:expansion} is insufficient to establish hardness throughout $a>b\xi$. The difficulty is that when $a<b/2$, there exists rare configurations of $Y$ that make $c_{\alpha} \equiv \E[\phi_{\al}(Y)\theta_1]$ substantially greater than its typical value whenever the hypergraph $\alpha$ is `dense'; that is, has many
hyperedges relative to its number of vertices. This in turn inflates the norm of the dual certificate $u$. A similar phenomenon was observed by Dhawan, Mao, and Wein~\cite{DMW-23} who considered detection of the planted dense subhypergraph; for this reason, they considered a conditioning approach to the low-degree likelihood ratio.

Similar to the approach of~\cite{DMW-23}, we introduce a method for incorporating conditioning into the orthogonal expansion. Let $\cE_{\star}$ be a ‘good’ event, chosen so that $\cE_{\star}$ occurs with high probability, but rules out the rare configurations which cause inflating $c_{\alpha}=\E[\phi_{\al}(Y)\theta_1]$ than the typical value of $\phi_{\al}(Y)\theta_1$. For a polynomial estimator $f \in \R_D[Y]$, we split the numerator in the low-degree correlation as
\[
    |\E[f(Y) \cdot x]| \leq |\E[f(Y) \cdot x \In_{\cE_{\star}}]| + |\E[f(Y) \cdot x \In_{\cE_{\star}^{c}}]|\,.
\]
We control the two terms on the RHS separately. Roughly speaking, the first term is handled by applying the orthogonal expansion method restricted to the event $\cE_{\star}$, and the second term is controlled using Cauchy Schwarz, leveraging the fact that $\cE_{\star}$ occurs with high probability. Define
\[
\begin{split}
    &\widetilde{c} = (\widetilde{c}_{\alpha})_{\alpha \in \sG} \equiv (\E[\In_{\cE_{\star}}\phi_{\alpha}(Y)x])_{\alpha \in \sG}\,, \quad\textnormal{and} \\
    &\widetilde{M} = (\widetilde{M}_{\beta\gamma, \alpha})_{(\beta, \gamma)\in \sP, \alpha \in \sG} \equiv (\E[\In_{\cE_{\star}}\phi_{\alpha}(Y)\psi_{\beta \gamma}(W,\theta)])_{(\beta, \gamma) \in \sP, \alpha \in \sG}\,.
\end{split}
\]
\begin{proposition}\label{prop:orthogonal:expansion:conditioning}
For a given event $\cE_{\star}$, let $\widetilde{c}$ and $\widetilde{M}$ be as above. Then, 
\[
    \Corr_{\leq D} \leq \frac{1}{\sqrt{\E[x^2]}} \lPa \inf_{u:\widetilde{M}^{\top}u=\widetilde{c}} \|u\| + \sqrt{\E[x^2 \In_{\cE_{\star}^{c}}]} \rPa\,.
\]
\end{proposition}

\begin{proof}
An application of Cauchy Schwarz gives
\[
    \sup_{f \in \R_D[Y]} \frac{|\E[f(Y)\cdot x\In_{\cE_{\star}^{c}}]|}{\sqrt{\E[f(Y)^2]}} \leq \sqrt{\E[x^2\In_{\cE_{\star}^{c}}]}\,.
\]
Now let $f \in \R_D[Y]$ and expand $f(Y) = \sum_{\alpha \in \sG} \hat{f}_{\alpha} \phi_{\alpha}(Y)$ so that $\E[f(Y)\cdot x\In_{\cE_{\star}}]=\langle \widetilde{c}, \hat{f}\rangle$. By Bessel's inequality, 
\[
    \E[f(Y)^2] \geq \E[f(Y)^2\In_{\cE_{\star}}] \geq \sum_{(\beta, \gamma) \in \sP} (\E[\In_{\cE_{\star}}f(Y)\psi_{\beta\gamma}(W,\theta)])^{2} = \big\|\widetilde{M} \hat{f}\big\|^{2}\,.
\]
Hence, for any $u$ with $\widetilde{M}^{\top}u=\widetilde{c}$,
\[
    \big|\E[f(Y)\cdot x\In_{\cE_{\star}}]\big|
    = \big|\langle \widetilde{c}, \hat{f}\rangle\big|
    = \big|\langle u, \widetilde{M}\hat{f}\rangle\big|
    \leq \|u\|\big\|\widetilde{M}\hat{f}\big\|
    \leq \|u\|\sqrt{\E[f(Y)^2]}\,,
\]
using Cauchy Schwarz and the previous display. 
Dividing by $\sqrt{\E[f(Y)^2]}$
and taking the supremum over $f$, then the infimum over such $u$, bounds
$\sup_{f \in \R_D[Y]}|\E[f(Y)\cdot x\In_{\cE_{\star}}]|/\sqrt{\E[f(Y)^2]}$ by
$\inf_{u:\widetilde{M}^{\top}u=\widetilde{c}}\|u\|$. Combined with the first
display, this gives the stated bound.
\end{proof}

\begin{corollary}\label{cor:orthogonal:expansion:conditioning}
If $x\in \{0,1\}$, then
\[
    \Corr_{\leq D} \leq \frac{1}{\sqrt{\E[x^2]}} \inf_{u:\widetilde{M}^{\top}u=\widetilde{c}} \|u\| + \sqrt{\Pb(\cE_{\star}^{c} \mid x = 1)}\,.
\]
\end{corollary}

Proposition~\ref{prop:orthogonal:expansion:conditioning} reduces the problem of establishing a lower bound to two tasks. First, one must determine an event $\cE_{\star}$ for which $\E[x^{2}\In_{\cE_{\star}^{c}}]=o(\E[x^2])$. Second, one needs to construct a solution to the conditioned linear system with small enough norm.

The main technical difficulty lies in the second task. Although conditioning on $\cE_{\star}$ reduces $\widetilde{c}_{\alpha}$ to its typical scale, it introduces dependencies between disconnected hypergraphs that were otherwise absent in the unconditioned analysis. Crucially, this breaks \textbf{Step 3} in the orthogonal expansion. In particular, the reduction to good graphs (Lemma~\ref{lem:reduce}) no longer applies and we cannot simply set $u_{\beta \gamma} = 0$ for uninformative pairs $(\beta, \gamma) \not \in \sP_{\star}$. Nevertheless, we prove that the conditioned system retains enough structure to admit a tractable analysis. The key property is that the independence of the entries $(\theta_{i})_{i\leq n}$  under the prior allows us to choose $(\psi_{\beta \gamma})_{(\be,\ga)\in \mathscr{P}}$ in such a way that  we have for any fixed $\al,\be\in \GGG$ (see Eq.~\eqref{eq:small:hypergraph:identity}) 
\begin{equation}
    \sum_{\gamma: (\beta, \gamma) \in \sP} \left(-\sqrt{\frac{\rho}{1-\rho}}\right)^{|\ga|}\widetilde{M}_{\beta\gamma, \alpha} = 0 \qquad\textnormal{unless $\beta$ is a union of connected components of $\al$}\,.
\label{eq:key:identity:ansatz}
\end{equation}
This identity suggests the dual certificate of the form 
\[
    u_{\beta \gamma} = \left(-\sqrt{\frac{\rho}{1-\rho}}\right)^{|\ga|} \cK(\beta)
\]
where $\cK: \sG \to \R$ is determined by a recursion over hypergraphs.
Substituting this form into $\widetilde{M}^{\top}u=\widetilde{c}$ and applying
the identity, the system reduces to a recursion in which $\cK(\alpha)$ is
determined by $\widetilde{c}_{\alpha}$ together with the values $\cK(\beta)$ over
the proper unions of connected components $\beta$ of $\alpha$. Since each such
$\beta$ has fewer connected components than $\alpha$, the recursion is
well-defined and solves explicitly, and the resulting $\cK(\alpha)$ inherits the
scale of $\widetilde{c}_{\alpha}$, up to a combinatorial factor counting the
connected components of $\alpha$.

\medskip
\noindent\textbf{Notations} The symbols $\R$, $\N$ denote the set of all real numbers and non-negative integers, respectively. We include $0 \in \N$ by convention. We let $[n]:=\{1,\ldots, n\}$ for $n=1,2,\ldots$.

For an index set $\Lambda$, and vectors $x = (x_{i})_{i\in \Lambda}, y = (y_{i})_{i\in \Lambda}\in \R^{\Lambda}$, we write $x \leq y$ if $x_{i} \leq y_{i}$ for all $i \in \Lambda$. We will identify subsets of $[n]$ with binary vectors of length $n$ in the natural way. That is, if $A \subseteq [n]$, then we identify it with the binary vector $\gamma$ such that $\gamma_{i} = 1$ whenever $i \in A$ and $\gamma_{i} = 0$ otherwise. Therefore, the relation $\gamma \leq A$ means that $\gamma_{i} = 1$ only when $i \in A$.  We use the subscript notation $\gamma_{A}$ to denote the subvector of $\gamma$ whose indices are in $A$.

For a hyperedge $e=(i_1,\ldots, i_r)$ of a hypergraph, let $V(e)=\{i_1,\ldots, i_r\}$ denote the vertices that participate in the hyperedge $e$. We will use the multi-index notation with superscripts where, for $\theta \in \{0,1\}^n$ and $\gamma \subseteq [n]$, we write $\theta^{\gamma}=\prod_{i=1}^{n}\theta_i^{\gamma_i}$. For $Y\in \R^{N}$ and $\alpha \in \N^{N}$, we write $Y^{\alpha}:=\prod_{e=1}^{N}Y_e^{\alpha_e}$. We denote $\R_D[Y]$ as the set of all polynomials in the entries of $Y$ with degree at most $D$.

We identify the collection of all $r$-uniform simple hyperedges on $[n]$ with $\binom{[n]}{r} = \{e \subseteq [n]: |e| = r\}$, and the collection of all $r$-uniform multi-hyperedges on $[n]$ with $\mc=\{e=(i_1,\ldots i_r)\in [n]^{r}:i_1\leq i_2\leq \ldots \leq i_r\}$ (multi-choose notation). Note that $\alpha \in \{0,1\}^{\binom{[n]}{r}}$ can be viewed as an $r$-uniform hypergraph on the vertices $[n]$ where the hypergraph includes the hyperedge $e\in \binom{[n]}{r}$ if and only if $\alpha_{e} = 1$. Similarly any $\alpha \in \N^{\mc}$ can be viewed as an $r$-uniform multi-hypergraph on the vertices $[n]$ where the multigraph has $\alpha_{e}$ copies of the hyperedge $e$ for all $e\in \mc$. We let $V(\alpha) \subseteq [n]$ denote the set of non-isolated vertices of $\alpha$. We let $E(\al)$ denote the edge set of $\al$, i.e. the multi-set of $(i_1,\ldots, i_r)$ with $i_1\leq \ldots \leq i_r$ corresponding to the edges of the multi-hypergraph $\al$.

For a parameter $a$, we use the asymptotic notation $O_{a}(\cdot)$ and $o_{a}(\cdot)$, etc., to suppress constant multiplicative factors that depend only on $a$. We use $\widetilde{O}_{a}(\cdot)$ and $\widetilde{o}_{a}(\cdot)$, etc., to suppress multiplicative constants as well as polylogarithmic factors.

\subsection*{Acknowledgments}
We thank Alexander S. Wein for discussions that helped inspire the framework in Section~\ref{subsec:proof:conditioning}.

\subsection*{Organization}
The remainder of the paper is organized as follows. Section~\ref{sec:sparse:PCA} proves the estimation lower bound for sparse tensor PCA (Theorem~\ref{thm:sparse:PCA}-(a)). Section~\ref{sec:planted:hypergraph} treats the planted dense subhypergraph model: the large planted set regime (Theorem~\ref{thm:planted:hypergraph}-(a)) in Section~\ref{subsec:planted:hypergraph:lower:bound}, and the small planted set regime (Theorem~\ref{thm:small:planted:hypergraph}-(a)) via the conditioning argument in Section~\ref{subsec:small:hypergraph:lower:bound}. Section~\ref{sec:general:PCA} proves the estimation lower bound for tensor PCA with a general prior (Theorem~\ref{thm:general:PCA}). Section~\ref{sec:upper:bounds} establishes the matching upper bounds (Theorems~\ref{thm:sparse:PCA}-(b),~\ref{thm:planted:hypergraph}-(b), and~\ref{thm:small:planted:hypergraph}-(b)). We defer analysis of the polynomial-time recovery algorithms to the appendix.

\section{Sparse tensor PCA}\label{sec:sparse:PCA}
This section is devoted to the proof of Theorem~\ref{thm:sparse:PCA}-(a). We will follow the orthogonal expansion approach outlined in Section~\ref{subsec:orthogonal:expansion}. If $\rho=0$ then $\theta=0$ almost surely and the result is trivial, so we assume $\rho > 0$ without loss of generality. We index the basis and the orthonormal family by
\[
	\sG = \{\alpha \in \N^{\mc}: |\alpha| \leq D\}\,, \qquad \sP = \sG \times \{0, 1\}^{n}\,.
\]
That is, $\sG$ consists of all $r$-uniform multi-hypergraphs (with no isolated vertices) having at most $D$ edges, counting multiplicities.

Before proceeding to the proof, we make the following simplification. Observe that the entries of the symmetrized noise tensor $\Wsy$ in Definition~\ref{def:symmetric:sparse:PCA} with repeated indices have inflated variance relative to those entries having only distinct indices (see e.g. Proposition B.1 of the arXiv version of~\cite{kunisky24tensor}). It is convenient to instead consider a noise-reduced symmetric tensor $Z$ whose entries are i.i.d. $\cN(0, 1)$ random variables, up to symmetry. That is, $Z_{e} \stackrel{iid}{\sim} \cN(0, 1)$ for $e \in \mc$. Accordingly, we shall work with the noise-reduced observation
\[
    \Ynr = \lambda \theta^{\otimes r} + Z\,.
\]
Throughout, as we work exclusively with the noise-reduced model, we abuse notation and simply write $Y \equiv \Ynr$. It is sufficient to prove a lower bound for this noise-reduced model, since the noise-inflated tensor can be simulated by adding independent Gaussian noise, which in turn can be simulated by a polynomial; see~\cite[Appendix A, Claim A.3]{SW-22}.

\subsection{Construction of the dual certificate}
We begin by specifying the basis $(\phi_\alpha)_{\alpha\in\sG}$ and the orthonormal family $(\psi_{\beta\gamma})_{(\beta,\gamma)\in\sP}$. Write $N := \mc$ and let $\{h_k\}_{k\in\N}$ denote the univariate Hermite polynomials, normalized so that $\deg(h_k)=k$ and $\E[h_k(G)\,h_\ell(G)]=\In_{k=\ell}$ for $G\sim\cN(0,1)$ (see e.g.~\cite{sze-book}). The corresponding multivariate Hermite polynomials are $H_\alpha(y) = \prod_{e=1}^{N} h_{\alpha_e}(y_e)$, where $y\in (\R^n)^{\otimes r}$ and $\alpha \in \sG$. For $\al,\be\in \GGG$ and $\ga\in \{0,1\}^n$, we set
\begin{equation}
	\phi_{\alpha}(Y)= H_{\alpha}(Y)\,, \quad \psi_{\beta \gamma} = \psi_{\beta \gamma}(Z, \theta) = H_{\beta}(Z) \lPa \frac{\theta - \rho}{\sqrt{\rho(1 - \rho)}} \rPa^{\gamma}\,.
\end{equation}
Note that $(\phi_{\al})_{\al\in \GGG}$ forms a basis of $\R_D[Y]$. Moreover, as $Z$ and $\theta$ are independent, and both the Hermite polynomials and the centered Bernoulli monomials are orthonormal in their respective $L^2$ spaces, the family $(\psi_{\beta\gamma})$ is orthonormal.

\begin{lemma}\label{lem:sparse:PCA:c:M}
For all $\alpha \in \sG$ and pairs $(\beta, \gamma) \in \sP$, 
\begin{equation*}
    c_{\alpha} = \frac{\lambda^{|\alpha|} \rho^{|V(\alpha) \cup \{1\}|}}{\sqrt{\alpha!}}\,, \quad M_{\beta \gamma, \alpha} = \In_{\beta \leq \alpha} \In_{\gamma \subseteq V(\alpha - \beta)} \cdot \sqrt{\frac{\beta!}{\alpha!}} \binom{\alpha}{\beta} \lPa \sqrt{\frac{1-\rho}{\rho}} \rPa^{|\gamma|} \lambda^{|\alpha - \beta|} \rho^{|V(\alpha - \beta)|}\,.
\end{equation*}
\end{lemma}

\begin{proof}
Recall that $Y = \lambda\,\theta^{\otimes r} + Z$ and write $X = \lambda\,\theta^{\otimes r}$ for the signal tensor. The Hermite expansion (see e.g.~\cite[Proposition~3.1]{SW-22}) gives
\begin{equation}
\begin{aligned}
	H_{\alpha}(Y) = \prod_{e = 1}^{N} h_{\alpha_{e}}(X_{e} + Z_{e}) &= \prod_{e = 1}^{N} \sum_{j = 0}^{\alpha_{e}} \lPa \frac{j!}{\alpha_{e}!} \rPa^{\frac{1}{2}} \binom{\alpha_{e}}{j} \cdot X_{e}^{\alpha_{e} - j}h_{j}(Z_{e}) \\
																	&= \sum_{\beta \leq \alpha} \lPa \frac{\beta!}{\alpha!} \rPa^{\frac{1}{2}} \binom{\alpha}{\beta} \cdot X^{\alpha - \beta}H_{\beta}(Z)\,.
\end{aligned}
\end{equation}
Taking expectations, we compute using the orthonormality of Hermite polynomials, 
\[
	c_{\alpha} = \E[H_{\alpha}(Y)x] = \frac{\E[X^{\alpha} \theta_{1}]}{\sqrt{\alpha!}} = \frac{\lambda^{|\alpha|} \rho^{|V(\alpha) \cup \{1\}|}}{\sqrt{\alpha!}}\,.
\]
Similarly, we can compute, 
\[
\begin{aligned}
    M_{\beta \gamma, \alpha} &= \E \lBr H_{\alpha}(Y)H_{\beta}(Z) \lPa \frac{\theta - \rho}{\sqrt{\rho(1 - \rho)}} \rPa^{\gamma} \rBr = \In_{\beta \leq \alpha} \cdot\sqrt{\frac{\beta!}{\alpha!}} \binom{\alpha}{\beta} \lambda^{|\alpha - \beta|} \E \lBr \theta^{V(\alpha - \beta)} \lPa \frac{\theta - \rho}{\sqrt{\rho(1 - \rho)}} \rPa^{\gamma}\rBr,
\end{aligned}
\]
where we used $\theta_i^2=\theta_i$ in the last step. The final expectation equals $\In_{\gamma\subseteq V(\al-\be)} \lPa \sqrt{\frac{1-\rho}{\rho}} \rPa^{|\gamma|} \rho^{|V(\alpha - \beta)|}$, so this concludes the proof.
\end{proof}
As outlined in \textbf{Step 3} in Section~\ref{subsec:orthogonal:expansion}, we isolate the sets $\GGG_\star$ and $\PPP_\star$ in order to use Lemma~\ref{lem:reduce}. Recall that a multi-hypergraph is connected if every pair of vertices is connected by a path of hyperedges (see Definition~\ref{def:hypertree}).
\begin{definition}\label{def:good:sparse:PCA}
A non-empty multi-hypergraph $\alpha \in \sG$ is called \textit{good} if $1 \in V(\alpha)$ and $\alpha$ is connected. The empty graph $\alpha = 0$ is good by convention. A pair $(\beta, \gamma) \in \sP$ is \textit{good} whenever $\beta$ is good and $\gamma \subseteq V(\beta) \cup \{1\}$. The set of all good graphs resp. pairings is denoted $\GGG_\star$ resp. $\PPP_\star$.
\end{definition}

\begin{lemma}\label{lem:good:reduction:sparse:PCA}
If $\alpha \in \sG$ is not good, then there exists $\mu \in \R$ and a good graph $\alpha_{\star} \in \GGG_\star$ such that $c_{\alpha} = \mu c_{\alpha_{\star}}$ and $M_{\beta \gamma, \alpha} = \mu M_{\beta \gamma, \alpha_{\star}}$ for all good pairings $(\beta, \gamma) \in \PPP_\star$.
\end{lemma}

\begin{proof}
If $1 \not \in V(\alpha)$ take $\alpha_{\star} = 0$. Otherwise, let $\alpha_{\star}$ be the connected component of $\alpha$ containing vertex $1$. In either case, set $\alpha_{0} = \alpha - \alpha_{\star}$ and take $\mu = \E[\phi_{\alpha_0}]$. By independence of components, 
\[
	c_{\alpha} = \E[\phi_{\alpha_0} \cdot \phi_{\alpha_{\star}} \theta_{1}] = \E[\phi_{\alpha_0}] \cdot \E[\phi_{\alpha_{\star}} \theta_{1}] = \mu c_{\alpha_{\star}}\,.
\]
Now let $(\beta , \gamma) \in \PPP_\star$. If $\beta$ has no vertex in common with $\alpha_0$, then we can again factor using independence, 
\[
	M_{\beta \gamma, \alpha} = \E[\phi_{\alpha_0} \cdot \phi_{\alpha_{\star}} \psi_{\beta \gamma}] = \E[\phi_{\alpha_0}] \cdot \E[\phi_{\alpha_{\star}} \psi_{\beta \gamma}] = \mu M_{\beta \gamma, \alpha_{\star}}\,.
\]
In the case that $\beta$ shares a vertex with $\alpha_0$, since $1\in V(\be)$ and $\be$ is connected, $\beta$ cannot be a subgraph of $\alpha$ nor $\alpha_\star$. By Lemma~\ref{lem:sparse:PCA:c:M}, it follows that $M_{\be\ga,\al}=M_{\be\ga,\alpha_*}=0$.
\end{proof}
Having Lemma~\ref{lem:good:reduction:sparse:PCA} in hand, we construct the dual certificate $u$ as
\begin{equation}
	u_{\beta \gamma} = \left(-\sqrt{\frac{\rho}{1 - \rho}} \right)^{|\gamma|}c_{\beta}\,,\quad \textnormal{for}\quad (\be,\ga)\in \PPP_\star\,,
\label{eq:sparse:PCA:solution}
\end{equation}
and set $u_{\be\ga}=0$ if $(\be,\ga)\notin \PPP_\star$.

\begin{lemma}\label{lem:sparse:PCA:u}
For sparse tensor PCA, the choice of $u$ in \eqref{eq:sparse:PCA:solution} solves $M^{\top}u=c$, i.e.~\eqref{eq:u:condition}.
\end{lemma}

\begin{proof}
By Lemma~\ref{lem:good:reduction:sparse:PCA}, the assumptions of Lemma~\ref{lem:reduce} are satisfied, thus it suffices to prove that $\sum_{(\be,\ga)\in \PPP_\star}M_{\be\ga,\al}u_{\be\ga}=c_{\al}$ for all $\al\in \GGG_\star$. Using that $M_{\be\ga,\al}=0$ unless $\be\leq \al$ and $\ga\subseteq V(\al-\be)$ hold by Lemma~\ref{lem:sparse:PCA:c:M}, and the definition of $\PPP_\star$, this reduces to checking that 
\[
M_{\al 0,\al}u_{\al 0}=c_{\al}-\sum_{\substack{\be\in \GGG_\star\\ \be\lneq \al}}\sum_{\ga\subseteq V(\al-\be)\cap(V(\be)\cup\{1\})}M_{\be\ga,\al}u_{\be\ga}\,.
\]
Since the proposed candidate $u$~\eqref{eq:sparse:PCA:solution} satisfies $u_{\al0}=c_{\al}$ and $M_{\al0,\al}=1$, so it suffices to check the last summation equals $0$. The crucial property of $u$~\eqref{eq:sparse:PCA:solution} is that
$ M_{\beta \gamma, \alpha}  u_{\beta \gamma}= (-1)^{|\gamma|} M_{\beta 0, \alpha } c_{\beta}$ since $M_{\beta \gamma, \alpha}=(\sqrt{\frac{1-\rho}{\rho}})^{|\gamma|} M_{\beta 0, \alpha}$ holds by Lemma~\ref{lem:sparse:PCA:c:M}. Thus, for any fixed $\al, \beta\in \GGG_\star$, such that $\be\lneq \al$, we have
\[
\sum_{\ga\subseteq V(\al-\be)\cap (V(\be)\cup \{1\})}M_{\be\ga,\al} u_{\be\ga}=M_{\be 0,\al}c_{\be}\sum_{\ga\subseteq V(\al-\be)\cap (V(\be)\cup \{1\})}(-1)^{|\ga|}=0\,,
\]
where the last equality holds since $V(\al-\be)\cap (V(\be)\cup \{1\})\neq \emptyset$ for any $\al,\be\in \GGG_\star$ with $\be\lneq \al$. Combining the two displays above concludes the proof.
\end{proof}

\subsection{Counting connected hypergraphs}
We next derive the necessary combinatorial estimates to bound $\|u\|$. To motivate these estimates, since $u$ in \eqref{eq:sparse:PCA:solution} satisfies $M^{\top}u=c$ by Lemma~\ref{lem:sparse:PCA:u}, we have by Proposition~\ref{prop:duality} that
\[
\begin{split}
\rho \cdot \Corr_{\leq D}^2\leq \|u\|^2
&=\left(c_{0}^{2} \sum_{\gamma \subseteq \{1\}} \lPa \frac{\rho}{1 - \rho} \rPa^{\gamma} + \sum_{\alpha \in \GGG_\star\setminus\{0\}} c_{\alpha}^{2} \sum_{\gamma \subseteq V(\alpha)} \lPa \frac{\rho}{1 - \rho} \rPa^{\gamma}\right)\\
&=\frac{\rho^2}{1-\rho}+\sum_{\al\in \GGG_\star\setminus\{0\}}\frac{\la^{2|\al|}\rho^{2|V(\al)|}}{\al!}\left(1+\frac{\rho}{1-\rho}\right)^{|V(\al)|}\,.
\end{split}
\]
Let $\hat\rho = \frac{\rho}{\sqrt{1-\rho}}$. Then, with the crude bound $\al!\geq 1$, we have the following. Every good $\alpha\in\sG_\star\setminus\{0\}$ is a connected multi-hypergraph on some $k\geq r$ vertices (one of which is vertex~$1$) with some $\ell\geq 1$ edges. Summing over vertex count, edge count, and vertex labels, 
\begin{equation}\label{eq:bound:corr:sparse:PCA}
\rho \cdot \Corr^2_{\leq D}\leq \hat{\rho}^2\bigg(1+\sum_{\al\in \GGG_\star\setminus\{0\}}\hat{\rho}^{2|V(\alpha)| - 2} \lambda^{2|\alpha|}\bigg)\leq \hat{\rho}^2\bigg(1+\sum_{k = r}^{rD} \sum_{\ell = \ell_k}^{D} \binom{n - 1}{k - 1}N_{k, \ell}^{\mathrm{mult}} \hat{\rho}^{2k - 2} \lambda^{2\ell}\bigg)\,,
\end{equation}
where 
\[
    \ell_k = \lceil(k-1)/(r-1)\rceil
\]
is the minimum number of edges in a connected $r$-uniform hypergraph on $k$ vertices, and $N_{k,\ell}^{\mathrm{mult}}$ denotes the number of connected $r$-uniform multi-hypergraphs on $[k]$ with $\ell$ edges, allowing self-loops.

To bound the sum in~\eqref{eq:bound:corr:sparse:PCA} below the sharp threshold, we need sufficiently tight estimates on $N_{k,\ell}^{\mathrm{mult}}$ as a function of $k$ and $\ell$. The strategy is to compare multi-hypergraph counts first to simple hypergraph counts, and then to the number of \emph{hypertrees} (see Definition~\ref{def:hypertree}). Recall that a hypertree is any connected and acyclic $r$-uniform hypergraph having the property that any two of its edges share at-most a single vertex in common.

We will need a count of labeled \emph{hyperforests}---disjoint unions of hypertrees---in which each tree has a distinguished vertex designated as its \textit{root}. An $r$-uniform hyperforest on $[k]$ with $t$ connected components (each a hypertree) has $\ell = (k-t)/(r-1)$ edges.
\begin{lemma}\label{lem:rooted:forests}
Let $k\geq 1$ and $t\geq 1$ be integers such that $\ell = (k-t)/(r-1)\in\N$. The number of labeled $r$-uniform forests on $[k]$ consisting of $t$ rooted hypertrees is
\[
	R_{k,t} = \frac{k!}{(t-1)!}\,\frac{k^{\ell-1}}{\ell!\,(r-1)!^{\ell}}\,.
\]
In particular, if we let $N_k$ be defined by
\[
	N_k = \frac{(k-1)!\,k^{\ell_k-1}}{\ell_k!\,(r-1)!^{\ell_k}}\,,
\]
and setting $t=1$, we recover $R_{k,1} = kN_k$. That is, when $(k-1)/(r-1)\in \N$, $N_k$ counts the number of labeled $r$-uniform hypertrees on $[k]$. For $r=2$, this reduces to the classical Cayley's formula.
\end{lemma}

\begin{proof}
This is~\cite[Theorem 1]{lavault2011note}. See also~\cite[Equation 6.35]{Bedini2008}.
\end{proof}
\begin{remark}\label{rem:hypertree:vs:graph}
A key difference between the graph case $r=2$ and the hypergraph case $r\geq 3$ is that, for $r=2$, every connected graph on $k$ vertices contains a spanning tree, so it is straightforward to bound $N_{k,\ell}^{\mathrm{mult}}$ using Cayley's formula and the stars-and-bars formula (see~\cite[Lemma 3.3]{arxiv-version}). For $r\geq 3$, a connected hypergraph on $k$ vertices need not contain a spanning hypertree, thus a different argument is needed in the hypergraph setting.
\end{remark}
We next compare $N_{k,\ell}^{\mathrm{mult}}$ to $N_{k,\ell}$, and then $N_{k,\ell}$ to $N_k$, which we treat as the baseline.
\begin{lemma}\label{lem:multi:to:simple}
For all $k,\ell\in\N$,
\[
	N_{k,\ell}^{\mathrm{mult}} \;\leq\; \sum_{m=\ell_k}^{\ell} N_{k,m}\,\binom{\ell+k-1}{\ell-m}\,,
\]
where $N_{k,m}$ denotes the number of connected $r$-uniform simple hypergraphs on $[k]$ with $m$ edges.
\end{lemma}
 
\begin{proof}
Any connected multi-hypergraph on $[k]$ with $\ell$ edges can be generated by first choosing a connected simple hypergraph with $m$ distinct edges for some $\ell_k\leq m\leq \ell$, then distributing the remaining $\ell - m$ edges among the $m$ distinct edges and $k$ self loops. The number of ways to perform the second step is $\binom{\ell+k-1}{\ell-m}$ by the stars-and-bars formula.
\end{proof}

\begin{lemma}[Simple hypergraph count relative to trees]\label{lem:simple:to:tree}
Fix the number of vertices $k\geq r$ and the number of edges $\ell\geq \ell_k\equiv \lceil (k-1)/(r-1)\rceil$. Let $\Delta=\ell-\ell_k$ denote the excess number of edges and denote $q=(r-1)\ell_k-(k-1)\in [0,r-2]$. For a constant $C_r>0$ depending only on $r\geq 2$, we have
\[
	\frac{N_{k,\ell}}{N_k}\leq (C_r\,k\ell_k)^{q}\cdot \Big(C_r\,k\,\Big(1+\frac{\ell_k}{\Delta}\Big)\Big)^{(r-1)\Delta}\,,
\]
where for $\Delta=0$, the RHS is understood as $(C_r k\ell_k)^q$.
\end{lemma}
 
\begin{proof}
To count connected simple hypergraphs, we use the incidence bipartite graph representation: the vertex classes are $[k]$ (vertices) and $\{e_1,\ldots,e_\ell\}$ (edge-labels), with $v\in[k]$ adjacent to $e_j$ whenever $v\in e_j$. In this representation, a connected hypergraph corresponds to a connected bipartite graph. We first choose a spanning tree of this bipartite graph with degree sequences $(d_v)_{v\in[k]}$ and $(d_{e_j})_{j\in[\ell]}$.  The number of labeled bipartite trees with a given degree sequence is given by $\frac{(k-1)!\,(\ell-1)!}{\prod_{v}(d_v-1)!\prod_{j}(d_{e_j}-1)!}$ (see e.g.~\cite[Section 2.4]{moon1970counting}). We then complete the connected hypergraph by connecting each $e_j$ to $r - d_{e_j}$ additional vertices. Note that since the edges of the original connected hypergraph are not labeled, the labeled incident bipartite graph counts every connected hypergraph exactly $\ell!$ times. Altogether, this yields
\[
	N_{k,\ell} \;\leq\; \frac{1}{\ell!}\sum_{(d_v),\,(d_{e_j})}\frac{(k-1)!\,(\ell-1)!}{\prod_{v}(d_v-1)!\prod_{j}(d_{e_j}-1)!}\prod_{j=1}^{\ell}\binom{k - d_{e_j}}{r - d_{e_j}}\,,
\]
where the sum is over valid degree sequences satisfying $\sum_v(d_v-1)=\ell-1$ and $\sum_j(d_{e_j}-1)=k-1$. Applying the multinomial theorem to compute the sum over vertex degrees, the RHS equals
\[
\frac{(k-1)!k^{\ell-1}}{\ell!}\sum_{(d_{e_j})}\prod_{j=1}^{\ell}\frac{(k-d_{e_j})!}{(k-r)!(r-d_{e_j})!}\leq \frac{(k-1)!k^{\ell-1}}{\ell!}\sum_{(d_{e_j})}\prod_{j=1}^{\ell}\frac{k^{r-d_{e_j}}}{(r-d_{e_j})!}\,.
\]
Let $t=(r-1)\ell-k+1$ denote the cyclomatic number of the incident bipartite graph. Since $\sum_{j=1}^{\ell}(r-d_{e_j})=t$, applying Vandermonde's identity to simplify the sum over edge degrees in the RHS above yields
\[
	N_{k,\ell} \leq  \frac{(k-1)!\,k^{\ell-1+t}}{\ell!\,(r-1)!^\ell}\,\binom{\ell(r-1)}{t}\,,
\]
It remains to divide by $N_k$ and simplify. Recalling $N_k = (k-1)!\,k^{\ell_k-1}/(\ell_k!\,(r-1)!^{\ell_k})$, we obtain
\begin{equation*}
	\frac{N_{k,\ell}}{N_k} \;\leq\; \frac{\ell_k!}{\ell!}\cdot\frac{k^{\Delta+t}}{(r-1)!^{\Delta}}\cdot\binom{\ell(r-1)}{t}\,.
\end{equation*}
Note that since $\ell = \ell_k + \Delta$, we have $\frac{\ell_k!}{\ell!}\leq \frac{1}{\ell_k^{\Delta}}$. Also, by the standard bound $\binom{n}{m}\leq (en/m)^m$, we have $\binom{\ell(r-1)}{t} \leq (\frac{e(r-1)\ell}{t})^{t}$. Substituting these bounds into the above inequality gives
\[
	\frac{N_{k,\ell}}{N_k} \;\leq\; \frac{k^{\Delta+t}}{\ell_k^{\Delta}\,(r-1)!^{\Delta}}\bigg(\frac{e(r-1)\ell}{t}\bigg)^{\!t}=\underbrace{k^{q}\bigg(\frac{e(r-1)\ell}{t}\bigg)^{\!q}}_{\text{prefactor}} \;\times\; \underbrace{\bigg(\frac{k^r}{\ell_k\,(r-1)!}\bigg)^{\!\Delta}\bigg(\frac{e(r-1)\ell}{t}\bigg)^{\!(r-1)\Delta}}_{\text{$\Delta$-dependent terms}}\,,
\]
where the last equality holds since $t=(r-1)\Delta+q$. The prefactor term is clearly bounded above by $(C_r k\ell)^q$. For the $\Delta$-dependent terms, we can check directly the lemma holds when $\Delta=0$, so assume that $\Delta \geq 1$. Note that $k\leq (r-1)\ell_k+1\leq r\ell_k$, so $\frac{k^r}{\ell_k (r-1)!}\leq \frac{r}{(r-1)!}k^{r-1}$ holds. Moreover, we have $(r-1)\ell/t \leq 1 + \ell_k/\Delta$ since $t\geq (r-1)\Delta$. Thus
\[
	\bigg(\frac{k^r}{\ell_k\,(r-1)!}\bigg)^{\!\Delta}\bigg(\frac{e(r-1)\ell}{t}\bigg)^{\!(r-1)\Delta} \;\leq\; \bigg(C_r\,k\,\Big(1+\frac{\ell_k}{\Delta}\Big)\bigg)^{\!(r-1)\Delta}\,,
\]
Combining the two estimates yields the stated bound.
\end{proof}

\subsection{Proof of Theorem~\ref{thm:sparse:PCA}-(a)}

\begin{proof}[Proof of Theorem~\ref{thm:sparse:PCA}-(a)]
Throughout, we let $C_r, C_r'$ denote constants that only depend on $r$, which may change from line to line. We have from~\eqref{eq:bound:corr:sparse:PCA} that $\rho\cdot\Corr^2_{\leq D}\leq \hat{\rho}^2(1+\Xi)$, where
\begin{equation}\label{eq:def:Xi}
    \Xi:=\sum_{k = r}^{rD} \sum_{\ell = \ell_k}^{D} \binom{n - 1}{k - 1}N_{k, \ell}^{\mathrm{mult}} \hat{\rho}^{2k - 2} \lambda^{2\ell} = \sum_{k = r}^{rD} \binom{n - 1}{k - 1} N_{k} \hat{\rho}^{2k - 2} \lambda^{2\ell_{k}} \underbrace{\sum_{\ell = \ell_{k}}^{D} \frac{N_{k, \ell}^{\mathrm{mult}}}{N_{k}} \lambda^{2\ell - 2\ell_{k}}}_{A_{k}}\,.
\end{equation}
By the upper bound of $N_{k,\ell}^{\mathrm{mult}}$ from Lemma~\ref{lem:multi:to:simple}, we can bound $A_k$ as
\[
\begin{aligned}
    A_{k} \leq \sum_{\ell = \ell_{k}}^{D} \sum_{m = \ell_{k}}^{\ell} \frac{N_{k, m}}{N_{k}} \binom{\ell+k+1}{\ell-m} \lambda^{2\ell - 2\ell_{k}}= \sum_{m = \ell_{k}}^{D} \frac{N_{k, m}}{N_{k}} \lambda^{2m - 2\ell_{k}} \sum_{\ell = m}^{D} \binom{\ell+k-1}{\ell-m} \lambda^{2\ell-2m}\,.
\end{aligned}
\]
For $k\leq rD$ and $\ell \leq D$, we can crudely bound $\binom{\ell+k-1}{\ell-m} \leq ((r+1)D)^{\ell-m}$. Moreover, taking $C$ large enough in the assumption $D^{r-1}\lambda^{2} \leq 1/C$, guarantees that $(r+1)D\la^2\leq 1/2$. As a result, the final summation is at most $2$. Combining this with Lemma~\ref{lem:simple:to:tree} and making the substitutions $\Delta = m - \ell_{k}$,
\[
    A_{k} \leq 2(C_{r}k\ell_k)^{q} \cdot \sum_{\Delta \geq 0} \lPa C_{r}k \lPa 1 + \frac{\ell_{k}}{\Delta} \rPa  \rPa^{(r - 1)\Delta} \la^{2\Delta}=2(C_{r}k\ell_k)^{q} \cdot \sum_{\Delta \geq 0} \lPa C_{r}k \lPa 1 + \frac{\ell_{k}}{\Delta} \rPa\la_0^2  \rPa^{(r - 1)\Delta}\,,
\]
where we defined $\lambda_0=\la^{1/(r-1)}$ and for $\Delta=0$, the summand is understood as $1$.  Recall that $q\leq r-2$ (cf. Lemma~\ref{lem:simple:to:tree}) and $k\leq C_r \ell_k$, so we have $(C_r k\ell_k)^{q}\leq (C_r'\ell_k)^{2r-4}$. Note that the final summation equals
\[
\begin{aligned}
    &1 + \sum_{1 \leq \Delta \leq \ell_{k}} \lPa C_{r}k \lPa 1 + \frac{\ell_{k}}{\Delta} \rPa \lambda_{0}^{2} \rPa^{(r - 1)\Delta} + \sum_{\Delta \geq \ell_{k}} \lPa C_{r}k \lPa 1 + \frac{\ell_{k}}{\Delta} \rPa \lambda_{0}^{2} \rPa^{(r - 1)\Delta} \\
    &\leq 1 + \ell_{k} \sup_{\Delta > 0} \lPa \frac{C_{r} k \ell_{k} \lambda_{0}^{2}}{\Delta} \rPa^{(r - 1)\Delta} + \sum_{\Delta \geq 0} \lPa C_{r}k \lambda_{0}^{2} \rPa^{(r - 1)\Delta}\,.
\end{aligned}
\]
Since $k \leq rD$ and $D\lambda^{2/(r-1)}=D\lambda_0^2 \leq 1/C$, the last summation on the RHS is uniformly bounded provided the constant $C$ (depending on $r$) is chosen large enough. The second term can be computed explicitly as
\[
	\sup_{\Delta > 0} \lPa \frac{C_{r} k \ell_{k} \lambda_{0}^{2}}{\Delta} \rPa^{(r - 1)\Delta} = \exp \lPa C_{r}' k \ell_k \lambda_{0}^{2} \rPa \leq \exp \lPa C_{r}'k D\lambda_0^2 \rPa \leq \exp(C_r' k/C)\,.
\]
Let $\delta_0\equiv C_r/C$, which can be made arbitrarily small by taking $C$ sufficiently large. Collecting these estimates altogether, we deduce that
\[
    A_{k} \leq C_{r}\ell_k^{2r-4}e^{\delta_0k}\,.
\]
Inserting this bound into \eqref{eq:def:Xi} yields
\begin{equation*}
    \Xi \leq C_{r} \sum_{k = r}^{rD} \ell_{k}^{2r - 4}e^{\delta_{0}k} \binom{n - 1}{k - 1}N_{k} \hat{\rho}^{2k - 2} \lambda^{2\ell_{k}}\,.
\end{equation*}
Using the definition of $N_{k}$ and the inequality $\binom{n - 1}{k - 1} \leq \frac{n^{k - 1}}{(k - 1)!}$, we can bound
\[
	\binom{n - 1}{k - 1}N_{k} \hat{\rho}^{2k - 2} \lambda^{2\ell_{k}} \leq \frac{k^{\ell_{k} - 1}}{\ell_{k}!(r - 1)!^{\ell_{k}}}(n \hat{\rho}^{2})^{k - 1} \lambda^{2\ell_{k}} \leq \lPa \frac{ek}{\ell_{k}(r - 1)!} \rPa^{\ell_{k}}(n^{r - 1} \hat{\rho}^{2r - 2} \lambda^{2})^{\frac{k - 1}{r - 1}}\,,
\]
where the last inequality used Stirling's approximation $\ell_{k}! \geq (\ell_{k}/e)^{\ell_{k}}$. Since $k \leq 1 + (r - 1)\ell_{k}$ and $0 \leq \ell_{k} - \frac{k - 1}{r - 1} \leq 1$, the first term on the RHS is at most, 
\[
	\lPa \frac{k}{\ell_{k}(r - 1)} \rPa^{\ell_{k}} \lPa \frac{e}{(r - 2)!} \rPa^{\ell_k} \leq C_{r} \lPa 1 + \frac{1}{\ell_{k}(r - 1)} \rPa^{\ell_{k}} \lPa \frac{e}{(r - 2)!} \rPa^{\frac{k - 1}{r - 1}} \leq C_{r}' \lPa \frac{e}{(r - 2)!} \rPa^{\frac{k - 1}{r - 1}}\,.
\]
Hence, putting everything together, we have
\[
    \Xi\leq C_r \sum_{k \geq r} \ell_{k}^{2r - 4}e^{\delta_{0}k} \lPa \frac{en^{r - 1} \hat{\rho}^{2r - 2} \lambda^{2}}{(r - 2)!} \rPa^{\frac{k - 1}{r - 1}} \leq C_r \sum_{k \geq r} \ell_{k}^{2r - 2}e^{\delta_{0}k}(1 - \eps)^{\frac{k - 1}{r - 1}}\,.
\]
Since $\ell_k\leq k$, given $\eps>0$, we may choose $C=C(\eps,r)>0$ sufficiently large, i.e. $\delta_{0}\equiv C_r/C$ sufficiently small, so that the last sum is finite, depending only on $\eps, r$. Recalling that $\rho\cdot \Corr_{\leq D}^2\leq \hat{\rho}^2(1+\Xi)$ and $\hat{\rho}\equiv \rho/\sqrt{1-\rho}$, this completes the proof.
\end{proof}

\section{Planted dense subhypergraph}\label{sec:planted:hypergraph}
This section is devoted to the proof of Theorem~\ref{thm:planted:hypergraph}-(a)
for $\rho \geq n^{-1/2}$ and Theorem~\ref{thm:small:planted:hypergraph}-(a) for $\rho \leq n^{-1/2}$. Theorem~\ref{thm:intro:small-hypergraph} is a
restatement of Theorem~\ref{thm:small:planted:hypergraph} in terms of $\MMSE$ (cf.\ Fact~\ref{fact:corr:mmse}). Theorem~\ref{thm:intro:large-hypergraph}
in turn follows from Theorem~\ref{thm:planted:hypergraph}: for the lower bound it suffices to take $\SNR = 1-\eps$, raising $q_1$ if
necessary, since $\MMSE_{\leq D}$ is monotonically decreasing in $q_1$
(see~\cite[Claim~A.2]{SW-22}). Setting $\SNR = 1-\eps$ corresponds to
$\la^2 = C(\eps,r)/(n^{r-1}\rho^{2r-2})$, so that the degree condition
$D^{r-1} \leq 1/(C\la^2)$ in Theorem~\ref{thm:planted:hypergraph}-(a) becomes
$D \lesssim n\rho^2 = n^{2\xi-1}$, which is the degree range in
Theorem~\ref{thm:intro:large-hypergraph}.

\subsection{Setup and notation} 
We introduce some notation specific to the planted subhypergraph model. Let
\[
	S = \{i\in[n]:\theta_i=1\} \qquad \textnormal{and} \qquad K_S = \{e\in[N]:\theta_i=1\;\forall\,i\in V(e)\}
\]
denote the planted vertex set and the complete $r$-uniform hypergraph on $S$, respectively. We write $K_S^c = [N]\setminus K_S$ for its complement. The \emph{planted hypergraph} $X$ is the random subhypergraph of $Y$ induced by the planted vertices, i.e., $V(X)=S$ and $E(X)=\{e\in K_S:Y_e=1\}$. We use $\E\nolimits_\theta[\cdot]$ to denote expectation conditional on $\theta$.
 
The index sets for the basis and orthonormal family are
\[
	\sG \;=\; \{\alpha\in\{0,1\}^{N}:|\alpha|\leq D\}\,, \qquad \sP \;=\; \{(\beta,\gamma)\in\sG\times\{0,1\}^n:\gamma\subseteq V(\beta)\cup\{1\}\}.
\]
Note that $\sG$ now consists only of simple hypergraphs with at most $D$ edges and $\sP$ has the extra constraint that $\gamma\subseteq V(\be)\cup\{1\}$, in contrast to the multi-hypergraphs used in Section~\ref{sec:sparse:PCA}.

\subsection{Large planted dense subhypergraph}\label{subsec:planted:hypergraph:lower:bound}
This section is devoted to the proof of Theorem~\ref{thm:planted:hypergraph}-(a). As in the sparse tensor PCA model, we follow the orthogonal expansion approach of Section~\ref{subsec:orthogonal:expansion}.  Define the polynomials in $Y$
\[
    \phi_{\alpha} = \phi_{\alpha}(Y) = (Y - q_{0})^{\alpha}\,.
\]
Note that $(\phi_{\al})_{\al\in \GGG}$ forms a basis of $\R_D[Y]$. For $(\be,\ga)\in \PPP$, define the polynomials in $(Y,\theta)$
\[
    \psi_{\beta \gamma} = \psi_{\beta \gamma}(Y, \theta) = \lPa \frac{\theta - \rho}{\sqrt{\rho(1 - \rho)}} \rPa^{\gamma} \prod_{e \in K_{S}} \lPa \frac{Y_{e} - q_{1}}{\sqrt{q_{1}(1 - q_{1})}} \rPa^{\beta_{e}} \prod_{e \not \in K_{S}} \lPa \frac{Y_{e} - q_{0}}{\sqrt{q_{0}(1 - q_{0})}} \rPa^{\beta_{e}}\,.
\]
Using the conditional independence of the entries of $Y$ given $\theta$, it is straightforward to see that $(\psi_{\beta\gamma})_{(\beta,\gamma)\in \PPP}$ forms an orthonormal family (see e.g.~\cite[Lemma 4.2]{arxiv-version}). 
\begin{remark}
When $q_1 = 1$, the normalization $\sqrt{q_1(1-q_1)}$ above incurs a division by
zero, and the orthonormal family $(\psi_{\beta\gamma})$ is no longer defined. Our
final result nonetheless remains valid by a continuity argument: for any
$f \in \R_D[Y]$, the mean squared error $\E[(f(Y) - \theta_1)^2]$ is a polynomial
in the entries of $(q_0, q_1, \rho)$ and hence a continuous function of $q_1$ on
$[q_0, 1]$. Consequently $\MMSE_{\leq D}$ is continuous in $q_1$ (thus $\Corr_{\leq D}$ as well), and the bound
at $q_1 = 1$ follows by letting $q_1 \nearrow 1$ in the bound established for
$q_1 < 1$. We thus assume $q_1 < 1$ for the rest of this section.
\end{remark}
\begin{lemma}\label{lem:planted:hypergraph:c:M}
Let $\alpha \in \sG$ and $(\beta, \gamma) \in \sP$. For the planted dense subhypergraph model, $c_{\al}\equiv \E[\phi_{\al}\theta_1]$ and $M_{\be\ga,\al}\equiv \E[\phi_{\al}\psi_{\be\ga}]$ can be computed as
\begin{equation*}
    c_{\alpha}=\rho^{|V(\alpha) \cup \{1\}|}(q_{1} - q_{0})^{|\alpha|}
\end{equation*}
and
\begin{equation*}
    M_{\beta \gamma, \alpha} = \In_{\beta \leq \alpha} \cdot (q_{1} - q_{0})^{|\alpha - \beta|}(q_{0}(1 - q_{0}))^{\frac{|\beta|}{2}} \cdot \E \lBr \theta^{V(\alpha - \beta)} \lPa \frac{\theta - \rho}{\sqrt{\rho(1 - \rho)}} \rPa^{\gamma} \lPa \frac{q_{1}(1 - q_{1})}{q_{0}(1 - q_{0})} \rPa^{\frac{|\beta \cap K_{S}|}{2}} \rBr,
\end{equation*}
where we recall $\theta^{V(\al-\be)}\equiv \prod_{i\in V(\al-\be)}\theta_i$.
\end{lemma}

\begin{proof}
First observe that $\E_{\theta}[(Y-q_0)^{\alpha}]=\In_{\theta_{i}=1, \forall i \in V(\alpha)} \cdot (q_1-q_0)^{|\alpha|}$. Thus, 
\[
    c_{\alpha}=\E[\theta_1 \cdot \E\nolimits_{\theta}[(Y - q_{0})^{\alpha}]] = (q_1-q_0)^{|\alpha|} \cdot \E[\In_{\theta_{i}=1, \forall i \in V(\alpha) \cup \{1\}}]=\rho^{|V(\alpha) \cup \{1\}|}(q_{1} - q_{0})^{|\alpha|}\,.
\]
To compute $M_{\be\ga,\al}$, condition on $\theta$ and compute the expectation over edges in $K_S$ and $K_S^c$ separately: 
\[
    \E\nolimits_{\theta} \lBr \prod_{e \not \in K_{S}} (Y_{e} - q_{0})^{\alpha_{e}} \lPa \frac{Y_{e} - q_{0}}{\sqrt{q_{0}(1 - q_{0})}} \rPa^{\beta_{e}} \rBr = \In_{(\alpha \triangle \beta) \cap K_{S}^{c} = \emptyset} \cdot (q_{0}(1 - q_{0}))^{\frac{|\beta \cap K_{S}^{c}|}{2}}\,,
\]
and
\[
    \E\nolimits_{\theta} \lBr \prod_{e \in K_{S}} (Y_{e} - q_{0})^{\alpha_{e}} \lPa \frac{Y_{e} - q_{1}}{\sqrt{q_{1}(1 - q_{1})}} \rPa^{\beta_{e}} \rBr = \In_{(\beta \cap K_{S}) \leq (\alpha \cap K_{S})} \cdot (q_{1} - q_{0})^{|(\alpha \setminus \beta) \cap K_{S}|}(q_{1}(1 - q_{1}))^{\frac{|\beta \cap K_{S}|}{2}}\,.
\]
Noting that $\In_{(\alpha \triangle \beta) \cap K_{S}^{c} = 0} \In_{(\beta \cap K_{S}) \leq (\alpha \cap K_{S})}=\In_{\beta \leq \alpha} \In_{\al-\be\leq K_S}$, and $\al-\be\leq K_S$ iff $\theta_i=1$ for all $i\in V(\al-\be)$, the product of these two conditional expectations is
\[
\begin{aligned}
    & \phantom{=} \In_{\beta \leq \alpha} \In_{\theta_i=1, \forall i \in V(\alpha-\beta)} (q_{1} - q_{0})^{|\alpha - \beta|}(q_{0}(1 - q_{0}))^{\frac{|\beta \cap K_{S}^{c}|}{2}}(q_{1}(1 - q_{1}))^{\frac{|\beta \cap K_{S}|}{2}} \\
    &= \In_{\beta \leq \alpha}\,\theta^{V(\alpha - \beta)} \,(q_{1} - q_{0})^{|\alpha - \beta|}(q_{0}(1 - q_{0}))^{\frac{|\beta|}{2}}\, \lPa \frac{q_{1}(1 - q_{1})}{q_{0}(1 - q_{0})} \rPa^{\frac{|\beta \cap K_{S}|}{2}}\,.
\end{aligned}
\]
Taking outer expectation over $\theta$, the stated formula for $M_{\beta \gamma, \alpha}=\E[\phi_{\al}\psi_{\be\ga}]$ follows.
\end{proof}

The notion of good graph carries over from the sparse tensor PCA model (Definition~\ref{def:good:sparse:PCA}), with the only change being that $\sG$ now consists of simple hypergraphs rather than multi-hypergraphs. We restate the definition for convenience.
\begin{definition}\label{def:good:hypergraph}
A non-empty hypergraph $\alpha \in \sG$ is \textit{good} if $1 \in V(\alpha)$ and $\alpha$ is connected. The empty graph $\alpha = 0$ is considered good by convention. A pairing $(\beta, \gamma) \in \sP$ is \textit{good} whenever $\beta$ is good and $\gamma \subseteq V(\beta) \cup \{1\}$. The set of all good graphs resp. pairings is denoted $\GGG_\star$ resp. $\PPP_\star$.
\end{definition}

\begin{lemma}\label{lem:good:reduction:hypergraph}
If $\alpha \in \sG$ is not good, then there exists $\mu \in \R$ and a good graph $\hat{\alpha} \in \GGG_\star$ such that $c_{\alpha} = \mu c_{\hat{\alpha}}$ and $M_{\beta \gamma, \alpha} = \mu M_{\beta \gamma, \hat{\alpha}}$ for all good pairings $(\beta, \gamma) \in \PPP_\star$.
\end{lemma}

\begin{proof}
The proof is identical to the proof of Lemma~\ref{lem:good:reduction:sparse:PCA}.
\end{proof}
By Lemmas~\ref{lem:reduce} and~\ref{lem:good:reduction:hypergraph}, it suffices to find $u$ supported on $\sP_\star$ solving $M^{\top}u=c$.
The proposed dual certificate is
\begin{equation}
\label{eq:hypergraph:solution}
	u_{\beta \gamma} = (-1)^{|\gamma|} \lPa \frac{\rho}{1 - \rho} \rPa^{\frac{1}{2}|\gamma|}(q_{0}(1 - q_{0}))^{-\frac{1}{2}|\beta|}c_{\beta}\,.
\end{equation}
This dual certificate has the same structure as in the sparse PCA model; see~\eqref{eq:sparse:PCA:solution}.

\begin{proposition}
The vector $u$ in \eqref{eq:hypergraph:solution} solves $M^{\top}u=c$.
\end{proposition}

\begin{proof}
By Lemmas~\ref{lem:reduce} and~\ref{lem:good:reduction:hypergraph}, it suffices to verify $\sum_{(\beta,\gamma)\in\sP_\star}M_{\beta\gamma,\alpha}\,u_{\beta\gamma}=c_\alpha$ for all $\alpha\in\sG_\star$. Since $M_{\be\ga,\al}=0$ unless $\be\leq \al$, we consider the contribution from $\be\le \al$. Observe the crucial identity
\begin{equation}\label{eq:large:hypergraph:identity}
	\sum_{\gamma\subseteq V(\beta)\cup\{1\}}(-1)^{|\gamma|}\bigg(\frac{\rho}{1-\rho}\bigg)^{\!|\gamma|/2}\bigg(\frac{\theta-\rho}{\sqrt{\rho(1-\rho)}}\bigg)^{\!\gamma} = \bigg(\frac{1-\theta}{1-\rho}\bigg)^{\!V(\beta)\cup\{1\}} = \frac{\In_{\theta_{V(\beta)\cup\{1\}}=0}}{(1-\rho)^{|V(\beta)\cup\{1\}|}}\,,
\end{equation}
where $\theta_{V(\be)\cup\{1\}}=0$ means $\theta_i=0, \forall i\in V(\be)\cup\{1\}$. Combining with the computation of $M_{\be\ga,\al}$ from Lemma~\ref{lem:planted:hypergraph:c:M}, the contribution from a given $\beta\leq\alpha$  in $\sum_{\ga\subseteq V(\be)\cup\{1\}}M_{\be\ga,\al}u_{\be\ga}$ involves the product $\theta_{V(\alpha-\beta)}\In_{\theta_{V(\beta)\cup\{1\}}=0}$. For any proper good subgraph $\beta\lneq\alpha$, connectivity of $\alpha$ forces $V(\alpha-\beta)\cap(V(\beta)\cup\{1\})\neq\emptyset$, so this product vanishes. The only surviving term is $\beta=\alpha$, for which $\theta_{V(\be)\cup\{1\}}=0$ forces $|\be\cap K_S|=0$. Thus, combining Lemma~\ref{lem:planted:hypergraph:c:M} and \eqref{eq:large:hypergraph:identity}, we have
\[
	\sum_{(\beta,\gamma)\in\sP_\star}M_{\beta\gamma,\alpha}\,u_{\beta\gamma} \;=\; c_\alpha\,(1-\rho)^{-|V(\alpha)\cup\{1\}|}\cdot\P(\theta_{V(\alpha)\cup\{1\}}=0) \;=\; c_\alpha\,.\qedhere
\]
\end{proof}

\begin{proof}[Proof of Theorem~\ref{thm:planted:hypergraph}-(a)]
Write $\lambda = (q_1-q_0)/\sqrt{q_0(1-q_0)}$. From~\eqref{eq:hypergraph:solution} and Lemma~\ref{lem:planted:hypergraph:c:M}, the dual certificate takes the form $u_{\beta\gamma}=(-1)^{|\gamma|}(\rho/(1-\rho))^{|\gamma|/2}\lambda^{|\beta|}\rho^{|V(\beta)\cup\{1\}|}$, which is identical to the sparse tensor PCA certificate. Since the good simple hypergraphs are a subset of the good multi-hypergraphs, the norm bound from the proof of Theorem~\ref{thm:sparse:PCA}-(a) applies directly, yielding the stated bound on $\Corr_{\leq D}$.
\end{proof}

\subsection{Small planted dense hypergraph}\label{subsec:small:hypergraph:lower:bound}
In this section, we prove Theorem~\ref{thm:small:planted:hypergraph}-(a). In particular, we fix $\xi\in (0,1/2]$ and $0<a<b$, and set
\[
    \rho = n^{\xi - 1}\,\,, \quad q_{0} = n^{-b}\,\,, \quad q_{1} = n^{-a}\,.
\]
Note that by Theorem~\ref{thm:planted:hypergraph}-(a) established in the previous section, there exists $C=C_r>0$ such that if $n^{r-1}\rho^{2r-2}\frac{(q_1-q_0)^2}{q_0(1-q_0)}\leq 1/C_r$ and $D^{r-1}\frac{(q_1-q_0)^2}{q_0(1-q_0)}\leq 1/C_r$, then $\Corr_{\leq D}\leq Cn^{-\frac{\xi}{2}}$. It follows that if $2a-b>0$, then there exists $\delta=\delta(a,b)$ such that $\Corr_{\leq D}\lesssim n^{-\frac{\xi}{2}}$. Thus, throughout this section, we focus on the other regime, where
\begin{equation}\label{eq:assume:small:PDS}
    b \geq 2a \qquad \text{and} \qquad a>b\xi\,.
\end{equation}

\begin{remark}\label{rem:small:hypergraph:quantitative}
The arguments in this section can be carried out for parameters $\rho, q_{0}, q_{1}$ which do not fall into this scaling regime. In the general ‘small planted hypergraph’ setting where $\rho \leq n^{-\frac{1}{2}}$, it is possible to obtain a low-degree hardness statement which tracks the dependence on the parameters and $D$, similar to Theorem~\ref{thm:sparse:PCA}-(a), under the assumption that there exists an $\eps$ for which Lemma~\ref{lem:event:bound} holds. A sufficient condition for this is the following: there is a positive constant $\tau$ such that, 
\[
    n\rho (D^{r}q_{1})^{\tau} = o(1)\,.
\]
Here, $\tau$ should be interpreted as the smallest density of a ‘dense’ subgraph, i.e. those graphs which we want to remove using conditioning. Assuming such a condition, it is possible to prove a low-degree hardness statement for $D \leq n^{\eps}$ under the hypothesis
\[
    n\rho^2 \lPa \frac{D^Cq_1^2}{q_0(1-q_0)} \rPa^{\tau} \leq 1/C\,,
\]
where $C$ is a constant which may depend on $\tau$ and $r$. In the scaling regime considered in this section, we may take $\tau=\xi/a+\delta$, for some small positive slack $\delta$, in which case the above constraint recovers the threshold $a>b\xi$.
\end{remark}

\subsubsection{High probability event}
We fix a constant $\delta=\delta(r,a,b,\xi) \in (0, 1)$ whose value will be determined later in Eq.~\eqref{eq:constraint:delta} below. We recall the notion of ‘dense’ graph, which appeared in~\cite{DMW-23} for proving detection hardness.

\begin{definition}\label{def:sparse:dense:graph}
A graph $\alpha$ is called sparse if $|\alpha| \leq (\frac{\xi}{a} + \delta) |V(\alpha)|$.
Any graph which is not sparse is called dense. For $k \in \N$, define 
\[
m_{k} := \bigg\lfloor \bigg(\frac{\xi}{a} + \delta\bigg)k \bigg\rfloor\,,
\]
and write $m_{\alpha} := m_{|V(\alpha)|}$ so that $\alpha$ is sparse if and only if $|\alpha| \leq m_{\alpha}$.
\end{definition}

Recall that $X$ denotes the planted hypergraph. Consider the event, 
\begin{equation}
    \cE \equiv \cE_D = \{\text{every subgraph of $X$ with at most $D$ edges is sparse.}\}
\end{equation}
The next lemma, an estimation analog of~\cite[Lemma 3.4]{DMW-23}, shows that the event $\cE$ holds with high probability, even after fixing a small subset of $\theta$. For a subset $A \subseteq [n]$, recall $\theta_A\equiv (\theta_{i})_{i \in A}$.

\begin{lemma}\label{lem:event:bound}
There exists a constant $\eps=\eps(r,a,\xi,\delta)>0$ such that for all $n$ sufficiently large and $D \leq n^{\eps}$, 
\[
    \min_{\substack{A \subseteq [n], |A| \leq n^{\xi} \\ \sigma \in \{0, 1\}^{A}}} \P(\cE \mid \theta_{A} = \sigma) \geq 1 - \tfrac{1}{2}n^{-a\delta/2}\,.
\]
\end{lemma}

\begin{proof}
Fix $A \subseteq [n]$ with $|A| \leq n^{\xi}$. Define the $(\theta_{i})_{i \not \in A}$-measurable event $\cA = \{\sum_{i \not \in A} \theta_i \leq \tfrac{3}{2}(n - |A|) \rho\}$. Since $\sum_{i \not \in A} \theta_i$ is a sum of i.i.d.\ $\Ber(\rho)$ variables, a Chernoff bound yields $\Pb(\cA^{c}) \leq e^{-C(n - 1) \rho}$ for a universal constant $C>0$. As $\cA$ is independent of $\theta_A$, we have for any $\sigma \in \{0,1\}^A$ that
\[
    \Pb(\cE^{c} \mid \theta_{A} = \sigma) \leq \Pb(\cE^{c} \cap \cA \mid \theta_{A} = \sigma) + \Pb(\cA^{c}) \leq \Pb(\cE^{c} \mid \cA, \theta_{A} = \sigma) + e^{-C(n - 1) \rho}\,.
\]
It thus suffices to bound $\Pb(\cE^{c} \mid \cA, \theta_{A} = \sigma)$. Fix any $\theta$ with $\cA$ holding and $\theta_A=\sigma$. Conditional on this realization of $\theta$, there are at most $n^{\xi} + \tfrac{3}{2}(n - 1) \rho \leq 3n^{\xi}$ vertices in the planted graph. For integers $r\leq k\leq rD$ and $m_k+1\leq \ell \leq D$, the number of candidate dense subgraphs with $k$ vertices and $\ell$ edges inside the planted graph is at most
\[
    \binom{3n^{\xi}}{k} \binom{\binom{k}{r}}{\ell} \leq (3n^{\xi})^{k} \lPa \frac{ek}{r} \rPa^{r\ell} \leq (3n^{\xi})^{k}(eD)^{r\ell}\,,
\]
where the first inequality uses the bound $\binom{k}{r}\leq (ek/r)^{r}$. Each such subgraph is contained in the planted hypergraph with probability at most $q_1^\ell = n^{-a\ell}$. Thus, a union bound over $k$ and $\ell$ gives
\[
	\P(\cE^c \mid \theta) \;\leq\; \sum_{k\geq r}\;\sum_{\ell\geq m_k+1}(3n^\xi)^k\,\big((eD)^r n^{-a}\big)^\ell\,.
\]
Choose $\eps<a/r$ small enough so that $(eD)^r n^{-a}\leq  1/2$ whenever $D\leq n^\eps$ and $n$ is sufficiently large. The inner geometric sum is then dominated by its first term, giving
\[
	\P(\cE^c\mid\theta) \leq 2\sum_{k\geq r}(3n^\xi)^k\,\big((eD)^r n^{-a}\big)^{(\xi/a+\delta)k} \;=\; 2\sum_{k\geq r}\big(3(eD)^{r(\xi/a+\delta)}\,n^{-a\delta}\big)^k\,,
\]
where we used $m_k+1\geq (\xi/a+\delta)k$. Choosing $\eps > 0$ small enough that $3(eD)^{r(\xi/a+\delta)}n^{-a\delta}\leq n^{-a\delta/2}$ for all $D\leq n^{\eps}$ and $n$ sufficiently large, the sum in the RHS is at most $4n^{-ra\delta/2}$. Because this holds uniformly over all $\theta$ with $\cA$ holding and $\theta_A=\sigma$, we have for all sufficiently large $n$
\[
	\P(\cE^{c}\mid\theta_A=\sigma)\leq 4n^{-ra\delta/2}+e^{-C(n-1)\rho} \leq \tfrac{1}{2}n^{-a\delta/2}\,.
\]
As this bound is uniform over all $|A|\leq n^{\xi}$ and $\sigma\in \{0,1\}^A$, this concludes the proof.
\end{proof}
For the rest of this section, we assume $D\leq n^{\eps}$ for $\eps>0$ appearing in Lemma~\ref{lem:event:bound}. We further assume $\eps < \xi$ so that $|V(\alpha) \cup \{1\}| \leq rn^{\eps}+1 \leq n^{\xi}$ for $\al \in \sG$ and $n$ sufficiently large. In particular, Lemma~\ref{lem:event:bound} implies that
\begin{equation}\label{eq:event:probability:below:by:1:2}
 \min_{\substack{\alpha \in \sG \\ \sigma \in \{0, 1\}^{V(\al)\cup\{1\}}}} \P(\cE \mid \theta_{V(\al)\cup\{1\}} = \sigma)\geq \frac{1}{2}\,.
\end{equation}

\subsubsection{Linear system with conditioning} 
We now implement the orthogonal expansion approach with conditioning $\cE$ described in Section~\ref{subsec:proof:conditioning}. The index sets $\sG$ and $\sP$, the basis $(\phi_\alpha)_{\alpha\in\sG}$, and the orthonormal family $(\psi_{\beta\gamma})_{(\beta,\gamma)\in\sP}$ remain as defined in Section~\ref{sec:planted:hypergraph}. By Proposition~\ref{prop:orthogonal:expansion:conditioning}, it suffices to determine a solution $u$ to the modified system $\widetilde{M}^\top u = \widetilde{c}$, where
\[
    \wc_{\alpha} = \E[\phi_{\alpha} \theta_1 \In_{\cE}]\,, \qquad \wM_{\beta \gamma, \alpha} = \E[\phi_{\alpha} \psi_{\beta \gamma} \In_{\cE}]\,,
\]
with $\|u\|$ sufficiently small. The conditioning controls the contribution from dense subgraphs, but introduces a complication: disconnected components of $Y$ may no longer be independent conditional on the event $\cE$, so the reduction to good graphs (Lemma~\ref{lem:good:reduction:hypergraph}) no longer applies.  The key challenge is to show that the coefficients $u_{\beta\gamma}$ for disconnected $\beta$ are sufficiently small, decaying fast enough in the number of connected components of $\be$.

\paragraph{Candidate dual certificate}
In the unconditional setting of Section~\ref{sec:planted:hypergraph}, the dual certificate was constructed so that the summation over $\gamma$ vanishes for all proper good subgraphs $\beta\lneq\alpha$, via the identity~\eqref{eq:large:hypergraph:identity}. That cancellation relied on $\alpha$ being good (connected, containing vertex~$1$). The reduction to good graphs is no longer available in this conditioned setting.  The identity~\eqref{eq:large:hypergraph:identity} nonetheless remains useful: it shows that the summation over $\gamma$ localizes to the set $\sC_\alpha$ of unions of connected components of $\alpha$, enabling a recursive construction of a candidate dual certificate.
\begin{lemma}\label{lem:small:hypergraph:key:identity}
For $\alpha, \beta \in \sG$, 
\begin{equation}
    \sum_{\gamma \subseteq V(\beta) \cup \{1\}} \lPa -\sqrt{\frac{\rho}{1 - \rho}} \rPa^{|\gamma|}\wM_{\beta \gamma, \alpha} = \In_{\beta \in \sC_{\alpha}} \cdot (q_{0}(1 - q_{0}))^{|\be|/2} H_{\beta \alpha}\,,
\label{eq:small:hypergraph:identity}
\end{equation}
where
\[
    \sC_{\alpha} = \{\beta \leq \alpha: \text{$\beta$ is a union of connected components of $\alpha$ and $1 \in V(\beta)$ if and only if $1 \in V(\alpha)$}\}\,,
\]
and for $\be\in \sC_{\al}$,
\[
    H_{\beta \alpha} := \E[\In_{\cE}\,\In_{\theta_{V(\alpha - \beta)} = 1} \cdot (Y - q_{0})^{\alpha - \beta} \mid \theta_{V(\beta) \cup \{1\}} = 0]\,.
\]
Here, for $\al=\be$, the indicator $\In_{\theta_{V(\al-\be)}=1}$ is understood as $1$ deterministically.
\end{lemma}

\begin{proof}
We first condition on $\theta$.
Since $\cE$ is measurable with respect to the edges in $K_S$, the edges outside $K_S$ remain conditionally independent of $\cE$. Thus, the conditional expectation $\E\nolimits_\theta[\phi_\alpha\,\psi_{\beta\gamma}\,\In_\cE]$ factors as
\begin{equation}\label{eq:cond:factor}
\bigg(\frac{\theta-\rho}{\sqrt{\rho(1-\rho)}}\bigg)^{\!\gamma}\cdot\E\nolimits_\theta\bigg[\In_\cE\prod_{e\in K_S}(Y_e-q_0)^{\alpha_e}\bigg(\frac{Y_e-q_1}{\sqrt{q_1(1-q_1)}}\bigg)^{\!\beta_e}\bigg]\cdot\E\nolimits_\theta\bigg[\prod_{e\notin K_S}(Y_e-q_0)^{\alpha_e}\bigg(\frac{Y_e-q_0}{\sqrt{q_0(1-q_0)}}\bigg)^{\!\beta_e}\bigg].
\end{equation}
The expectation over $K_S^c$ is computed exactly as in Lemma~\ref{lem:planted:hypergraph:c:M}:
\begin{equation}\label{eq:KSc:factor}
	\E\nolimits_\theta\bigg[\prod_{e\notin K_S}(Y_e-q_0)^{\alpha_e}\bigg(\frac{Y_e-q_0}{\sqrt{q_0(1-q_0)}}\bigg)^{\!\beta_e}\bigg] =\In_{(\alpha\triangle\beta)\cap K_S^c=\emptyset}\cdot(q_0(1-q_0))^{|\beta\cap K_S^c|/2}\,.
\end{equation}
Here, we used that if $(\alpha\triangle\beta)\cap K_S^c=\emptyset$, then $\be\cap K_S^{c}=\al\cap K_S^{c}$.
We now sum over $\gamma\subseteq V(\beta)\cup\{1\}$. Applying the identity~\eqref{eq:large:hypergraph:identity} to the $((\theta-\rho)/\sqrt{\rho(1-\rho)})^{\gamma}$ terms in~\eqref{eq:cond:factor},
\begin{equation}\label{eq:gamma:summed}
\begin{aligned}
	&\sum_{\gamma\subseteq V(\beta)\cup\{1\}}\bigg(-\sqrt{\frac{\rho}{1-\rho}}\bigg)^{\!|\gamma|}\E\nolimits_\theta[\phi_\alpha\,\psi_{\beta\gamma}\,\In_\cE]\\
	&=\In_{\theta_{V(\beta)\cup\{1\}}=0}\,\In_{(\alpha\triangle\beta)\cap K_S^c=\emptyset}\cdot(1-\rho)^{-|V(\beta)\cup\{1\}|}(q_0(1-q_0))^{|\beta|/2}\cdot\E\nolimits_\theta\big[\In_\cE\,(Y-q_0)^{\alpha\cap K_S}\big]\,,
\end{aligned}
\end{equation}
where we used that if $\theta_{V(\be)\cup\{1\}}=0$, then we have $\be\cap K_S=\emptyset$. Note that the $(1-\rho)^{-|V(\be)\cup\{1\}|}$ term is precisely $\big(\P(\theta_{V(\be)\cup\{1\}}=0)\big)^{-1}$. Thus, taking outer expectation w.r.t. $\theta$ yields
\begin{equation}\label{eq:after:tower}
	\sum_{\gamma\subseteq V(\beta)\cup\{1\}}\left(-\sqrt{\frac{\rho}{1-\rho}}\right)^{|\gamma|}\widetilde{M}_{\beta\gamma,\alpha}=(q_0(1-q_0))^{|\beta|/2}\cdot\E\big[\In_{(\alpha\triangle\beta)\cap K_S^c=\emptyset}\,\In_\cE\,(Y-q_0)^{\al\cap K_S}\;\big|\;\theta_{V(\beta)\cup\{1\}}=0\big]\,.
\end{equation}
Now observe that $\theta_{V(\beta)\cup\{1\}}=0$ implies $\beta\subseteq K_S^c$. Under this condition, we have $(\alpha\triangle\beta)\cap K_S^c=\emptyset$ if and only if the following two conditions hold: (i) $\beta\leq\alpha$, and (ii) every edge of $\alpha-\beta$ lies in $K_S$. Recalling that $K_S=\{e\in [N]:\theta_i=1, \forall i \in V(e)\}$, the condition (ii) is equivalent to $\theta_{V(\al-\be)}=1$. Thus, conditions (i) and (ii) together force $V(\alpha-\beta)\cap(V(\beta)\cup\{1\})=\emptyset$, which means $\beta$ must be a union of connected components of $\alpha$ and $1\in V(\beta)$ whenever $1\in V(\alpha)$. Therefore, the right-hand side of~\eqref{eq:after:tower} vanishes unless $\beta\in\sC_\alpha$. When $\be\in \sC_{\al}$ and $\be\subseteq K_S^{c}$, the constraint $\theta_{V(\al-\be)}=1$ implies that $\al\cap K_S=\al-\be$, so the conditional expectation in~\eqref{eq:after:tower} reduces to $H_{\be\al}$, concluding the proof.
\end{proof}

Having Lemma~\ref{lem:small:hypergraph:key:identity} in hand, we search for a dual certificate $u$ of the form, 
\begin{equation}
    u_{\beta \gamma} = \lPa -\sqrt{\frac{\rho}{1 - \rho}} \rPa^{|\gamma|}(q_{0}(1 - q_{0}))^{-|\be|/2} \cF(\beta)\,.
\label{eq:small:hypergraph:ansatz}
\end{equation}
for a function $\cF: \sG \rightarrow \R$ to be determined. Substituting~\eqref{eq:small:hypergraph:ansatz} into the modified system $\widetilde{M}^\top u = \widetilde{c}$ and applying Lemma~\ref{lem:small:hypergraph:key:identity} gives the following corollary.
\begin{corollary}\label{cor:small:hypergraph:recursion}
Define a function $\cF:\sG\to\R$ recursively as follows. If $\alpha=0$ or $\alpha$ is connected and contains vertex~$1$, let $\cF(\al)=\wc_{\al}/\P(\cE \mid \theta_{V(\alpha) \cup \{1\}} = 0)$. Otherwise, for general $\alpha$, define
\begin{equation}
    \cF(\alpha) = \frac{1}{\Pb(\cE \mid \theta_{V(\alpha) \cup \{1\}} = 0)} \lPa \wc_{\alpha} - \sum_{\beta \in \sC_{\alpha} \setminus \{\alpha\}} H_{\beta \alpha} \cF(\beta) \rPa\,,
\label{eq:recursion:F}
\end{equation}
where the recursion is well-defined since every $\beta\in\sC_\alpha\setminus\{\alpha\}$ has strictly fewer connected components than $\alpha$, and the denominator is positive by \eqref{eq:event:probability:below:by:1:2}. Then the vector $u = (u_{\beta\gamma})_{(\beta,\gamma)\in\sP}$ defined by~\eqref{eq:small:hypergraph:ansatz} solves $\widetilde{M}^\top u = \widetilde{c}$.
\end{corollary}
\begin{proof}
  By Lemma~\ref{lem:small:hypergraph:key:identity}, the system $\widetilde{M}^\top u = \widetilde{c}$ with $u$ of the form~\eqref{eq:small:hypergraph:ansatz} is equivalent to
\[
\sum_{\beta\in\sC_\alpha}H_{\beta\alpha}\cF(\beta) = \widetilde{c}_\alpha\,,\qquad\forall\alpha\in\sG\,.
\]
Since $H_{\alpha\alpha} = \P(\cE\mid\theta_{V(\alpha)\cup\{1\}}=0)$, isolating the $\beta=\alpha$ term and rearranging gives exactly~\eqref{eq:recursion:F}.
\end{proof}
The squared norm of the dual certificate $u$ given in Corollary~\ref{cor:small:hypergraph:recursion} is
\[
    \|u\|^{2} = \sum_{\alpha \in \sG} \frac{\cF(\alpha)^{2}}{(q_{0}(1 - q_{0}))^{|\alpha|}} \sum_{\gamma \subseteq V(\alpha) \cup \{1\}} \lPa \frac{\rho}{1 - \rho} \rPa^{|\gamma|} = \sum_{\alpha \in \sG} \frac{\cF(\alpha)^{2}}{(q_{0}(1 - q_{0}))^{|\alpha|}(1 - \rho)^{|V(\alpha) \cup \{1\}|}} =: \sum_{\alpha \in \sG} \cK(\alpha)\,,
\]
where the second equality uses the identity $\sum_{\gamma\subseteq A}(\rho/(1-\rho))^{|\gamma|}=(1-\rho)^{-|A|}$. We separate the last summation on the RHS as
\begin{equation}\label{eq:bound:u:in:terms:of:Xi}
    \|u\|^2=\sum_{\text{$\alpha$ is sparse}} \cK(\al) + \sum_{\text{$\alpha$ is dense}} \cK(\al)=:\Xi_{\textsf{sparse}}+\Xi_{\textsf{dense}}\,.
\end{equation}
where $\Xi_{\mathsf{sparse}}$ and $\Xi_{\mathsf{dense}}$ denote the contribution from sparse and dense $\al$, respectively.

To bound $\|u\|^2$, we require pointwise estimates on $|\cF(\alpha)|$. The appropriate bound depends on whether $\alpha$ is sparse or dense (see Definition~\ref{def:sparse:dense:graph}). We introduce two reference quantities:
\begin{equation}\label{eq:def:s:d}
    s_{\alpha} := \rho^{|V(\alpha)|}(2q_{1})^{|\alpha|}\,,\qquad
    d_{\alpha} := \rho^{|V(\alpha)|}D^{m_{\alpha}}q_{0}^{|\alpha|} \lPa \frac{2q_{1}}{q_{0}} \rPa^{m_{\alpha}}\,.
\end{equation}
We will show that, up to an additional complexity factor, $\cF(\alpha)$ is comparable to $s_{\alpha}$ when $\alpha$ is sparse and $\cF(\alpha)$ is comparable to $d_{\alpha}$ when $\alpha$ is dense. An observation that we will repeatedly use is
\begin{equation}\label{eq:observation:alpha:sparse}
    \textnormal{$\al$ is sparse}\implies s_{\al}\le d_{\al}\,.
\end{equation}
Another crucial observation is that if $\al$ is dense with $|\al| \gg m_{\al}$, then $d_{\al} \ll s_{\al}$ for $D \leq n^{\eps}$ and small enough $\eps$. The improvement from $s_{\al}$ to $d_{\al}$ comes from conditioning on $\cE$, which ensures that at most $m_\alpha$ edges of $\alpha$ lie inside the planted hypergraph $X$; the remaining $|\alpha|-m_\alpha$ edges lie outside $X$ and contribute $q_0$ each, giving the scale $d_\alpha$. We begin by bounding $\widetilde{c}_\alpha$ and $H_{\beta \alpha}$ in terms of $s_\alpha$ and $d_{\alpha}$, depending on whether or not $\alpha$ is dense. Afterwards, we use the recursion~\eqref{eq:recursion:F} to bound $\cF(\alpha)$.
We remark that when $1\notin V(\al)$, the extra factor of $\rho$ in the bound of $|\wc_{\al}|$ in the next lemma plays an important role for showing that graphs not containing vertex $1$ contribute negligibly to $\|u\|^2$.
\begin{lemma}\label{lem:small:hypergraph:bound:c:A}
For any $\alpha \in \sG$ and $\beta \in \sC_{\alpha}$, 
\[
    |\widetilde{c}_{\alpha}| \leq \rho^{\In_{1 \notin V(\alpha)}} s_{\alpha} \,\quad\textnormal{and}\quad |H_{\beta \alpha}| \leq s_{\alpha-\beta}\,.
\]
Moreover, if $\alpha$ is dense per Definition~\ref{def:sparse:dense:graph} and $\be\in \sC_{\al}$, 
\[
    |\widetilde{c}_{\alpha}| \leq \rho^{\In_{1 \notin V(\alpha)}} d_{\alpha} \,\quad\textnormal{and}\quad |H_{\beta \alpha}| \leq d_{\alpha-\beta}\,.
\]
\end{lemma}

\begin{proof}
We first establish two pointwise bounds on the conditional expectation $\E\nolimits_\theta[\In_\cE\,|Y-q_0|^\alpha]$, and then apply them to $\widetilde{c}_\alpha$ and $H_{\beta\alpha}$ separately.

Fix $\theta$ and suppose $\theta_{V(\alpha)}=1$, so that $\alpha\subset K_S$. Each edge $e\in E(\alpha)$ then has $Y_e\sim\Ber(q_1)$, and $\E[|Y_e-q_0|] = q_1(1-q_0)+q_0(1-q_1)\leq 2q_1$. We claim
\begin{equation}\label{eq:pointwise:generic}
	\E\nolimits_\theta[\In_\cE\,|Y-q_0|^\alpha] \leq (2q_1)^{|\alpha|}\,,
\end{equation}

\begin{equation}\label{eq:pointwise:dense}
	\text{if $\alpha$ is dense:}\qquad \E\nolimits_\theta[\In_\cE\,|Y-q_0|^\alpha] \leq D^{m_\alpha}\,q_0^{|\alpha|}\bigg(\frac{2q_1}{q_0}\bigg)^{\!m_\alpha}\,.
\end{equation}
The bound~\eqref{eq:pointwise:generic} follows immediately from $\In_\cE\leq 1$ and the product structure of $|Y-q_0|^\alpha$. For~\eqref{eq:pointwise:dense}, observe that on the event $\cE$, the subgraph $\{e\in E(\alpha):Y_e=1\}$ has at most $m_\alpha$ edges since $\alpha\subset K_S$ and the event $\cE$ forces every subgraph of the planted graph $X$ with at most $D$ edges to be sparse. Thus at most $m_\alpha$ edges have $Y_e=1$, contributing $|1-q_0|\leq 1$ each, and the remaining $|\alpha|-m_\alpha$ edges have $Y_e=0$, contributing $q_0$ each. Summing over the $\binom{|\alpha|}{\ell}\leq D^\ell$ choices for the $\ell\leq m_\alpha$ edges with $Y_e=1$,
\[
	\E\nolimits_\theta[\In_\cE\,|Y-q_0|^\alpha] \leq \sum_{\ell=0}^{m_\alpha}\binom{|\alpha|}{\ell}\,q_1^\ell\,q_0^{|\alpha|-\ell} \leq (m_\alpha+1)\,D^{m_\alpha}\,q_0^{|\alpha|}\bigg(\frac{q_1}{q_0}\bigg)^{\!m_\alpha} \leq D^{m_\alpha}\,q_0^{|\alpha|}\bigg(\frac{2q_1}{q_0}\bigg)^{\!m_\alpha},
\]
where the second inequality uses $q_1\geq q_0$, and the last uses $m_\alpha+1\leq 2^{m_\alpha}$.
 
On the other hand, suppose $\theta_{V(\alpha)}\neq 1$. That is, some vertex $i\in V(\alpha)$ has $\theta_i=0$. Then $i\notin S$, so every edge $e\in E(\alpha)$ containing $i$ lies outside $K_S$ and satisfies $Y_e\sim\Ber(q_0)$ conditionally independently of $\cE$ given $\theta$. Since $\E\nolimits_{\theta}[Y_e-q_0]=0$ for such edges, we have
\begin{equation}\label{eq:pointwise:vanish}
	\theta_{V(\alpha)}\neq 1 \implies \E\nolimits_\theta[\In_{\mathcal{E}}(Y-q_0)^\alpha] = 0\,.
\end{equation}
With \eqref{eq:pointwise:generic}, \eqref{eq:pointwise:dense}, and \eqref{eq:pointwise:vanish}, we prove the desired bound on $|\wc_{\al}|$. By the tower property and~\eqref{eq:pointwise:vanish},
\[
	|\widetilde{c}_\alpha| = |\E[\theta_1\,\In_\cE\,(Y-q_0)^\alpha]| \leq \E[\In_{\theta_{V(\alpha)\cup\{1\}}=1}\,\E\nolimits_\theta[\In_\cE\,|Y-q_0|^\alpha]]\,.
\]
Applying~\eqref{eq:pointwise:generic} and using $\P(\theta_{V(\alpha)\cup\{1\}}=1) = \rho^{|V(\alpha)\cup\{1\}|} = \rho^{\In_{1\notin V(\alpha)}}\rho^{|V(\alpha)|}$,
\[
	|\widetilde{c}_\alpha| \leq \rho^{|V(\alpha)\cup\{1\}|}(2q_1)^{|\alpha|} = \rho^{\In_{1\notin V(\alpha)}}\,s_\alpha\,.
\]
If $\alpha$ is dense, applying~\eqref{eq:pointwise:dense} instead gives $|\widetilde{c}_\alpha|\leq \rho^{\In_{1\notin V(\alpha)}}\,d_\alpha$.
 
Next, we bound $|H_{\be\al}|$. Recalling that $H_{\beta\alpha} = \E[\In_\cE\,\In_{\theta_{V(\alpha-\beta)}=1}\,(Y-q_0)^{\alpha-\beta}\mid\theta_{V(\beta)\cup\{1\}}=0]$, we have by tower property
\[
	|H_{\beta\alpha}| \leq \P(\theta_{V(\alpha-\beta)}=1\mid\theta_{V(\beta)\cup\{1\}}=0)\cdot\sup_{\theta:\theta_{V(\alpha-\beta)}=1}\E\nolimits_\theta[\In_\cE\,|Y-q_0|^{\alpha-\beta}]\,.
\]
Since $\beta\in\sC_\alpha$, the sets $V(\alpha-\beta)$ and $V(\beta)\cup\{1\}$ are disjoint, so the conditioning does not affect $(\theta_i)_{i\in V(\alpha-\beta)}$. For a given $\theta$ such that $\theta_{V(\alpha-\beta)}=1$, bounding $\E_\theta[\one_\cE\,|Y-q_0|^{\alpha-\beta}]$ by~\eqref{eq:pointwise:generic} and using $\P(\theta_{V(\alpha-\beta)}=1)=\rho^{|V(\alpha-\beta)|}$, we obtain $|H_{\beta\alpha}|\leq s_{\alpha-\beta}$. If $\alpha-\beta$ is dense, applying~\eqref{eq:pointwise:dense} instead gives $|H_{\beta\alpha}|\leq d_{\alpha-\beta}$. This concludes the proof.
\end{proof}

We next bound $\mathcal{F}(\al)$ by the complexity function $\cH: \sG \rightarrow \R_{\geq 0}$ defined recursively as  
\[
    \cH(\alpha) = 2 + 2\sum_{\beta \in \sC_{\alpha} \setminus \{\alpha\}} \cH(\beta)\,,\qquad \cH(\al)=2~~\textnormal{if $\al$ is connected.}
\]
\begin{lemma}\label{lem:small:hypergraph:bound:F}
Let $\eps>0$ be small enough and $n$ be large enough such that \eqref{eq:event:probability:below:by:1:2} holds. Then for all $\alpha \in \sG$, 
\[
    |\cF(\alpha)| \leq \cH(\alpha) \rho^{\In_{1 \notin V(\alpha)}} s_{\alpha}\,.
\]
Moreover, if $\alpha$ is dense per Definition~\ref{def:sparse:dense:graph}, 
\[
    |\cF(\alpha)| \leq \cH(\alpha) \rho^{\In_{1 \notin V(\alpha)}} d_{\alpha}\,.
\]
\end{lemma}

\begin{proof}
Write
\[
    S_\alpha:=\rho^{\In_{1\notin V(\alpha)}}s_\alpha\,,
    \qquad
    R_\alpha:=\rho^{\In_{1\notin V(\alpha)}}d_\alpha\,.
\]
By \eqref{eq:event:probability:below:by:1:2}, every denominator appearing in
the recursion for $\cF$ is at least $1/2$. 

We first prove the bound with $S_\alpha$. The proof is by induction on
$|\sC_\alpha|$. If $\alpha=0$ or $\alpha$ is connected, then
$\cH(\alpha)=2$, and Lemma~\ref{lem:small:hypergraph:bound:c:A} gives
\[
    |\cF(\alpha)|
    \le
    2|\widetilde c_\alpha|
    \le
    2S_\alpha
    =
    \cH(\alpha)S_\alpha\,.
\]
Now assume the claim holds for all
$\beta\in\sC_\alpha\setminus\{\alpha\}$. For such $\beta$, the definition of
$\sC_\alpha$ implies $1\in V(\be)$ iff $1\in V(\al)$, and the vertex sets of $\beta$ and $\alpha-\beta$ are disjoint. Hence $  s_{\alpha-\beta}S_\beta=S_\alpha $. Using the recursion for $\cF$, Lemma~\ref{lem:small:hypergraph:bound:c:A},
and the induction hypothesis,
\[
\begin{aligned}
    |\cF(\alpha)|
    &\le
    2\bigg(
        |\widetilde c_\alpha|
        +
        \sum_{\beta\in\sC_\alpha\setminus\{\alpha\}}
        |H_{\beta\alpha}|\,|\cF(\beta)|
    \bigg)  \\
    &\le
    2\bigg(
        S_\alpha
        +
        \sum_{\beta\in\sC_\alpha\setminus\{\alpha\}}
        s_{\alpha-\beta}\cH(\beta)S_\beta
    \bigg)  =
    \bigg(
        2+
        2\sum_{\beta\in\sC_\alpha\setminus\{\alpha\}}\cH(\beta)
    \bigg)S_\alpha
    =
    \cH(\alpha)S_\alpha\,.
\end{aligned}
\]
This proves the generic bound.

Now suppose $\alpha$ is dense. We prove the bound with $R_\alpha$ by the
same induction. The base case follows from the dense estimate on
$\widetilde c_\alpha$ from Lemma~\ref{lem:small:hypergraph:bound:c:A}: $ |\cF(\alpha)|
    \le
    2|\widetilde c_\alpha|
    \le
    2R_\alpha
    =
    \cH(\alpha)R_\alpha $.  For the induction step, first note that the previously established generic bound
implies
\[
    |\cF(\beta)|\le \cH(\beta)R_\beta
\]
for every proper $\beta\in\sC_\alpha$: if $\beta$ is dense this is the
induction hypothesis, while if $\beta$ is sparse it follows from
$s_\beta\le d_\beta$ (cf. \eqref{eq:observation:alpha:sparse}). Similarly, Lemma~\ref{lem:small:hypergraph:bound:c:A} implies that $|H_{\beta\alpha}|\le d_{\alpha-\beta}$ for all $\beta\in\sC_\alpha\setminus\{\alpha\}$. Finally, since $\beta$ and $\alpha-\beta$ are disjoint unions of connected
components,
\[
    d_{\alpha-\beta}R_\beta\le R_\alpha\,.
\]
Indeed, recalling the definition of $d_{\al}$ in \eqref{eq:def:s:d}, the vertex and edge counts add, and $ m_{\alpha-\beta}+m_\beta\le m_\alpha$ by subadditivity of the floor function. Also, the definition of $\sC_{\al}$ implies that $1\in V(\be)$ if and only if $1\in V(\al)$. Thus,
\[
\begin{aligned}
    |\cF(\alpha)|
    &\le
    2\bigg(
        |\widetilde c_\alpha|
        +
        \sum_{\beta\in\sC_\alpha\setminus\{\alpha\}}
        |H_{\beta\alpha}|\,|\cF(\beta)|
    \bigg) \\
    &\le
    2\bigg(
        R_\alpha
        +
        \sum_{\beta\in\sC_\alpha\setminus\{\alpha\}}
        d_{\alpha-\beta}\cH(\beta)R_\beta
    \bigg)\le
    \bigg(
        2+
        2\sum_{\beta\in\sC_\alpha\setminus\{\alpha\}}\cH(\beta)
    \bigg)R_\alpha
    =
    \cH(\alpha)R_\alpha\,.
\end{aligned}
\]
This proves the desired bound for dense $\al$.
\end{proof}

\begin{lemma}\label{lem:bound:H}
For any $\alpha \in \sG$ with $w(\alpha)$ connected components, 
\[
    \cH(\alpha) \leq 2\cdot w(\alpha)! \cdot 3^{w(\alpha)}\,.
\]
\end{lemma}

\begin{proof}
Let $(h_w)_{w\ge0}$ be defined by
\[
    h_0=2,
    \qquad
    h_w=2+2\sum_{j=0}^{w-1}\binom{w}{j}h_j
    \quad\text{for }w\ge1.
\]
We first claim that
\[
    \cH(\alpha)\le h_{w(\alpha)}\,.
\]
Indeed, if $\beta$ is a union of connected components of $\alpha$ and
$w(\beta)=j$, then there are at most $\binom{w(\alpha)}{j}$ possible choices
for such $\beta$. The definition of $\cH$ therefore gives, by induction on
$w(\alpha)$,
\[
    \cH(\alpha)
    \le
    2+2\sum_{j=0}^{w(\alpha)-1}\binom{w(\alpha)}{j}h_j
    =
    h_{w(\alpha)}\,.
\]
It remains to show that
\[
    h_w\le 2\cdot w!\cdot 3^w
    \qquad\text{for all }w\ge0\,.
\]
The case $w=0$ is immediate, and $h_1=2+2h_0=6=2\cdot1!\cdot3$.
Assume the claim holds for all $j<w$. Then
\[
\begin{aligned}
    h_w
    \le
    2+2\sum_{j=0}^{w-1}\binom{w}{j}2\cdot j!\cdot 3^j  =
    2+4w!\sum_{j=0}^{w-1}\frac{3^j}{(w-j)!} =
    2+4w!\,3^w\sum_{s=1}^{w}\frac{3^{-s}}{s!}\,.
\end{aligned}
\]
Using $\sum_{s=1}^{\infty}\frac{3^{-s}}{s!}
    =
    e^{1/3}-1
    <
    \frac{5}{12}$, we obtain, for $w\ge2$,
\[
    h_w
    \le
    2+\frac{5}{3}\cdot w!\,3^w
    \le
    2\cdot w!\,3^w\,.
\]
This completes the induction and hence the proof.

\end{proof}
\subsubsection{Bounding the norm $\|u\|$}
Recall that $\|u\|^2\leq \Xi_{\mathsf{sparse}}+\Xi_{\mathsf{dense}}$ from~\eqref{eq:bound:u:in:terms:of:Xi}. To finish the proof of Theorem~\ref{thm:small:planted:hypergraph}-(a), we bound $\Xi_{\mathsf{sparse}}$ and $\Xi_{\mathsf{dense}}$ separately. 

\begin{lemma}\label{lem:small:hypergraph:sparse:case}
There exists constants $\eps=\eps(r,a,b,\xi)>0$ and $\delta=\delta(a,b,\xi)$ such that for all $n$ sufficiently large and $D \leq n^{\eps}$, it holds that $\Xi_{\mathsf{sparse}} \leq C/n$.
\end{lemma}

\begin{proof}
Throughout the proof, $C,C'$ denote constants depending only on
$r,a,b,\xi$, whose value may change from line to line. Let $\alpha$ be
sparse with $k=|V(\alpha)|$ vertices and $\ell=|\alpha|$ edges. Then $k/r\leq \ell \leq m_k$, where we recall $m_k=\lfloor (\xi/a+\delta)k\rfloor$. By Lemma~\ref{lem:bound:H} and Stirling's formula,
\begin{equation}\label{eq:bound:H:stirling}
    \cH(\alpha)
    \le
    2\cdot w(\alpha)!\,3^{w(\alpha)}
    \le
    Ck^{1/2}(Ck)^{k/r}\,,
\end{equation}
where we used $w(\alpha)\le k/r$. Combining with Lemma~\ref{lem:small:hypergraph:bound:F} (recall $s_{\al}=\rho^k (2q_1)^{\ell}$),
\[
    |\cF(\alpha)|
    \le
    Ck^{1/2}(Ck)^{k/r}
    \rho^{\In_{1\notin V(\alpha)}}\rho^k(2q_1)^\ell\,.
\]
Hence,
\begin{equation}\label{eq:bound:K:sparse}
    \cK(\alpha)\equiv \frac{\cF(\alpha)^{2}}{(q_{0}(1 - q_{0}))^{|\alpha|}(1 - \rho)^{|V(\alpha) \cup \{1\}|}}
    \le
    Ck(Ck)^{2k/r}
    \rho^{2k+2\In_{1\notin V(\alpha)}}
    \frac{q_1^{2\ell}}
    {(q_0(1-q_0))^\ell(1-\rho)^{k}}\,,
\end{equation}
where we bounded $(1-\rho)^{|V(\al)\cup\{1\}|}\geq (1-\rho)^k$. Note that the number of sparse $\alpha$ with $1\in V(\alpha)$,
$k$~vertices, and $\ell$ edges is at most 
\begin{equation}\label{eq:number:alpha:planted:subhypergraph}
\binom{n-1}{k-1}\binom{\binom{k}{r}}{\ell}
    \le
    \left(\frac{Cn}{k}\right)^{k-1}
    \left(\frac{Ck^r}{\ell}\right)^\ell\,.
    \end{equation}
Removing the constraint $1\in V(\alpha)$ increases the count by at most a
factor $n$, while the summand gains an extra factor $\rho^2$. Therefore the
total contribution of graphs with these fixed values of $(k,\ell)$ is at
most the rooted contribution multiplied by $ 1+n\rho^2\le 2$ since $\rho=n^{\xi-1}$ and $\xi\le 1/2$. Substituting $q_1=n^{-a}$, $q_0=n^{-b}$, and using also $(1-\rho)^{-k}\le 2^{k}$ and
$(q_0(1-q_0))^{-\ell}\le 2^\ell n^{b\ell}$ for large $n$, we get
\[
\begin{aligned}
    \sum_{\substack{\alpha\textnormal{ sparse}\\ |V(\alpha)|=k,\ |\alpha|=\ell}}
    \cK(\alpha)
    &\le
    \left(\frac{Cn}{k}\right)^{k-1}
    (Ck)^{1+2k/r}
    n^{2(\xi-1)k}
    \left(\frac{Ck^r}{\ell}\right)^\ell
    n^{(b-2a)\ell}\,.
\end{aligned}
\]
Since $\ell\ge k/r$, we have $ (Ck^r/\ell)^\ell
    \le
    (C'k^{r-1})^\ell$. Thus, 
\[
\begin{aligned}
    \sum_{\alpha\textnormal{ sparse}}\cK(\alpha)
    &\le
    \frac{C}{n}\sum_{k=r}^{rD}
    k^2\left(\frac{Cn}{k}\right)^{k}
    (Ck)^{2k/r}
    n^{2(\xi-1)k}
    \sum_{\ell=k/r}^{m_k}
    \left(Ck^{r-1}n^{b-2a}\right)^\ell\,.
\end{aligned}
\]
Under the assumption $b\ge 2a$ (cf. \eqref{eq:assume:small:PDS}), the inner sum is bounded by
\[
    \sum_{\ell=k/r}^{m_k}
    \left(Ck^{r-1}n^{b-2a}\right)^\ell
    \le
    m_k\left(Ck^{r-1}n^{b-2a}\right)^{m_k}\,.
\]
Since $m_k\le (\xi/a+\delta)k$, we obtain for some constant $C_{\delta}=C(\delta, r,a,b,\xi)>0$
\[
\begin{aligned}
    \sum_{\alpha\textnormal{ sparse}}\cK(\alpha)
    &\le
    \frac{C}{n}
    \sum_{k=r}^{rD}
    k^3
    \left(
        Cn^{-1+\frac{b\xi}{a}+\delta(b-2a)}
        k^{C_{\delta}}
    \right)^k\,.
\end{aligned}
\]
Because $a>b\xi$, we can choose $\delta=\delta(a,b,\xi)>0$ small enough that
\[
    -1+\frac{b\xi}{a}+\delta(b-2a)
    \le
    -\frac12\left(1-\frac{b\xi}{a}\right).
\]
Then choose $\eps>0$ small enough so that, whenever $D\le n^\eps$ and
$k\le rD$, $Cn^{-\frac12(1-b\xi/a)}k^{C}\le \frac12$. It follows that
\[
    \sum_{\alpha\textnormal{ sparse}}\cK(\alpha)
    \le
    \frac{C}{n}
    \sum_{k=r}^{\infty}\frac{k^3}{2^k}
    \le
    \frac{C}{n}\,,
\]
which concludes the proof.
\end{proof}

\begin{lemma}\label{lem:small:hypergraph:dense:case}
There exists constants $\eps=\eps(r,a,b,\xi)>0$ and $\delta=\delta(a,b,\xi)$ such that for all $n$ sufficiently large and $D \leq n^{\eps}$, it holds that $\Xi_{\mathsf{dense}} \leq C/n$.
\end{lemma}

\begin{proof}
Throughout, $C,C'>0$ denote constants depending only on $r,a,b,\xi$, whose
value may change from line to line. Let $\alpha$ be dense with
$k=|V(\alpha)|$ vertices and $\ell=|\alpha|$ edges. Then
\[
    \ell\ge m_k+1,
    \qquad
    m_k:=\left\lfloor \left(\frac{\xi}{a}+\delta\right)k\right\rfloor\,.
\]
Recall the bound $\HH(\al)\leq Ck^{1/2}(Ck)^{k/r}$; see~\eqref{eq:bound:H:stirling}, which is a consequence of Lemma~\ref{lem:bound:H}. Combining with Lemma~\ref{lem:small:hypergraph:bound:F} (recall that $d_{\al}=\rho^k D^{m_k}q_0^{\ell}(2q_1/q_0)^{m_k}$),
\[
    |\cF(\alpha)|
    \le
    Ck^{1/2}(Ck)^{k/r}
    \rho^{\In_{1\notin V(\alpha)}}
    \rho^kD^{m_k}q_0^\ell
    \left(\frac{2q_1}{q_0}\right)^{m_k}\,.
\]
Thus, applying this bound to $\cK(\al)$ analogously to \eqref{eq:bound:K:sparse} yields
\[
    \cK(\alpha)\le
    Ck(Ck)^{2k/r}
    \rho^{2k+2\In_{1\notin V(\alpha)}}
    D^{2m_k}
    \left(\frac{q_1}{q_0}\right)^{2m_k}
    \frac{q_0^{2\ell}}
    {(q_0(1-q_0))^\ell(1-\rho)^{k}}\,.
\]
Recall from \eqref{eq:number:alpha:planted:subhypergraph} that the number of dense $\alpha$ with $1\in V(\alpha)$,
$k$ vertices, and $\ell$ edges is at most $(Cn/k)^{k-1} (Ck^r/\ell)^\ell$. Removing the constraint $1\in V(\alpha)$ increases this count by at most a
factor $n$, while the summand gains an extra factor $\rho^2$. Since
$n\rho^2=n^{2\xi-1}\le1$, the total contribution for fixed $(k,\ell)$ is at
most a constant times the rooted contribution. Hence
\[
\begin{aligned}
    \sum_{\substack{\alpha\textnormal{ dense}\\ |V(\alpha)|=k,\ |\alpha|=\ell}}
    \cK(\alpha)
    \le
    \left(\frac{Cn}{k}\right)^{k-1}
    \left(\frac{Ck^r}{\ell}\right)^\ell
    k(Ck)^{2k/r}
    \rho^{2k}D^{2m_k} 
    \left(\frac{q_1}{q_0}\right)^{2m_k}
    \frac{q_0^{2\ell}}
    {(q_0(1-q_0))^\ell(1-\rho)^{k+1}}\,.
\end{aligned}
\]
Substituting $q_0=n^{-b}$, $q_1=n^{-a}$, and $\rho=n^{\xi-1}$, and using
$\ell\geq k/r$, $(1-\rho)^{-k}\le 2^{k}$ and $(1-q_0)^{-\ell}\le 2^\ell$ for large enough $n$, gives
\[
\begin{aligned}
    \sum_{\alpha\textnormal{ dense}}\cK(\alpha)
    \le
    C\sum_{k=r}^{rD}
    \left(\frac{Cn}{k}\right)^{k-1}
    (Ck)^{1+2k/r}
    n^{2(\xi-1)k}
    D^{2m_k}n^{2(b-a)m_k} \sum_{\ell=m_k+1}^{D}
    \left(Ck^{r-1}\right)^\ell n^{-b\ell}\,.
\end{aligned}
\]
Choose $\eps>0$ small enough that $Ck^{r-1}n^{-b}\le \frac12$ whenever $D\leq n^\eps$ and $k\leq r D$. Then, the inner sum is crudely bounded by $2\left(C k^{r-1}n^{-b}\right)^{m_k}$. Hence,
\[
\begin{aligned}
    \sum_{\alpha\textnormal{ dense}}\cK(\alpha)
    &\le
    \frac{C}{n}
    \sum_{k=r}^{rD}
    k^2
    \left(Cn^{2\xi-1}k^{-1+2/r}\right)^k
    \left(CD^2k^{r-1}n^{b-2a}\right)^{m_k}\,.
\end{aligned}
\]
Using $b\geq 2a$ from \eqref{eq:assume:small:PDS}, $m_k\le(\xi/a+\delta)k$ and absorbing powers of $k$ and $D$ into
$k^{C_{\delta}}D^{C_{\delta}}$ for some $C_{\delta}=C(\delta,r,a,b,\xi)>0$ we obtain
\[
    \sum_{\alpha\textnormal{ dense}}\cK(\alpha)
    \le
    \frac{C}{n}
    \sum_{k=r}^{rD}
    k^2
    \left(
        Cn^{-1+\frac{b\xi}{a}+\delta(b-2a)}
        k^{C_\delta}D^{C_\delta}
    \right)^k\,.
\]
Choose $\delta=\delta(a,b,\xi)>0$ so that
\begin{equation}\label{eq:constraint:delta}
    -1+\frac{b\xi}{a}+\delta(b-2a)
    \le
    -\frac12\left(1-\frac{b\xi}{a}\right)\,.
\end{equation}
Then choose $\eps=\eps(r,a,b,\xi)>0$ small enough so that,
whenever $D\le n^\eps$ and $k\le rD$,
\[
    Cn^{-\frac12(1-b\xi/a)}k^{C_\delta}D^{C_\delta}
    \le
    \frac12\,.
\]
It follows that
\[
    \sum_{\alpha\textnormal{ dense}}\cK(\alpha)
    \le
    \frac{C}{n}
    \sum_{k=r}^{\infty}\frac{k^2}{2^k}
    \le
    \frac{C}{n}\,,
\]
which concludes the proof.
\end{proof}

\begin{proof}[Proof of Theorem~\ref{thm:small:planted:hypergraph}-(a)]
If $2a>b$, then Theorem~\ref{thm:planted:hypergraph}-(a) applies to yield $\Corr_{\leq D} \leq Cn^{(\xi-1)/2}$ whenever $D \leq Cn^{(2a-b)/(r-1)}$. Otherwise, assume that $2a \leq b$. Then, choose $\eps, \delta > 0$ such that the conclusions of Lemmas~\ref{lem:small:hypergraph:sparse:case} and~\ref{lem:small:hypergraph:dense:case} hold, and set $\delta_1=\eps$. Our construction of a dual certificate $u$ in~\eqref{eq:small:hypergraph:ansatz} satisfies $\wM^{\top}u=\wc$ by Corollary~\ref{cor:small:hypergraph:recursion} with $\|u\|^2=\Xi_{\mathsf{sparse}}+\Xi_{\mathsf{dense}}$. Thus, by Lemmas~\ref{lem:small:hypergraph:sparse:case} and~\ref{lem:small:hypergraph:dense:case}, with Corollary~\ref{cor:orthogonal:expansion:conditioning}, we have
\[
    \Corr_{\leq D} \leq \frac{\|u\|}{\sqrt{\E[x^{2}]}} + \sqrt{\Pb(\cE^{c} \mid \theta_1=1)} \leq \frac{C}{\sqrt{n\rho}} + n^{-a\delta/4} = Cn^{-\xi/2} + n^{-a\delta/4}\,,
\]
where the bound $\P(\mathcal{E}^c\mid \theta_1=1)\leq n^{-a\delta/2}$ follows from Lemma~\ref{lem:event:bound}. Setting $\delta_2=\min(\xi/2, a\delta/4)$ completes the proof.
\end{proof}

\section{Tensor PCA with a general prior}\label{sec:general:PCA}
In this section we prove Theorem~\ref{thm:general:PCA}. Throughout, we fix a prior distribution $\pi$ satisfying $\E[\pi]=0$, $\E[\pi^2]=1$, and $\E[|\pi|^t]\leq (K_0t)^{\nu t}$ for all $t\geq 1$, where $K_0>0$ and $\nu\geq 0$ are constants depending only on $\pi$. The estimand is $x = \prod_{i=1}^m\theta_i$ for a fixed integer $m\geq 2$. As noted in the case of sparse tensor PCA model, for technical convenience, we shall work with the symmetrized tensor $Y\equiv\Ysy$ and prove statements for the symmetric model. All of our results transfer to the non-symmetric model; see Lemma~\ref{lem:equivalence:of:noise:models} for the correspondence.

However, unlike the sparse tensor PCA and planted dense subhypergraph models, where the orthogonal expansion approach is used, here the \emph{cumulant expansion} by Schramm and Wein~\cite{SW-22} suffices. This approach bounds the low-degree correlation directly in terms of joint cumulants of the signal and observation (this can be viewed as a special case of the orthogonal expansion approach; see Remark~\ref{rem:cumulant}): 
\begin{equation}\label{eq:correlation:cumulant:bound}
	\Corr_{\leq D}^2 \leq \sum_{|\alpha|\leq D}\frac{\kappa_\alpha^2}{\alpha!}\,,
\end{equation}
where the sum ranges over all multi-hypergraphs $\alpha\in\N^{\mc}$ with at most $D$ edges and $\kappa_\alpha$ is the joint cumulant
\begin{equation}\label{eq:cumulant:def}
	\kappa_\alpha = \E[x\,X^\alpha] - \sum_{0\leq\beta\lneq\alpha}\binom{\alpha}{\beta}\E[X^{\alpha-\beta}]\,\kappa_\beta\,.
\end{equation}
\subsection{Reduction to good graphs}
For $\al\in \N^{\mc}$, define
\[
\bar{\alpha} := \alpha + \In_{[m]}
\]
as the graph obtained by adding the $m$-hyperedge $\{1, \hdots, m\}$ to $\alpha$. Note that, unlike $\al$, $\bar{\al}$ is not $r$-uniform unless $m=r$.

\begin{definition}\label{def:general:PCA:good}
A graph $\alpha$ is good if $\alpha = 0$ or $\bar{\alpha}$ is connected and every vertex in $\bar{\alpha}$ has degree at least $2$.
\end{definition}

\begin{lemma}\label{lem:cumulant:vanishing}
If $\alpha$ is not good, then $\kappa_{\alpha} = 0$.
\end{lemma}

\begin{proof}
The case $\alpha = 0$ can be verified directly since $\E[x] = 0$. So assume
$\alpha$ is non-empty. Recall that $\kappa_\alpha$ is the joint cumulant of
the collection $\{x\}\cup \{X_e:\alpha_e\ge 1\}$, where each $X_e$ appears
with multiplicity $\alpha_e$ and $ X_e=\lambda \prod_{i\in e}\theta_i$. If $\bar{\alpha}$ is disconnected, then the collection
$\{x\}\cup \{X_e:\alpha_e\ge 1\}$ splits into two independent nonempty
subcollections. As the joint cumulant of independent blocks is
zero~\cite[Proposition 2.11]{SW-22}, it follows that $\kappa_{\alpha}=0$. Finally, suppose that $\bar{\alpha}$ has a vertex $i$ of degree-$1$. In the
combinatorial formula for cumulants~\cite[Definition 2.10]{SW-22}, fix any partition $\pi$ of the cumulant variables and let $B\in b(\pi)$ be the unique
block containing the variable in which $\theta_i$ appears. The expectation of this block $\E[\prod_{j\in B}X_j]$ equals to zero, because
$\theta_i$ has mean zero and is independent of all other factors in that
block. Hence each term in the cumulant expansion vanishes, and
$\kappa_\alpha=0$.
\end{proof}
\subsection{Bounding the cumulants}
We introduce three quantities necessary to bound $|\kappa_\alpha|$ for good graphs. For $\al\in \N^{\mc}$, define the number of excess edges
\[
    \delta(\alpha) = \sum_{\substack{i \in V(\alpha) \\ \deg_{\alpha}(i) \geq 3}} \deg_{\alpha} (i)\,,
\]
and define the complexity function $\HH(\al)$ via the recursion
\begin{equation}\label{eq:def:H}
    \HH(\alpha) = \sum_{\substack{\beta \lneq \alpha \\ \text{$\beta$ is good}}} \binom{\alpha}{\beta} \cH(\beta)\,,\qquad \cH(0)=1\,.
\end{equation}
Finally, for an integer $t\geq 1$, define
\[
    M(t) = \max_{0 \leq s \leq t} \E[|\pi|^{s}]\,.
\]
Observe that $M(\cdot)$ satisfies the sub-multiplicative property $M(s)M(t) \leq M(s + t)$ by H\"{o}lder's inequality. Furthermore, the moment condition on the prior $\pi$ implies that $M(t) \leq (Kt)^{\nu t}$ for some constant $K > 0$.
The following lemma is a generalization of $r=2$ case in~\cite[Lemma 5.3]{arxiv-version}.

\begin{lemma}\label{lem:PCA:cumulant:bound}
For any good $\alpha$, 
\begin{equation}
    |\kappa_{\alpha}| \leq \lambda^{|\alpha|}M(\delta(\bar{\alpha}))\HH(\alpha)\,.
\label{eq:cumulant:recursion:bound}
\end{equation}
\end{lemma}
\begin{proof}
We proceed by induction on the number of good subgraphs of $\alpha$. If $\alpha=0$, then $\kappa_0=0$ and the bound holds trivially; if $\alpha$ is non-empty and has no proper good subgraph, then $\cH(\alpha)=\cH(0)=1$. Thus it suffices to show $|\kappa_\alpha|\leq\lambda^{|\alpha|}M(\delta(\bar\alpha))$. Since $\alpha$ has no proper good subgraph, the cumulant recursion~\eqref{eq:cumulant:def} gives $|\kappa_\alpha|\leq\E[|x\,X^\alpha|]$. Using $X_e = \lambda\prod_{i\in e}\theta_i$ and the independence of the $\theta_i$'s, we have $\E[|x\,X^\alpha|] = \lambda^{|\alpha|}\prod_{i\in V(\bar\alpha)}\E[|\pi|^{\deg_{\bar\alpha}i}]$. Since $\alpha$ is good, every vertex of $\bar\alpha$ has degree at least~$2$. The vertices with $\deg_{\bar\alpha}i=2$ contribute $\E[\pi^2]=1$ each. The remaining vertices contribute
\[
	\prod_{\substack{i\in V(\bar\alpha)\\\deg_{\bar\alpha}i\geq 3}}\E[|\pi|^{\deg_{\bar\alpha}i}] \leq M\bigg(\sum_{\substack{i\in V(\bar\alpha)\\\deg_{\bar\alpha}i\geq 3}}\deg_{\bar\alpha}i\bigg) = M(\delta(\bar\alpha))\,,
\]
where the inequality uses the sub-multiplicativity of $M$. This finishes the proof of the base-case.
 
For the inductive step, assume~\eqref{eq:cumulant:recursion:bound} holds for all good $\beta\lneq\alpha$. By a triangle inequality, we have
\begin{equation}\label{eq:cumulant:induction}
	|\kappa_\alpha| \leq \E[|x\,X^\alpha|] + \sum_{\substack{0\lneq\beta\lneq\alpha\\\text{$\beta$ good}}}\binom{\alpha}{\beta}\,\E[|X^{\alpha-\beta}|]\,|\kappa_\beta|\,,
\end{equation}
where we used $\kappa_{\be}=0$ unless $\be$ is good (cf. Lemma~\ref{lem:cumulant:vanishing}). The first term is bounded by $\lambda^{|\alpha|}M(\delta(\bar\alpha))$ as in the base case. For each good $\beta\lneq\alpha$ in the sum, we have
\[
	\E[|X^{\alpha-\beta}|] = \lambda^{|\alpha-\beta|}\prod_{i\in V(\alpha-\beta)}\E[|\pi|^{\deg_{\alpha-\beta}i}] \leq \lambda^{|\alpha-\beta|}\,M(\delta(\alpha-\beta))\,,
\]
where vertices with degree~$1$ contribute $\E[|\pi|]\leq \sqrt{\E[\pi^2]}=1$ and those with degree~$2$ contribute $\E[\pi^2]=1$. By the induction hypothesis, $|\kappa_\beta|\leq\lambda^{|\beta|}M(\delta(\bar\beta))\,\HH(\beta)$. Combining,
\[
	\E[|X^{\alpha-\beta}|]\,|\kappa_\beta| \leq \lambda^{|\alpha|}\,M(\delta(\alpha-\beta))\,M(\delta(\bar\beta))\HH(\beta)\,.
\]
Note that $\delta(\alpha-\beta)+\delta(\bar\beta)\leq\delta(\bar\alpha)$ since for every $i\in V(\bar\alpha)$, we have $\deg_{\bar\alpha}(i) = \deg_{\alpha-\beta}(i) + \deg_{\bar\beta}(i)$. Combining with the sub-multiplicativity of $M(\cdot)$, it follows that $M(\delta(\alpha-\beta))M(\delta(\bar\beta))\leq M(\delta(\bar\alpha))$. Substituting into~\eqref{eq:cumulant:induction} and factoring out $\lambda^{|\alpha|}M(\delta(\bar\alpha))$,
\[
	|\kappa_\alpha| \leq \lambda^{|\alpha|}\,M(\delta(\bar\alpha))\bigg(1+\sum_{\substack{0\lneq\beta\lneq\alpha\\\text{$\beta$ good}}}\binom{\alpha}{\beta}\HH(\beta)\bigg) = \lambda^{|\alpha|}M(\delta(\bar\alpha))\HH(\alpha)\,,
\]
which concludes the proof.
\end{proof}

\begin{proposition}\label{prop:general:PCA:bound:H}
For any good $\alpha$ such that $\alpha\neq 0$, 
\begin{equation}
    \HH(\alpha) \leq (2|\alpha|)^{r|\alpha| - 2|V(\alpha)| + m}\,.
\label{eq:general:PCA:bound:H}
\end{equation}
\end{proposition}
The proof of Proposition~\ref{prop:general:PCA:bound:H} parallels the $r=2$ case established in~\cite[Lemma~5.4]{arxiv-version}. The main new ingredient for $r\geq 3$ is the following lemma, which controls how the quantity $r|\al|-2|V(\al)|+m$ decreases when an edge is removed.
\begin{lemma}\label{lem:good:peeling}
Let $\alpha$ be non-zero and good. Fix an edge $e\in E(\al)$, and let
$\alpha_\star(e)$ be the maximal good subgraph of $\alpha-\mathbf 1_e$,
obtained as the union of all good subgraphs of $\alpha-\mathbf 1_e$.
If $\alpha_\star(e)\neq 0$, 
\[
    r|\alpha_\star(e)|-2|V(\alpha_\star(e))|+m
    \le
    r|\alpha|-2|V(\alpha)|+m-1\,.
\]
\end{lemma}
\begin{proof}
Write $\beta=\alpha_\star(e)$. Note that since $\be\leq \al-\In_e$, we have $\al-\be\neq 0$. Also, since $\beta$ is non-zero
and good, every vertex in $[m]$ belongs to $V(\beta)$: otherwise that
vertex would have degree exactly one in $\bar\beta$, coming from
the added hyperedge $\mathbf 1_{[m]}$, contradicting goodness. Similarly,
$[m]\subseteq V(\alpha)$. Let
\[
    U:=V(\alpha)\setminus V(\beta)=V(\al-\be)\setminus V(\be)\,.
\]
Since $U\cap [m]=\emptyset$, the degree of any $u\in U$ in
$\bar\alpha$ equals its degree in $\alpha$, and this degree is at least
$2$ because $\alpha$ is good. Also since $u\notin V(\beta)$, all edges of
$\alpha$ incident to $u$ belong to $\al-\be$. Hence $\deg_{\al-\be}u\geq 2$ for all $u\in U$. Let $b$ denote the number of edges $e\in E(\al-\be)$ such that one of its vertices $v\in V(e)$ is contained in $V(\be)$. Then
\[
    r|\al-\be|
    \geq
    b+\sum_{u\in U}\deg_{\al-\be}u
    \ge
    b+2|U|=2(|V(\al)-V(\be)|)+b\,.
\]
It remains to show that $b\ge 1$. If $b=0$, then $V(\al-\be)\cap V(\be)=\emptyset$. Since $[m]\subseteq V(\be)$ and $\al-\be\neq0$, this implies that the graph $\bar\alpha$ has at least one nonempty
component supported outside $V(\beta)$ and another component containing
$\bar\beta$, contradicting the connectedness of $\bar\alpha$. Therefore $b\ge 1$, concluding the proof.
\end{proof}
\begin{proof}[Proof of Proposition~\ref{prop:general:PCA:bound:H}]
For a good graph $\be$, write
\[
    p(\be):=r|\be|-2|V(\be)|+m\,.
\]
If $\be$ is non-zero and good, then $p(\be)\ge 0$, since every vertex
of $\bar\be$ has degree at least $2$:
\begin{equation}\label{eq:good:edges:vertices}
    r|\be|
    =
    \sum_{i\in V(\be)}\deg_\be i
    =
    -m+\sum_{i\in V(\be)}\deg_{\bar\be} i
    \ge
    -m+2|V(\be)|\,.
\end{equation}
We prove the claim by induction on the number of good subgraphs of
$\alpha$. If $\alpha$ is non-zero and has
no non-zero proper good subgraph, then $\cH(\al)=\cH(0)=1$, thus the base-case holds. For the induction step, suppose $\alpha$ has at least one non-zero proper good subgraph, and
assume the bound holds for all proper good subgraphs of $\alpha$. Note that using the recursive definition of $\cH$ (cf.~\eqref{eq:def:H}), we have
\[
\begin{split}
    \HH(\alpha)=
    1+ \sum_{\substack{0\lneq \be\lneq \al\\\textnormal{$\be$ is good}}}\binom{\al}{\be}\HH(\be)\leq 1+\sum_{e\in E(\al)}\,\sum_{\substack{0\lneq \be\leq \al-\In_e\\\textnormal{$\be$ is good}}}\binom{\al-\In_e}{\be}\HH(\be)\,,
\end{split}
\]
where the final inequality holds by viewing $\sum_{\be}\binom{\al}{\be}$ as sum over subsets of the edge (multi-)set $E(\al)$. For each $e\in E(\al)$, write $\al_\star(e)$ to be the maximal good subgraph of $\al-\In_e$. If $\al_\star(e)=0$, the corresponding inner sum is zero. If $\al_\star(e)\neq 0$, then every good
$\beta\le \alpha-\mathbf 1_e$ satisfies $\beta\le \al_\star(e)$. Moreover, $\binom{\al-\In_e}{\be}=\binom{\al_\star(e)}{\be}$ holds since $\al-\In_e$ and $\al_\star(e)$ can only differ on hyperedges $e'$ where $(\al_\star(e))_{e'}=0$ by the maximality of $\al_\star(e)$. As a result,
\begin{equation}\label{eq:bound:HH:intermediate}
\HH(\al)\leq 1+\sum_{\substack{e\in E(\al)\\\al_\star(e)\neq 0}}\sum_{0\lneq \be \leq \al_\star(e)} \binom{\al_\star(e)}{\be}\HH(\be)=1+\sum_{\substack{e\in E(\al)\\\al_\star(e)\neq 0}}(2\HH(\al_\star(e))-1)\,,
\end{equation}
where the last identity uses $2\HH(\al_\star(e))=\sum_{\be\leq \al_\star(e)}\binom{\al_\star(e)}{\be}\HH(\be)$ from the definition of $\HH(\cdot)$ (cf.~\eqref{eq:def:H}). By Lemma~\ref{lem:good:peeling}, for all $e\in E(\al)$ with $\al_\star(e)\neq 0$, we have $p(\al_\star(e))\leq p(\al)-1$. Thus, applying inductive hypothesis to $\al_\star(e)$ yields
\[
\HH(\al_\star(e))\leq (2|\al_\star(e)|)^{p(\al_\star(e))}\leq (2|\al|)^{p(\al)-1}\,.
\]
Applying this bound to the RHS of \eqref{eq:bound:HH:intermediate} gives
\[
\HH(\al)\leq 2\sum_{\substack{e\in E(\al)\\\al_\star(e)\neq 0}} \HH(\al_\star(e))\leq 2|\al|\cdot (2|\al|)^{p(\al)-1}=(2|\al|)^{p(\al)}\,,
\]
which concludes the proof.
\end{proof}
\begin{lemma}\label{lem:count:good:general:PCA}
For any $k,\ell\geq 1$, such that $m\leq k\leq r\ell$, the number of good graphs $\alpha\in\N^{\mc}$ with
$|V(\alpha)|=k$ and $|\alpha|=\ell$ is at most
\[
    \left(\frac{Cn}{k}\right)^{k-m}(Ck^{r-1})^\ell\,,
\]
where $C \equiv C(r,m)>0$.
\end{lemma}

\begin{proof}
Since $\alpha$ is good and nonzero, $[m]\subseteq V(\alpha)$. Thus the number of ways to choose the vertex set is at most 
\[
    \binom{n-m}{k-m}
    \le
    \left(\frac{e(n-m)}{k-m}\right)^{k-m}
    \le
    \left(\frac{Cn}{k}\right)^{k-m}\,,
\]
where the first bound is Stirling's approximation (when $k=m$, $((n-m)/(k-m))^{k-m}$ is understood as $1$), and in the last step we used that $m$ is fixed, so $k-m$ is comparable to
$k$ after increasing $C \equiv C(r,m)$ to handle the finitely many cases $m\le k<2m$. On a fixed set of $k$ vertices, the number of $r$-uniform multi-edge types is $ N_k=\binom{k+r-1}{r}\le C_r k^r $. Hence, by a stars and bars counting, the number of multi-hypergraphs with $\ell$ edges
on this vertex set is at most $\binom{N_k+\ell-1}{\ell}
    \le
    \left(\frac{Ck^r}{\ell}+C\right)^\ell$, increasing $C$ if needed.
Since goodness implies $r\ell+m\ge 2k$ (cf. \eqref{eq:good:edges:vertices}), and $m$ is fixed, we have
$\ell\ge c k$ for some $c \equiv c(r,m)>0$ after adjusting constants for finitely many small $k$.
Therefore, we conclude that $\left(\frac{Ck^r}{\ell}+C\right)^\ell
    \le
    (C'k^{r-1})^\ell$ for some $C' \equiv C'(r,m)$, as stated.
\end{proof}

\subsection{Proof of Theorem~\ref{thm:general:PCA}}
\begin{proof}[Proof of Theorem~\ref{thm:general:PCA}]
Throughout, we let $C,C'$ denote constants depending only on the parameters $r, m,K,\nu$ which may change from line to line. Applying the bound on $|\kappa_{\al}|$ from Lemma~\ref{lem:PCA:cumulant:bound} to the cumulant bound~\eqref{eq:correlation:cumulant:bound} yields
\begin{equation}
    \Corr_{\leq D}^{2} \leq \sum_{\substack{|\alpha| \leq D \\ \text{$\alpha$ is good}}} \frac{\kappa_{\alpha}^{2}}{\alpha!} \leq \sum_{\substack{1\leq |\alpha| \leq D \\ \text{$\alpha$ is good}}} \lambda^{2|\alpha|}M(\delta(\bar{\alpha}))^{2}\HH(\alpha)^{2}\,,
\label{eq:general:PCA:lower:bound:intermediate:step}
\end{equation}
where we used $\ka_0=\E[x]=0$. Note that for any good $\al$, we can bound $\delta(\bar\alpha)$ as
\[
    \delta(\bar{\alpha}) = \sum_{\substack{i \in V(\alpha) \\ \deg_{\bar{\alpha}} i \geq 3}} \deg_{\bar{\alpha}} i \leq \sum_{\substack{i \in V(\alpha) \\ \deg_{\bar{\alpha}} i \geq 3}} 3(\deg_{\bar{\alpha}} i - 2) = \sum_{i \in V(\alpha)} 3(\deg_{\bar{\alpha}} i - 2) = 3(r|\alpha| - 2|V(\alpha)| + m)\,,
\]
where the second equality holds since every vertex in $\bar\al$ has degree at least $2$ by Definition~\ref{def:general:PCA:good}.
Consider any good $\alpha$ with $k$ vertices and $\ell$ edges. For such $\al$, the preceding shows $r\ell \geq 2k-m$, and
\[
    M(\delta(\bar{\alpha})) \leq M(3(r\ell - 2k + m)) \leq (C(r\ell - 2k + m))^{3\nu(r\ell - 2k + m)} \leq (CD)^{3\nu(r\ell - 2k + m)}\,,
\]
where the second inequality holds by our assumption
$\E[|\pi|^{t}]\leq (Kt)^{\nu t}$. Moreover, by Proposition~\ref{prop:general:PCA:bound:H}
\[
    \HH(\alpha) \leq (CD)^{r\ell - 2k + m}\,.
\]
Collecting these estimates together and using Lemma~\ref{lem:count:good:general:PCA} to bound the number of good graphs with $k$ vertices and $\ell$ edges, we have
\[
\begin{aligned}
    \Corr_{\leq D}^2& \leq\sum_{k = m}^{rD} \sum_{\ell \geq \frac{1}{r}(2k - m)} \lPa \frac{Cn}{k} \rPa^{k - m}(Ck^{r - 1})^{\ell} \lambda^{2\ell}(CD)^{2(r\ell - 2k + m)}(CD)^{6\nu(r\ell - 2k + m)} \\
    &= \sum_{k = m}^{rD} \lPa \frac{Cn}{k} \rPa^{k - m} (Ck^{r - 1})^{\frac{2k-m}{r}}\lambda^{\frac{2(2k - m)}{r}} \sum_{\Delta \geq 0} \left(Ck^{r-1} \lambda^2(CD)^{(6\nu + 2)r}\right)^{\frac{\Delta}{r}}\,,
\end{aligned}
\]
where we made the reparameterization $\Delta = r\ell - 2k + m$ (i.e. $\ell=\frac{\Delta+2k-m}{r}$) in the second line. Since $k\leq rD$, the summation over $\Delta$ is at most $\sum_{\Delta \geq 0} \left(C\lambda^2D^{(6\nu + 3)r}\right)^{\frac{\Delta}{r}}$. By assumption, we have $\la D^{(6\nu+3)r/2}\leq 1/C'$, so taking $C' \equiv C'(r,m,K,\nu)$ sufficiently large, this sum is at most $2$. Therefore,
\[
\begin{aligned}
    \Corr_{\leq D}^2 \leq 2\sum_{k = m}^{rD} \lPa \frac{Cn}{k} \rPa^{k - m} (Ck^{r - 1})^{\frac{2k-m}{r}}\lambda^{\frac{2(2k - m)}{r}}= 2\lPa \frac{1}{Cn} \rPa^{\frac{m}{2}} \sum_{k = m}^{rD} k^{\frac{m}{2}} \lPa C^{\frac{r+2}{r}} n \lambda^{\frac{4}{r}}k^{\frac{r - 2}{r}} \rPa^{k - \frac{m}{2}}\,.
\end{aligned}
\]
By assumption, we have $\lambda n^{r/4}D^{(r-2)/4}\leq 1/C'$, so by choosing $C'$ large enough, $C^{\frac{r+2}{r}} n \lambda^{\frac{4}{r}}k^{\frac{r - 2}{r}}\leq 1/2$ for all $k\leq rD$. Thus,
\[
    \Corr^2_{\leq D}\leq  2\lPa \frac{1}{Cn} \rPa^{\frac{m}{2}}\sum_{k\geq 1} k^{\frac{m}{2}}\left(\frac{1}{2}\right)^{k-\frac{m}{2}}=C'n^{-\frac{m}{2}}\,,
\]
where $C'>0$ only depends on $r,m,K,\nu$. This concludes the proof.
\end{proof}

\section{Estimation upper bounds}\label{sec:upper:bounds}
\subsection{Sparse tensor PCA}\label{subsec:sparse:PCA:upper:bound}
This section is devoted to the proof of Theorem~\ref{thm:sparse:PCA}-(b). As in the lower bound, we work with a noise-reduced symmetric tensor $Z=(Z_e)$ whose entries are independent, up to symmetry, with distribution $Z_e \sim \cN(0, 1)$ for $e \in \mc$. Our estimator depends only on entries with distinct indices, and thus any results proved for this estimator under the noise-reduced model transfer directly to the noise-inflated model.

\subsubsection{Constructing the estimator}\label{subsec:constructing:tree:estimator}
Our candidate estimator is based on a weighted sum of the number of trees of a prescribed size and structure. Throughout, we assume
\begin{equation}\label{eq:thm:PCA:b:assumptions}
    \frac{en^{r - 1} \rho^{2r - 2} \lambda^{2}}{(r - 2)!} \geq 1 + \eps, \quad \rho = \omega(n^{-1} \log^{6 + 3/(r-1)}{n}), \quad \rho = o(\log^{-6(r-1) - 3}{n})
\end{equation}
and fix this $\eps>0$ throughout this section. We also let
\begin{equation}\label{eq:choice:k:ell}
    \ell = \left\lceil \frac{4}{\eps}\log(1/\rho)\right\rceil\,,\qquad k= (r-1)\ell+1\,,
\end{equation}
and let $\sT \equiv \sT_{\ell}$ (see Definition~\ref{def:special:trees}). Recall that any $\alpha \in \sT$ is a tree, rooted at vertex $1$, consisting of two disjoint rooted trees whose roots are attached to the distinguished
vertex $1$ by two root-incident edges. Note that every such $\alpha$ has exactly $2\ell + 2$ edges and $(r-1)(2\ell+2)+1=2k + 2r - 3$ vertices.

Write $C \equiv C(r,\eps)$ to denote a constant depending only on $r$ and $\eps$ which may vary from line to line. We prove Theorem~\ref{thm:sparse:PCA}-(b) by showing that the polynomial $f$ defined by
\[
    f(Y) = \sum_{\alpha \in \sT} Y^{\alpha}
\]
achieves $1-o(1)$ correlation. As a first step, compute $\E[Y^{\al}x]$  by first integrating out the independent Gaussian noise: 
\[
    \E[Y^{\alpha}x] = \E[(X + Z)^{\alpha} \theta_{1}] = \E[X^{\alpha} \theta_{1}] = \lambda^{|\alpha|} \E \lBr \theta_{1} \prod_{i \in V(\alpha)} \theta_{i}^{\deg_{i}} \rBr = \lambda^{|\alpha|} \prod_{i \in V(\alpha)} \E[\theta_{i}] = \lambda^{|\alpha|} \rho^{|V(\alpha)|}\,,
\]
where we used that the entries of $\theta$ take values in $\{0, 1\}$. Thus, 
\begin{equation}
    \E[xf(Y)] = \sum_{\alpha \in \sT} \E[Y^{\alpha}x] = |\sT| \rho^{2k + 2r - 3} \lambda^{2\ell + 2}\,.
\label{eq:sparse:PCA:upper:bound:numerator}
\end{equation}
It remains to show the second moment is of comparable order to~\eqref{eq:sparse:PCA:upper:bound:numerator}. Expand into three cases: 
\begin{equation}
    \E[f(Y)^{2}] = \sum_{\alpha \in \sT} \E[Y^{2\alpha}] + \sum_{\substack{\alpha, \beta \in \sT \\ \alpha \cap \beta = 0}} \E[Y^{\alpha + \beta}] + \sum_{\substack{\alpha, \beta \in \sT \\ \alpha \cap \beta \neq 0, \alpha \neq \beta}} \E[Y^{\alpha + \beta}]\,.
\label{eq:sparse:PCA:upper:bound:second:moment}
\end{equation}
We prove Theorem~\ref{thm:sparse:PCA}-(b) by showing that this second is asymptotic to $\Gamma$, where
\begin{equation}\label{eq:def:Gamma}
\Gamma := \frac{\E[xf(Y)]^{2}}{\E[x^{2}]} = |\sT|^{2} \rho^{4k + 4r - 7} \lambda^{4\ell + 4}\,.
\end{equation}
Towards this end, we first estimate the size of the set $\sT$. To form a tree $\alpha \in \sT$, we can first choose two disjoint vertex sets of size $k$ from $\{2,\hdots,n\}$ and form a rooted tree on each. The number of ways to do this is exactly (see Lemma~\ref{lem:rooted:forests})
\[
    \binom{n - 1}{k} \binom{n - k - 1}{k}R_{k, 1}^{2} = \frac{R_{k, 1}^{2}}{k!^{2}} \frac{(n - 1)!}{(n - 2k - 1)!}\,.
\]
Next, the $2(r-2)$ non-root leaf vertices in the two root-incident edges are chosen, giving
\[
    \frac{1}{2}\binom{n - 1 - 2k}{r - 2} \binom{n - 2k - r + 1}{r - 2} = \frac{1}{2(r - 2)!^{2}} \frac{(n - 2k - 1)!}{(n - 2k - 2r + 3)!}
\]
possibilities. Combining these two counts, we obtain $|\sT| = \frac{1}{2(r - 2)!^{2}} \frac{(n - 1)!}{(n - 2k - 2r + 3)!} \lPa \frac{R_{k, 1}}{k!} \rPa^{2}$. From the assumptions~\eqref{eq:thm:PCA:b:assumptions}, we have $\rho = \widetilde{\omega}(1/n)$ which implies $k=O_\eps(\log{n})$ and consequently $\frac{(n - 1)!}{(n - 2k - 2r + 3)!}=(1-o(1))n^{2k + 2r - 4}$. With $2k + 2r - 4 = (r - 1)(2\ell + 2)$, the preceding gives the asymptotic formula
\begin{equation}
    |\sT| = (1 - o(1)) \frac{n^{(r - 1)(2\ell + 2)}}{2(r - 2)!^{2}} \lPa \frac{R_{k, 1}}{k!} \rPa^{2}\,.
\label{eq:asymptotic:tree:counting}
\end{equation}
To bound the second moment, we show that
\begin{equation}
    \sum_{\substack{\alpha, \beta \in \sT \\ \alpha \cap \beta = 0}} \E[Y^{\alpha + \beta}] = (1+o(1)) \Gamma\,.
\label{eq:sparse:PCA:disjoint:trees}
\end{equation}
For any $\alpha, \beta \in \sT$ with $\alpha \cap \beta = \emptyset$ and $v = |V(\alpha) \cap V(\beta) \setminus \{1\}|$ shared non-root vertices, 
\[
    \E[Y^{\alpha + \beta}] = \E[Y^{\alpha \triangle \beta}] = \lambda^{|\alpha \triangle \beta|} \rho^{|V(\alpha \triangle \beta)|} = \lambda^{4\ell + 4} \rho^{4k + 4r - v - 7}\,.
\]
To count the total number of such pairs, we may first choose a set of $v$ shared vertices; then choose two disjoint vertex sets from the remaining vertices of $\alpha$ and $\beta$ are drawn. The number of choices is at most, 
\[
\begin{aligned}
    \binom{n - 1}{v} \binom{n - 1 - v}{2k + 2r - 4 - v} \binom{n - 2k - 2r + 3}{2k + 2r - 4 - v} &= \frac{1}{v!(2k + 2r - 4 - v)!^{2}} \frac{(n - 1)!}{(n - 4k - 4r + 7 + v)!} \\
    & \leq \frac{n^{(r - 1)(4\ell + 4) - v}}{v!(2k + 2r - 4 - v)!^{2}}\,.
\end{aligned}
\]
Once the vertex sets have been selected, we partition the vertices into two disjoint subsets, each of size $k+r-2$, corresponding to its two rooted subtrees together with its root-incident edge. Within each subset, we choose $r-2$ leaf vertices to be incident to the root, and form a tree on the remaining $k$ vertices. The number of ways to do this is
\[
    \lPa \frac{1}{2} \binom{2k + 2r - 4}{k + r - 2} \binom{k + r - 2}{k}^{2}R_{k, 1}^{2} \rPa^{2} = \lPa \frac{(2k + 2r - 4)!}{2(r - 2)!^{2}} \lPa \frac{R_{k, 1}}{k!} \rPa^{2} \rPa^{2}\,.
\]
Combining this count with the calculation for $\E[Y^{\alpha + \beta}]$ gives
\[
\begin{aligned}
    \sum_{\substack{\alpha, \beta \in \sT \\ \alpha \cap \beta = 0}} \E[Y^{\alpha + \beta}] & \leq \sum_{v \geq 0} \frac{\lambda^{4\ell + 4} \rho^{4k + 4r - v - 7}n^{(r - 1)(4\ell + 4) - v}}{v!(2k + 2r - 4 - v)!^{2}} \lPa \frac{(2k + 2r - 4)!}{2(r - 2)!^{2}} \lPa \frac{R_{k, 1}}{k!} \rPa^{2} \rPa^{2} \\
    &= \lambda^{4\ell + 4} \rho^{4k + 4r - 7} \lPa \frac{n^{(r - 1)(2\ell + 2)}}{2(r - 2)!^{2}} \lPa \frac{R_{k, 1}}{k!} \rPa^{2} \rPa^{2} \sum_{v \geq 0} \frac{(n \rho)^{-v}}{v!} \lPa \frac{(2k + 2r - 4)!}{(2k + 2r - 4 - v)!} \rPa^{2}\,.
\end{aligned}
\]
The prefactor outside the summation is $(1+o(1)) \Gamma$. Using $\frac{(2k + 2r - 4)!}{(2k + 2r - 4 - v)!} \leq (2k + 2r - 4)^{v} \leq (Ck)^{v}$, the last summation on the RHS is bounded by, 
\[
    \sum_{v \geq 0} \frac{1}{v!} \lPa \frac{Ck^{2}}{n \rho} \rPa^{v} \leq \exp \lPa \frac{Ck^{2}}{n\rho} \rPa = 1 + o(1)\,,
\]
where the last step uses $k^{2} = o(n\rho)$ when $\rho = \omega(n^{-1} \log^{6+3/(r-1)}{n})$. The preceding implies~\eqref{eq:sparse:PCA:disjoint:trees}, as claimed. Next we show that, 
\begin{equation}
    \sum_{\alpha \in \sT} \E[Y^{2\alpha}] = o(\Gamma)\,.
\label{eq:sparse:PCA:identical:copies}
\end{equation}
For the purpose of reusing calculations for the planted dense subhypergraph model, define 
\begin{equation}\label{eq:def:eta:zeta}
\eta := \lambda^{2}\,,\qquad \zeta := \rho^{r-1}\eta\,.
\end{equation}
We will require an estimate on the second moment of a single tree: 
\begin{lemma}\label{lem:single:tree:second:moment}
For any $\alpha \in \sT$, 
\begin{equation*}
    \E[Y^{2\alpha}] \leq (1 + \zeta)^{|\alpha|}
\end{equation*}
\end{lemma}

\begin{proof}
Recall that $X=\lambda \theta^{\otimes r}$ is the signal tensor. Take expectation with respect to the mean-zero independent Gaussian noise: 
\[
    \E[Y^{2\alpha}] = \E \lBr \prod_{e \in \alpha}(X_e + Z_e)^2 \rBr = \E \lBr \prod_{e \in \alpha}(X_e^2 + 1) \rBr = \E \lBr \prod_{e \in \alpha}(\eta\theta^{V(e)} + 1) \rBr\,.
\]
Enumerate the edges of $\alpha$ as $\{e_j\}_{j=1}^{|\alpha|}$ in rooted exploration order; in other words, choose $e_{1}$ as one of the root-incident edges, and for each $j \ge 2$, choose a new edge $e_j$ which shares exactly one vertex with the previously explored edges $\cup_{m < j} e_m$. Let $A_1$ be the set of non-root vertices in $e_1$, and for each $j \ge 2$, let $A_j$ be the set of new vertices introduced by $e_j$. Each $A_j$ has cardinality $r - 1$, the sets $A_j$ are pairwise disjoint, and $A_j \subseteq V(e_j)$. Because $\theta$ is binary-valued, we have $\theta^{V(e_j)} \le \theta^{A_j}$ for all $j = 1,\hdots,|\alpha|$. Hence, 
\[
    \E \lBr \prod_{e \in \alpha}(\eta\theta^{V(e)} + 1) \rBr = \E \lBr \prod_{j = 1}^{|\alpha|}(\eta\theta^{V(e_{j})} + 1) \rBr \leq \E \lBr \prod_{j = 1}^{|\alpha|}(\eta\theta^{A_{j}} + 1) \rBr = (1 + \zeta)^{|\alpha|}\,,
\]
where the last equality holds by independence since the sets $\{A_{j}\}_{j = 1}^{|\alpha|}$ are pairwise disjoint. This proves the claim.
\end{proof}

Using Lemma~\ref{lem:rooted:forests} and Stirling's approximation $\ell! \leq e\sqrt{2\pi \ell}(\ell/e)^{\ell}$, we have the lower bound
\[
    \frac{R_{k, 1}}{k!} = \frac{k^{\ell - 1}}{\ell!(r - 1)!^{\ell}} \geq \frac{C}{k\sqrt{\ell}} \lPa \frac{e}{(r - 2)!} \rPa^{\ell} \lPa \frac{k}{(r - 1)\ell} \rPa^{\ell} \geq \frac{C}{\ell^{\frac{3}{2}}} \lPa \frac{e}{(r - 2)!} \rPa^{\ell - 1}\,,
\]
where the last step used the bound $\frac{k}{(r - 1)\ell} \geq 1$. Combining with the asymptotic formula~\eqref{eq:asymptotic:tree:counting} gives
\begin{equation}
    |\sT| \geq C\ell^{-3} \lPa \frac{en^{r - 1}}{(r - 2)!} \rPa^{2\ell + 2}\,.
\label{eq:asymptotic:tree:lower:bound}
\end{equation}
With this estimate in hand, Lemma~\ref{lem:single:tree:second:moment} then implies
\[
    \frac{1}{\Gamma} \sum_{\alpha \in \sT} \E[Y^{2\alpha}] \leq \frac{1}{\Gamma} \sum_{\alpha \in \sT} (1 + \zeta)^{|\alpha|} = \frac{(1 + \zeta)^{2\ell + 2}}{|\sT| \rho^{4k + 4r - 7} \lambda^{4\ell + 4}} \leq C\rho^{-1} \ell^{3} \lPa \frac{(1 + \zeta)(r - 2)!}{en^{r - 1} \rho^{2r - 2} \lambda^{2}} \rPa^{2\ell + 2}\,.
\]
To bound the last quantity on the RHS, we shall assume that
\begin{equation}
    \lambda = \sqrt{\frac{(r - 2)!(1 + \eps)}{en^{r-1}\rho^{2r-2}}}\,.
\label{eq:sparse:PCA:upper:bound:assumption}
\end{equation}
This condition on $\lambda$ can be assumed without loss of generality, because increasing $\lambda$ only makes estimation easier (see e.g.~\cite[Claim A.2]{SW-22}). Note that the assumption $\rho = \widetilde{\omega}(1/n)$ guarantees that $\zeta = \rho^{r - 1} \lambda^{2} = \lPa \frac{C}{n\rho} \rPa^{r-1}=o(1)$. The assumption $\frac{en^{r - 1} \rho^{2r - 2} \lambda^{2}}{(r - 2)!} \geq 1 + \eps$ then gives
\[
    \rho^{-1} \ell^{3} \lPa \frac{(1 + \zeta)(r - 2)!}{en^{r - 1} \rho^{2r - 2} \lambda^{2}} \rPa^{2\ell + 2} \leq \rho^{-1}\ell^3 \lPa \frac{1+o(1)}{1+\eps} \rPa^{2\ell+2} \leq \rho^{-1} \ell^{3}(1-\eps/4)^{2\ell + 2} \leq \ell^3\rho\,,
\]
where the last inequality holds by our choice $\ell=\lceil \frac{4}{\eps} \log(1/\rho)\rceil$. The assumption $\rho=o(\log^{-6(r - 1) - 3}{n})$ ensures that the RHS of the above display is $o(1)$, thus proving~\eqref{eq:sparse:PCA:identical:copies}. Lastly, we show that
\begin{equation}
    \sum_{\substack{\alpha, \beta \in \sT \\ \alpha \cap \beta \neq 0, \alpha \neq \beta}} \E[Y^{\alpha + \beta}] = o(\Gamma)\,.
\label{eq:sparse:PCA:overlapping:trees}
\end{equation}
Bounding this final case is the most technical step of the upper bound. We shall require some additional notation. Given a pair of trees $\alpha, \beta \in \sT$, the \textit{core} is the forest $\gamma = \alpha \cap \beta \cup \{1\}$, where we treat $\{1\}$ as a connected component of the core even when it is an isolated vertex. Note that since $\al, \be$ are trees, $\al\cap\be$ is a forest. If $\alpha$ and $\beta$ share a root incident edge, then $\alpha \cap \beta$ contains vertex $1$, in which case $\gamma=\al\cap \be$. If not, then $\ga$ is the union of $\al\cap \be$ with isolated vertex $1$. Fix a pair of trees $\alpha, \beta \in \sT$ for now. Let
\begin{itemize}
    \item[i.] $t$ be the number of connected components of the core $\gamma$ and the components be $\pi_{1}, \hdots, \pi_{t}$.
   \item[ii.] $a := |\alpha \cap \beta|$ be the total number of edges in the core.
  \item[iii.] $w := |(V(\alpha) \cap V(\beta)) \setminus V(\gamma)|$ be the number of shared vertices not contained in the core.
   \item[iv.] $B_{j} := V(\pi_{j}) \cap V(\alpha \triangle \beta)$ be the vertices where $\pi_j$ intersects $\alpha \triangle \beta$, with $b_{j} = |B_{j}|$. The vertices in $B := \cup_{j = 1}^{t} B_{j}=V(\al\triangle \be)\cap V(\ga)$ are called \textit{branch points}. Put $b = |B|$.
\end{itemize}
We remark that these quantities are precisely the parameters used in~\cite[Section 7.1]{arxiv-version} to prove an estimation upper bound for the planted submatrix model.

\begin{lemma}\label{lem:covariance:overlapping:pairs}
Consider any $\alpha, \beta \in \sT$ such that $\al\neq \be$ and let $\gamma=\al\cap\be \cup \{1\}$ be the core as defined above. Let $t$ be the number of connected components of $\ga$, $a=|\al\cap \be|$, $b=|V(\al\tri\be)\cap V(\ga)|$, and similarly $w=|V(\al)\cap V(\be)\setminus V(\ga)|$. Then
\begin{equation}
    \E[Y^{\alpha + \beta}] \leq \rho^{4k + 4r - 6 - 2a(r - 1) - t - w} \lambda^{4\ell + 4 - 2a}(1 + \zeta)^{a - \frac{b - t}{r - 1}}(\rho + \zeta)^{\frac{b - t}{r - 1}}\,.
\end{equation}
\end{lemma}

\begin{proof}
Recall from~\eqref{eq:def:eta:zeta} that $\eta\equiv \la^2, \zeta\equiv\rho^{r-1}\eta$, and $Y=X+Z$ where $X=\lambda\theta^{\otimes r}$. For a subset of vertices $A \subseteq [n]$, recall that $\theta_A=(\theta_i)_{i \in A}$ denotes the subvector of $\theta$ indexed by $A$. Integrating out the independent Gaussian noise $Z$ gives
\[
\begin{aligned}
    \E[Y^{\alpha + \beta}] 
    = \E \lBr X^{\alpha \triangle \beta} \prod_{e \in \alpha \cap \beta}(X_e^2 + 1) \rBr = \rho^{|V(\alpha \triangle \beta)|} \lambda^{|\alpha \triangle \beta|} \cdot \E \lBr \prod_{e \in \alpha \cap \beta} (\eta \theta^{V(e)} + 1) \, \Bigg | \, \theta_{V(\alpha \triangle \beta)} = 1 \rBr\,,
\end{aligned}
\]
where the last equality holds since $X_e= \lambda \theta^{V(e)}$ and $X^{\al\tri\be}= \lambda^{|\al\tri\be|} \theta^{V(\al\tri\be)}$.
We bound the conditional expectation by computing over the components of the core separately.

Each component $\pi_{j}$ contains vertices adjacent to $\alpha \triangle \beta$ and so we must modify the argument used in Lemma~\ref{lem:single:tree:second:moment} to account for these conditioned vertices. Since $\alpha$ and $\beta$ are distinct
trees, every core component has at least one branch point, so choose
$u_j\in B_j\equiv V(\pi_j)\cap V(\al\tri \be)$. Explore the hypertree $\pi_j$ from $u_j$. For each edge
$e\in\pi_j$, let $A_e$ be the set of new vertices introduced by $e$, so the
sets $(A_e)_{e\in\pi_j}$ are disjoint, each has size $r-1$, and their union
is $V(\pi_j)\setminus\{u_j\}$. Put
\[
    q_{e} = |A_{e} \cap B_{j}|\,,
\]
which equals the number of vertices in $A_{e}$ which are affected by the conditioning. Note that
\begin{equation}
    \sum_{e \in \pi_{j}} q_{e} = b_{j} - 1
\label{eq:affected:branch:points}
\end{equation}
by construction. Since $\theta^{V(e)}\le \theta^{A_e}$ and the vertices in
$A_e\cap B_j$ are conditioned to be equal to $1$, 
\[
\begin{aligned}
    \E\left[
        \prod_{e\in\pi_j}(1+\eta\theta^{V(e)})
        \,\middle|\,
        \theta_{V(\alpha\triangle\beta)}=1
    \right]
    &\le  \E \lBr \prod_{e \in \pi_{j}} (\eta \theta^{A_{e}} + 1) \, \Bigg | \, \theta_{V(\alpha \triangle \beta)} = 1 \rBr\\
    &\leq 
    \prod_{e\in\pi_j}\left(1+\eta\rho^{r-1-q_e}\right)  =
    \rho^{1-b_j}
    \prod_{e\in\pi_j}(\rho^{q_e}+\zeta)\,,
\end{aligned}
\]
where the last step used the identity \eqref{eq:affected:branch:points}. Let $m_{j} := |\{e \in \pi_{j}: q_{e} \geq 1\}|$ be the number of edges $e$ for which $A_{e}$ contains a branch point. Since each edge can contain at most $r - 1$ branch points,  we have   $m_{j} \geq \frac{b_{j} - 1}{r - 1}$. Note that for edges with $q_e=0$, the factor in the product in the RHS is $1+\zeta$, and for edges with
$q_e\ge1$, the factor is at most $\rho+\zeta$. Hence
\[
\begin{aligned}
    \prod_{e\in\pi_j}(\rho^{q_e}+\zeta)
    \le
    (1+\zeta)^{|\pi_j|}
    \left(\frac{\rho+\zeta}{1+\zeta}\right)^{m_j}  \le
    (1+\zeta)^{|\pi_j|}
    \left(\frac{\rho+\zeta}{1+\zeta}\right)^{\frac{b_j-1}{r-1}}\,.
\end{aligned}
\]
Therefore
\[
    \E\left[
        \prod_{e\in\pi_j}(1+\eta\theta^{V(e)})
        \,\middle|\,
        \theta_{V(\alpha\triangle\beta)}=1
    \right]
    \le
    \rho^{1-b_j}
    (1+\zeta)^{|\pi_j|}
    \left(\frac{\rho+\zeta}{1+\zeta}\right)^{\frac{b_j-1}{r-1}}\,.
\]
As disjoint components in the core are independent, we take the product of the above conditional expectation for $j=1,\hdots,t$ to get
\[
    \E \lBr \prod_{e \in \alpha \cap \beta} (\eta \theta^{V(e)} + 1) \, \Bigg | \,  \theta_{V(\alpha \triangle \beta)} = 1 \rBr \leq \rho^{t - b}(1 + \zeta)^{a} \lPa \frac{\rho + \zeta}{1 + \zeta} \rPa^{\frac{b - t}{r - 1}}\,,
\]
where we used $b =|B|=\sum_{j} |B_{j}|=\sum_{j} b_{j}$ and $a =|\al \cap \be|=\sum_{j} |\pi_{j}|$. It follows that
\[
    \E[Y^{\alpha + \beta}] \leq \rho^{|V(\alpha \triangle \beta)|} \lambda^{|\alpha \triangle \beta|} \rho^{t - b}(1 + \zeta)^{a - \frac{b - t}{r - 1}}(\rho + \zeta)^{\frac{b - t}{r - 1}}\,.
\]
Lastly, note that $|\alpha \triangle \beta| = 4\ell + 4 - 2a$. Since $|V(\al)|=|V(\be)|=2k+2r-3$ and $|V(\gamma)| = a(r - 1) + t$,
\[
\begin{aligned}
    |V(\alpha \triangle \beta)| &= |V(\alpha \triangle \beta) \setminus V(\gamma)| + |V(\alpha \triangle \beta) \cap V(\gamma)| \\
    &= |V(\alpha) \setminus V(\gamma)| + |V(\beta) \setminus V(\gamma)| - |(V(\alpha) \cap V(\beta)) \setminus V(\gamma)| + |B| \\
    &= 4k + 4r - 6 - 2a(r - 1) - 2t - w + b\,.
\end{aligned}
\]
Combining with the preceding inequality, we obtain the stated bound.
\end{proof}

The next step is to estimate the number of pairs $\alpha, \beta \in \sT$ giving rise to a realizable tuple of parameters.

\begin{lemma}\label{lem:counting:overlapping:pairs}
Given  $t, a, w, b\geq 1$, the number of pairs $\alpha, \beta \in \sT$ such that the number of connected components of the core $\ga=\al\cap \be \cup \{1\}$ is $t$, and satisfies $a=|\al\cap \be|$, $b=|V(\al\tri\be)\cap V(\ga)|$, and $w=|V(\al)\cap V(\be)\setminus V(\ga)|$ is at most
\[
    C_r^{t + b + w} \ell^{3b + 2w}n^{1 - w - t} \lPa \frac{en^{r - 1}}{(r - 2)!} \rPa^{4\ell + 4 - a}\,,
\]
where $C_r>0$ is a constant which only depends on $r$. 
\end{lemma}

\begin{proof}
Throughout, we let $v=a(r-1)+t$ denote the number of vertices in the core, and $C_r>0$ denote a constant that only depends on $r$. By Lemma~\ref{lem:rooted:forests}, the number of possible cores with $v$ vertices and $a$ edges is at most
\[
    \binom{n-1}{v-1}R_{v,t}= \binom{n-1}{v-1}  \frac{v!}{(t-1)!}
    \frac{v^{a-1}}{a!(r-1)!^a}\leq n^{v-1}
    \left(\frac{ev}{a(r-1)!}\right)^a\,,
\]
where we used $\binom{n-1}{v-1}\leq \frac{n^{v-1}}{(v-1)!}$, $(t-1)!\geq 1$, $a!\geq (a/e)^a$ in the last step. Since $v=a(r-1)+t$,
\[
    \left(\frac{v}{a(r-1)}\right)^a
    =
    \left(1+\frac{t}{a(r-1)}\right)^a
    \le
    e^{t/(r-1)}\,.
\]
Thus, the number of cores is at most
\begin{equation}\label{eq:numer:core}
    C_r^t n^{-1+t}
    \left(\frac{en^{r-1}}{(r-2)!}\right)^a\,.
\end{equation}
Next, we count the number of ways to construct $\al$ and $\be$ from the core. First, choose the $w$ shared non-core vertices and the $b$ branch points.
This contributes at most
\begin{equation}\label{eq:non-core:branch}
    \binom{n}{w}\binom{v}{b}
    \le
    n^w(C_rk)^b\,.
\end{equation}
It remains to reconstruct $\alpha\setminus \be$ and
$\be\setminus \al$. Put
\[
    L:=2\ell+2-a,
    \qquad
    u:=2k+2r-3-v+b\,.
\]
Here $L$ is the number of non-core edges in one tree, say
$\alpha\setminus\beta$, since each tree in $\sT$ has $2\ell+2$
edges and the core has $a$ edges. Also, $u$ is the number of vertices
remaining when we delete the core from $\alpha$ but retain the branch
points (recall $\al$ has $2k+2r-3$ vertices). The resulting object is an $r$-uniform forest. Its number of edges is
$L$, and its number of vertices is $u$, so its number of connected
components is
\[
    u-(r-1)L=b+1-t\,,
\]
where we used $k\equiv (r-1)\ell+1$ and $v\equiv a(r-1)+t$. Thus each side can be overcounted by first choosing its non-core,
non-shared vertices and then choosing a rooted $r$-uniform forest on
$u$ vertices with $b+1-t$ components, so by Lemma~\ref{lem:rooted:forests}, the number of ways to choose $\al\setminus \be$ is at most
\[
    \binom{n-w}{u-b-w}R_{u,b+1-t}=\binom{n-w}{u-b-w}\frac{u!}{(b-t)!}\frac{u^{L-1}}{L!(r-1)!^L}\,.
\]
By the crude bounds $u!/(u-b-w)!\leq u^{b+w}$ and $u^{L-1}\le u^L$, this is at most
\[
\begin{aligned}
    \binom{n-w}{u-b-w}R_{u,b+1-t}
    \le n^{u-b-w}u^{b+w}
    \frac{u^L}{L!(r-1)!^L} \leq
    n^{1-t-w}(C_rk)^{b+w}
    \frac{n^{(r-1)L}u^L}{L!(r-1)!^L}\,, 
\end{aligned}
\]
where the last step uses $u\leq C_rk$. Since $L!\geq (L/e)^L$, we have
\[
    \frac{n^{(r-1)L}u^L}{L!(r-1)!^L}
    \le
    \left(\frac{u}{(r-1)L}\right)^L
    \left(\frac{en^{r-1}}{(r-2)!}\right)^L\,.
\]
Moreover, $b \geq t \geq 1$ because each core component has at least one branch point and the core is non-empty. It follows that
$0\le b+1-t\le b$ and using $u=(r-1)L+b+1-t$, we get
\[
    \left(\frac{u}{(r-1)L}\right)^L
    =
    \left(1+\frac{b+1-t}{(r-1)L}\right)^L
    \le
    \exp\left(\frac{b+1-t}{r-1}\right)
    \le
    C_r^b\,.
\]
Hence the number of choices for $\al\setminus \be$ is at most
\[
    n^{1-t-w}(C_rk)^{b+w}C_r^b
    \left(\frac{en^{r-1}}{(r-2)!}\right)^L\,.
\]
The same bound applies to $\be\setminus \al$. Thus, squaring and multiplying this with the number of ways to choose the core~\eqref{eq:numer:core}, the shared vertices and branch points~\eqref{eq:non-core:branch}, we obtain that the number of pairs of possible $\al,\be$ is at most
\[
\begin{aligned}
    C_r^{t+b+w}
    k^{3b+2w}
    n^{1-t-w}
    \left(\frac{en^{r-1}}{(r-2)!}\right)^{a+2L}\,.
\end{aligned}
\]
Because $L=2\ell+2-a$, we have $a+2L=4\ell+4-a$. Finally, $k^{3b+2w}\le C_r^{b+w}\ell^{3b+2w}$ holds since $k\le C_r\ell$. Absorbing constants gives the claimed bound.
\end{proof}

Lastly, we require the following structural lemma on trees in $\sT$, which will allow us to boost the estimate from this case from $O(\Gamma)$ to $o(\Gamma)$. Note that, up to this point, it was not necessary to use the special structure of the trees at all.

\begin{lemma}\label{lem:branch:points}
If $\alpha, \beta \in \sT$ are distinct, non-disjoint trees with $b = 1$, then necessarily $a \geq \ell$.
\end{lemma}

\begin{proof}
We distinguish three cases: 

\begin{itemize}
    \item[1.] Both trees share the same set of root-incident edges. In this case, the root vertex $1$ is not isolated in the core and in particular cannot be a branch point. Consider the two subtrees in $\alpha$. Aside from the root, these two trees have no vertices in common. Since $b = 1$, exactly one of the subtrees has a branch point.

    For the subtree of $\alpha$ that contains no branch points, we claim it must be entirely contained in the core. Otherwise, this subtree contains an edge that branches off from the core, which creates a branch point, a contradiction. Hence $a \geq \ell + 1$ in this case.

    \item[2.] The two trees share exactly one root-incident edge, call it $e$. The root vertex $1$ is then a branch point. Consider the two subtrees of $\alpha$. Since $b = 1$, the subtree containing $e$ cannot contain any non-root branch points, and the same argument as in Case 1 implies that this entire subtree lies in the core.

    \item[3.] The two trees share no root-incident edges. Then the root vertex $1$ is isolated in the core and is therefore a branch point by definition. Consider again the two subtrees of $\alpha$. Since $\alpha \cap \beta \neq 0$, one of the two subtrees must intersect the core, producing a branch point $i$. Note that $i \neq 1$, since the root-incident edges are not in the core. Thus $b \geq 2$, contradicting $b = 1$.
\end{itemize}
\end{proof}

We are now ready to prove~\eqref{eq:sparse:PCA:overlapping:trees}. From Lemmas~\ref{lem:covariance:overlapping:pairs} and~\ref{lem:counting:overlapping:pairs}, the LHS of~\eqref{eq:sparse:PCA:overlapping:trees} is at most
\[
    \sum_{a, t, b, w} \rho^{4k + 4r - 6 - 2a(r - 1) - t - w} \lambda^{4\ell + 4 - 2a}(1 + \zeta)^{a} \lPa \frac{\rho + \zeta}{1 + \zeta} \rPa^{\frac{b - t}{r - 1}}C^{t + b + w} \ell^{3b + 2w}n^{1 - t - w} \lPa \frac{en^{r - 1}}{(r - 2)!} \rPa^{4\ell + 4 - a}\,.
\]
Rearranging and grouping terms together, the above summation is equal to
\[
    \rho^{4k + 4r - 7} \lambda^{4\ell + 4} \lPa \frac{en^{r - 1}}{(r - 2)!} \rPa^{4\ell + 4} \sum_{\substack{a, t \geq 1 \\ b \geq t}} \lPa \frac{en^{r - 1} \rho^{2r - 2} \lambda^{2}}{(1 + \zeta)(r - 2)!} \rPa^{-a} \lPa \frac{C}{n\rho} \rPa^{t - 1} \lPa \frac{C \ell^{3}(\rho + \zeta)}{1 + \zeta} \rPa^{\frac{b - t}{r - 1}} \sum_{w \geq 0} \lPa \frac{C \ell^{2}}{n\rho} \rPa^{w}\,.
\]
By~\eqref{eq:asymptotic:tree:lower:bound}, the prefactor outside the summation in the above display is at most $\Gamma$. Furthermore, the assumption $\rho = \omega(n^{-1} \log^{6 + 3/(r-1)}{n})$ guarantees that $\ell^2/(n\rho) = o(1)$ and so the summation over $w$ is bounded by an absolute constant. Hence, we bound the preceding display by
\[
    C\ell^{6} \sum_{\substack{a, t \geq 1 \\ b \geq t}} \lPa \frac{en^{r - 1} \rho^{2r - 2} \lambda^{2}}{(1 + \zeta)(r - 2)!} \rPa^{-a} \lPa \frac{C}{n\rho} \rPa^{t - 1} \lPa C \ell^{3}(\rho + \zeta) \rPa^{\frac{b - t}{r - 1}}\,,
\]
where we dropped the $1 + \zeta$ factor in the last term by absorbing it into a constant factor, recalling that $\zeta = o(1)$. Using the assumption $\frac{en^{r - 1} \rho^{2r - 2} \lambda^{2}}{(r - 2)!} \geq 1 + \eps$ and taking $n$ large so that $\frac{1+o(1)}{1+\eps} \leq 1 -\eps/4$, we obtain
\begin{equation}
    \sum_{\substack{\alpha, \beta \in \sT \\ \alpha \cap \beta \neq 0, \alpha \neq \beta}} \E[Y^{\alpha + \beta}] \leq C\ell^6 \sum_{\substack{a, t \geq 1 \\ b \geq t}} (1 - \eps/4)^{a} \lPa \frac{C}{n\rho} \rPa^{t - 1} \lPa C \ell^{3}(\rho + \zeta) \rPa^{\frac{b - t}{r - 1}}\,.
\label{eq:sparse:PCA:upper:bound:final:case:estimate}
\end{equation}
Recalling~\eqref{eq:sparse:PCA:upper:bound:assumption}, the assumptions $\rho = o(\log^{-6(r - 1) - 3}{n})$ and $\rho = \omega(n^{-1} \log^{6 + 3/(r-1)}{n})$ imply that
\[
    C\ell^{3}(\rho + \zeta) \leq C\ell^{3}\rho + \lPa \frac{C\ell^{3/(r-1)}}{n\rho} \rPa^{r-1} = o(1)\,.
\]
Choose $n$ sufficiently large so that the $o(1)$ term on the RHS is at most $1/2$. Isolate the $t = 1$ case: 
\[
\begin{aligned}
    \sum_{a, b \geq 1} (1 - \eps/4)^{a} \lPa C \ell^{3}(\rho + \zeta) \rPa^{\frac{b - 1}{r - 1}} &= \sum_{a \geq \ell} (1 - \eps/4)^{a} + \sum_{\substack{a \geq 1 \\ b \geq 2}} (1 - \eps/4)^{a} \lPa C \ell^{3}(\rho + \zeta) \rPa^{\frac{b - 1}{r - 1}} \\
    & \leq (1 - \eps/4)^{\ell} + O_{r, \eps} \lPa \lPa \ell^3(\rho + \zeta) \rPa^{1/(r-1)} \rPa\,,
\end{aligned}
\]
where we use Lemma~\ref{lem:branch:points} for the $b = 1$ case on the RHS. For the $t \geq 2$ case, directly take summation over $a$ and $b \geq t$ to obtain
\[
    \sum_{a \geq 1, t \geq 2} \sum_{b \geq t} (1 - \eps/4)^{a} \lPa \frac{C}{n\rho} \rPa^{t - 1} \lPa C \ell^{3}(\rho + \zeta) \rPa^{\frac{b - t}{r - 1}} = O_{r,\eps} \lPa \frac{1}{n\rho} \rPa\,.
\]
Combining the above casework and recalling that $(1 - \eps/4)^{\ell} \leq \rho$ and $\zeta=\lPa \frac{C}{n\rho} \rPa^{r-1}$, the RHS of~\eqref{eq:sparse:PCA:upper:bound:final:case:estimate} is at most
\[
\begin{aligned}
    & C\ell^6(1-\eps/4)^{\ell} + C\ell^6(\ell^3(\rho+\zeta))^{1/(r-1)} + \frac{C\ell^6}{n\rho} \\
    &\leq C\ell^{6}\rho + C\lPa \ell^{6(r - 1) + 3}\rho + C \lPa \frac{\ell^{6 + 3/(r-1)}}{n\rho} \rPa^{r - 1} \rPa^{1/(r-1)} + \frac{C\ell^{6}}{n\rho}\,.
\end{aligned}
\]
As before, the assumptions on $\rho$ guarantee that the RHS of the preceding display is $o(1)$.

\begin{proof}[Proof of Theorem~\ref{thm:sparse:PCA}-(b)]
Combining~\eqref{eq:sparse:PCA:disjoint:trees},~\eqref{eq:sparse:PCA:identical:copies}, and~\eqref{eq:sparse:PCA:overlapping:trees}, we obtain $\E[f(Y)^2] = (1+o(1))\frac{\E[f(Y)x]^2}{\E[x^2]}$. Rearranging gives $\frac{\E[f(Y)x]^2}{\E[f(Y)^2] \E[x^2]} = 1-o(1)$ and thus $\Corr_{\leq C \log{n}} = 1-o(1)$.
\end{proof}

\subsection{Planted dense subhypergraph}\label{subsec:planted:hypergraph:upper:bound}
This section is devoted to the proofs of Theorem~\ref{thm:planted:hypergraph}-(b) and Theorem~\ref{thm:small:planted:hypergraph}-(b).

\subsubsection{Large planted dense hypergraph}
We first consider the more general case without the restriction that $\rho \leq n^{-1/2}$. In analogy to the proof of our lower bound, the calculations from the planted subtensor model can be reused here.

\begin{proof}[Proof of Theorem~\ref{thm:planted:hypergraph}-(b)]
Recall the tree family $\sT\equiv\sT_\ell$ from Definition~\ref{def:special:trees}, which was used in Section~\ref{subsec:sparse:PCA:upper:bound}. Recall from~\eqref{eq:choice:k:ell} that we have set $\ell =\lceil \frac{4}{\eps}\log(1/\rho)\rceil$ and $k= (r-1)\ell+1$ so that every tree $\alpha\in\sT$ has $|V(\al)|=2k+2r-3$ and $|\al|=2\ell+2$. Set
\[
    \widetilde Y_e:=\frac{Y_e-q_0}{\sqrt{q_0(1-q_0)}}\,,
    \qquad
    \lambda:=\frac{q_1-q_0}{\sqrt{q_0(1-q_0)}}\,,
    \qquad
    \eta:=\frac{q_1-q_0}{q_0(1-q_0)}\,.
\]
Consider the polynomial
\[
    f(Y):=\sum_{\alpha\in\sT}\widetilde Y^\alpha\,.
\]
We first compute $\E[f(Y)x]$. Since the root vertex $1$ belongs to every $\alpha\in\sT$,
\[
\begin{aligned}
    \E[f(Y)\theta_1]
    =
    \sum_{\alpha\in\sT}
    \E[\theta_1\widetilde Y^\alpha] =
    \sum_{\alpha\in\sT}
    \lambda^{|\alpha|}
    \P(\theta_{V(\alpha)}=1) =
    |\sT|\rho^{2k+2r-3}\lambda^{2\ell+2}\,.
\end{aligned}
\]
Therefore, $\Gamma\equiv \frac{\E[f(Y)\theta_1]^2}{\E[\theta_1^2]}=|\sT|^2\rho^{4k+4r-7}\lambda^{4\ell+4}$ holds as in the sparse tensor PCA model~\eqref{eq:def:Gamma}.

It remains to bound second-moment $\E[f(Y)^2]$. We claim that the second-moment estimates from the sparse tensor PCA proof apply verbatim. To see this, fix
$\alpha,\beta\in\sT$. Conditioning on $\theta$, we have $\E_{\theta}[\widetilde{Y}_e]=\lambda \theta^{V(e)}$. Also, $ \E_{\theta}[\widetilde Y_e^2]=1$ if $\theta^{V(e)}=0$, and if $\theta^{V(e)}=1$, then
\[
    \E\nolimits_{\theta}[\widetilde Y_e^2]=\frac{q_1(1-q_1)+(q_1-q_0)^2}{q_0(1-q_0)}
    =
    1+(1-2q_0)\eta \leq 1+\eta\,.
\]
Thus, $\E_{\theta}[\widetilde{Y}_e^2]\leq 1+\eta \theta^{V(e)}$ and
\[
\begin{aligned}
    \E[\widetilde Y^{\alpha+\beta}]
    =
    \E\left[
        \widetilde Y^{\alpha\triangle\beta}
        \widetilde Y^{2(\alpha\cap\beta)}
    \right]  \le
    \rho^{|V(\alpha\triangle\beta)|}
    \lambda^{|\alpha\triangle\beta|}
    \E\left[
        \prod_{e\in\alpha\cap\beta}
        \bigl(1+\eta\theta^{V(e)}\bigr)
        \,\middle|\,
        \theta_{V(\alpha\triangle\beta)}=1
    \right]\,.
\end{aligned}
\]
This is exactly the upper bound used in the sparse tensor PCA second-moment
analysis, with $\eta=\lambda^2$ there replaced by the present value $ \eta=\frac{q_1-q_0}{q_0(1-q_0)}$. Consequently, Lemmas~\ref{lem:single:tree:second:moment},
\ref{lem:covariance:overlapping:pairs}, and
\ref{lem:counting:overlapping:pairs} imply
\[
    \E[f(Y)^2]\le (1+o(1))\Gamma\,
\]
provided
\[
    \frac{en^{r-1}\rho^{2r-2}\lambda^2}{(r-2)!}\ge 1+\eps\,,
    \quad
    \rho=\omega\!\left(n^{-1}\log^{6+\frac{3}{r-1}}n\right)\,,\quad
    \rho+\zeta=o\!\left(\log^{-6(r-1)-3}n\right)\,,
\]
where $\zeta=\rho^{r-1}\eta$. Without loss of generality, we may replace the last assumption with 
\[
    \rho = o(\log{n}^{-6(r - 1) - 3}), \quad q_{0} = \omega(n^{1-r}\log^{12(r-1)+6}{n})\,.
\]
Indeed, we may assume that $q_{1}$ is as small as possible (see~\cite[Claim A.2]{SW-22}), since increasing $q_{1}$, with all other parameters held fixed, only improves estimation. Thus, for a given $q_{0}$, we may assume that the bound on $\lambda$ holds with equality, i.e. 
\[
\lambda = \sqrt{\frac{(r - 2)!(1 + \eps)}{en^{r-1}\rho^{2r-2}}}\,.
\]
Given this constraint, $\zeta = \frac{\rho^{r-1}\lambda}{\sqrt{q_{0}(1-q_{0})}}$ is $o(\log{n}^{-6(r - 1) - 3})$ provided $q_{0} = \omega(n^{1-r}\log^{12(r-1)+6}{n})$. It follows that $\frac{\E[f(Y)\theta_1]^2}{\E[f(Y)^2]\E[\theta_1^2]} \geq 1-o(1)$ under the stated assumptions, whence $\Corr_{\le C\log n}=1-o(1)$.
\end{proof}

\subsubsection{Small dense planted hypergraph}\label{subsec:small:hypergraph:upper:bound}
Now we specialize to the special scaling regime where, 
\[
    \rho = n^{\xi - 1}, \quad q_{0} = n^{-b}, \quad q_{1} = n^{-a}
\]
for fixed parameters $\xi \in (0, 1/2]$ and $a, b \in (0, r - 1)$ with $a < b$.

As pointed out in the proof of our lower bound for small dense planted hypergraph, the existence of particularly dense graphs will contribute substantially to the correlation in this regime. Thus, our estimator will count these dense graphs, and this motivates the next definition.

\begin{definition}
A hypergraph $\alpha$ is \textit{strongly balanced} if $\frac{|\beta|}{|V(\beta)| - 1} \leq \frac{|\alpha|}{|V(\alpha)| - 1}$ for all $0 \neq \beta \leq \alpha$.
\end{definition}

\begin{remark}
Every strongly balanced graph is balanced in the sense of~\cite[Definition 4.3]{DMW-23}.
\end{remark}

\begin{fact}\label{fact:strongly:balanced:hypergraph}
If $\frac{1}{r - 1} \leq \frac{\ell}{k - 1}$ and $\ell \leq \binom{k}{r}$, then there exists a strongly balanced hypergraph with $k$ vertices and $\ell$ edges.
\end{fact}

\begin{proof}
See~\cite[Theorem 2.6]{RV-88}.
\end{proof}

\begin{proof}[Proof of Theorem~\ref{thm:small:planted:hypergraph}-(b)]
Since $b \leq r - 1$ and $a < b\xi$, we can choose integers $k$ and $\ell$ such that, 
\[
    \frac{1}{b} < \frac{\ell}{k} < \frac{\ell}{k - 1} < \frac{\xi}{a}\,.
\]
By Fact~\ref{fact:strongly:balanced:hypergraph}, there is a strongly balanced hypergraph $\alpha_\star$ with $k$ vertices and $\ell$ edges. Take $D = \ell$ and fix $\alpha_\star$. Let $\sB = \{\alpha \in \{0, 1\}^{\binom{[n]}{r}}: \text{$1 \in V(\alpha)$ and $\alpha$ is isomorphic to $\alpha_\star$}\}$. Note that, 
\[
    |\sB| = \binom{n - 1}{k - 1}L_\star = (1 - o(1)) \frac{L_\star n^{k - 1}}{(k - 1)!}\,,
\]
where $L_\star = \frac{k!}{|\Aut(\alpha_\star)|}$ is the number of graphs isomorphic to $\alpha_\star$ on a prescribed set of $k$ vertices, which here is a constant since $\alpha_\star$ is being held fixed. The binomial coefficient counts the number of ways to select the vertices. Our candidate estimator will be: 
\[
    f(Y) = \sum_{\alpha \in \sB} Y^{\alpha}\,.
\]
The following is useful for estimating the correlation. Let $\delta = \max \lPa \frac{q_{0}}{q_{1}}, \rho^{-1}(\frac{q_{0}}{q_{1}})^{\ell/k} \rPa$ and compute
\[
    \frac{q_0}{q_1} = n^{a - b}\,, \quad \rho^{-1} \lPa \frac{q_{0}}{q_{1}} \rPa^{\ell/k} = n^{1 - \xi + (a - b)\ell/k} \leq n^{1 - \xi + (a - b)/b} = n^{-\xi + a/b}\,.
\]
The condition $a < b\xi$ guarantees the last exponent is strictly negative. As a consequence of the above, we see that $\delta = o(1)$ since $k$ is held constant. In particular, we have $\delta \leq 1$ for all large $n$.

\begin{lemma}\label{lem:small:hypergraph:upper:bound:estimate:A}
For any $\alpha \in \sB$, 
\[
    \E[Y^{\alpha}x] \geq \rho^{k}q_{1}^{\ell}\,.
\]
\end{lemma}

\begin{proof}
By non-negativity, 
\[
    \E[Y^{\alpha}x] \geq \E[\In_{\theta_{V(\alpha)} = 1} \E\nolimits_{\theta}[Y^{\alpha}]] = q_{1}^{\ell} \cdot \Pb(\theta_{V(\alpha)} = 1) = \rho^{k}q_{1}^{\ell}\,.
\]
\end{proof}

\begin{lemma}\label{lem:small:hypergraph:upper:bound:estimate:B}
For any $\alpha, \beta \in \sB$, 
\[
    \E[Y^{\alpha + \beta}] \leq (1 + 4^k\delta) \rho^{|V(\alpha \cup \beta)|}q_{1}^{|\alpha \cup \beta|} = (1 + 4^{k}\delta) \rho^{2k - 1}q_{1}^{2\ell} \lPa  \rho^{-|V(\alpha) \cap V(\beta)| + 1}q_{1}^{-|\alpha \cap \beta|} \rPa\,.
\]
\end{lemma}

\begin{proof}
This a more careful book-keeping of~\cite[Lemma 4.8]{DMW-23}. Expanding over all possible realizations of $V(\alpha \cup \beta) \cap \theta$, the expectation $\E[Y^{\alpha + \beta}]$ is bounded by
\[
    \sum_{\sigma \in \{0, 1\}^{V(\alpha \cup \beta)}} \P(\theta_{V(\alpha \cup \beta)} = \sigma) \P(Y^{\alpha + \beta}=1 \mid \theta_{V(\alpha \cup \beta)} = \sigma) \leq \sum_{\sigma \in \{0, 1\}^{V(\alpha \cup \beta)}} \rho^{|\sigma|}q_0^{|\alpha \cup \beta|} \lPa \frac{q_1}{q_0} \rPa^{|E_\sigma|}\,,
\]
where $E_\sigma$ consists of those edges in $\alpha \cup \beta$ whose vertices are contained entirely in $\sigma$. Viewing $\sigma$ as a subset of $V(\alpha\cup\beta)$, we argue that the summand is maximized when $\sigma$ is precisely $V(\alpha \cup \beta)$, in which case $E_\sigma = \alpha\cup\beta$. Additionally, we claim that any $\sigma \subset V(\alpha\cup\beta)$ is strictly smaller by an extra factor of $\delta$. The claim then follows since there are at most $4^k$ possibilities for $\sigma$. Towards this end, we note that it is sufficient to instead consider maximizing
\[
    \rho^{|V(\gamma)|}q_0^{|\alpha \cup \beta|} \lPa \frac{q_1}{q_0} \rPa^{|\gamma|}
\]
over all $\gamma \subseteq \alpha\cup\beta$. Indeed, if the subgraph $(\sigma, E_\sigma)$ contains an isolated vertex, then each isolated vertex contributes an extra factor of $\rho \leq 1$. Thus, we can delete any isolated vertices without increasing the value of the objective. The goal now is to show that the maximizer is $\gamma=\alpha\cup\beta$, and the value for any proper subgraph is smaller by an additional factor of $\delta$. It suffices to show the latter. Assume that $\gamma$ is a proper subgraph of $\alpha\cup\beta$. Following the proof of~\cite[Lemma 4.8]{DMW-23}, make the definitions $\gamma_1 := \gamma \cap \alpha$ and $\gamma_2 := (\alpha\cap\beta)\cup(\gamma \setminus \gamma_1)$ so that
\[
    \rho^{|V(\alpha\cup\beta)|-|V(\gamma)|} \lPa \frac{q_1}{q_0} \rPa^{|\alpha\cup\beta|-|\gamma|} \geq \underbrace{\rho^{|V(\alpha)|-|V(\gamma_1)|} \lPa \frac{q_1}{q_0} \rPa^{|\alpha|-|\gamma_1|}}_{\Xi_1} \cdot \underbrace{\rho^{|V(\beta)|-|\gamma_2|} \lPa \frac{q_1}{q_0} \rPa^{|\beta|-|\gamma_2|}}_{\Xi_2}\,.
\]
By~\cite[Equations 29a and 29b]{DMW-23} we have $\Xi_1, \Xi_2 \geq 1$. Notice that $\gamma_1 \subseteq \alpha$ and $\gamma_2 \subseteq \beta$ always, with at least one of the containments being strict. Suppose that $\gamma_1 \subset \alpha$. If $V(\alpha) = V(\gamma_1)$, then $\Xi_1$ is at least $q_1/q_0$; otherwise
\[
    \Xi_1 = \lPa \rho \lPa \frac{q_1}{q_0} \rPa^{\frac{|\alpha|-|\gamma_1|}{|V(\alpha)| - |V(\gamma_1)|}} \rPa^{|V(\alpha)| - |V(\gamma_1)|} \geq \rho \lPa \frac{q_1}{q_0} \rPa^{\frac{|\alpha|-|\gamma_1|}{|V(\alpha)| - |V(\gamma_1)|}} \geq \rho \lPa \frac{q_1}{q_0} \rPa^{\ell/k}\,,
\]
where the last inequality holds by~\cite[Claim 4.5]{DMW-23}. Thus, we conclude that $\Xi_1 \geq \delta^{-1}$ whenever $\gamma_1 \subset \alpha$. If $\gamma_1=\alpha$, then necessarily $\gamma_2 \subset \beta$ and repeating the preceding argument with $\alpha$ replaced with $\beta$ and $\gamma_1$ replaced with $\gamma_2$ gives $\Xi_2 \geq \delta^{-1}$. In summary, we have shown that
\[
    \rho^{|V(\alpha\cup\beta)|-|V(\gamma)|} \lPa \frac{q_1}{q_0} \rPa^{|\alpha\cup\beta|-|\gamma|} \geq \delta^{-1}
\]
whenever $\gamma \subset \alpha \cup \beta$. Taking reciprocal on both sides of this inequality yields the claim.
\end{proof}

We can now estimate, using Lemma~\ref{lem:small:hypergraph:upper:bound:estimate:A}, 
\[
    \E[f(Y)x] = \sum_{\alpha \in \sB} \E[Y^{\alpha}x] \geq |\sB| \rho^{k}q_{1}^{\ell}\,.
\]

Observe that, for any pair of graphs $\alpha, \beta \in \sB$ having non-empty intersection, the property of strongly balanced hypergraphs implies $\frac{|\alpha \cap \beta|}{|V(\alpha) \cap V(\beta)| - 1} \leq \frac{|\alpha \cap \beta|}{|V(\alpha \cap \beta)| - 1} \leq \frac{\ell}{k - 1}$. Therefore
\[
    \rho^{-|V(\alpha) \cap V(\beta)| + 1}q_{1}^{-|\alpha \cap \beta|} = (\rho^{-1}q_{1}^{-\frac{|\alpha \cap \beta|}{|V(\alpha) \cap V(\beta)| - 1}})^{|V(\alpha) \cap V(\beta)| - 1} \leq (\rho^{-1}q_{1}^{-\frac{\ell}{k - 1}})^{|V(\alpha) \cap V(\beta)| - 1}\,.
\]
By Lemma~\ref{lem:small:hypergraph:upper:bound:estimate:B}, 
\[
\begin{aligned}
    \E[Y^{\alpha + \beta}] &\leq (1 + 4^k\delta) \rho^{2k - 1}q_{1}^{2\ell} \lPa \rho^{-|V(\alpha) \cap V(\beta)| + 1}q_{1}^{-|\alpha \cap \beta|} \rPa \\
    & \leq (1 + 4^k\delta) \rho^{2k - 1}q_{1}^{2\ell} \lPa \rho^{-1}q_{1}^{-\frac{\ell}{k - 1}} \rPa^{|V(\alpha) \cap V(\beta)| - 1}\,.
\end{aligned}
\]
Substituting $\rho=n^{\xi-1}$ and $q_1=n^{-a}$ gives $\rho^{-1}q_{1}^{-\ell/(k-1)} = n^{1 - \xi + a\ell/(k-1)}$. From our choice of $k, \ell$, there exists $\eps_{0} > 0$ such that $-\xi + a\ell/(k-1) + \eps_{0} \leq 0$. Inserting this into the above display gives
\[
    \E[Y^{\alpha + \beta}] \leq (1 + 4^k\delta) \rho^{2k - 1}q_{1}^{2\ell} \lPa n^{1 - \eps_{0}} \rPa^{|V(\alpha) \cap V(\beta)| - 1}\,.
\]
Note that Lemma~\ref{lem:small:hypergraph:upper:bound:estimate:B} implies that this estimate is valid even when $\alpha \cap \beta = 0$.

For a fixed $\alpha \in \sB$, let $v \in \{0, \hdots, k - 1\}$ represent the number of possible vertices not equal to $1$ that are shared between $\alpha$ and another hypergraph $\beta \in \sB$. For a given $v$, the total number of possible $\beta$ such that $|V(\alpha) \cap V(\beta)| - 1 = v$ is upper bounded by, 
\[
    \binom{k - 1}{v} \binom{n - 1}{k - 1 - v}L_\star \leq \frac{n^{k - 1 - v}}{(k - 1 - v)!} \binom{k - 1}{v} L_\star\,.
\]
The reasoning here is that the first binomial coefficient counts the number of ways to choose those shared vertices (which are not $1$) and the second binomial coefficient bounds the number of ways to choose the non-shared vertices in $\beta$. Once the vertices have been chosen, the number of ways to arrange the edges in $\beta$ is bounded by $L_\star$. Thus, 
\[
\begin{aligned}
    \E[f(Y)^{2}] = \sum_{\alpha, \beta \in \sB} \E[Y^{\alpha + \beta}] & \leq (1 + 4^k\delta) \rho^{2k - 1}q_{1}^{2\ell} \sum_{\alpha \in \sB} \sum_{\beta \in \sB} \lPa n^{1 - \eps_{0}} \rPa^{|V(\alpha) \cap V(\beta)| - 1} \\
    & \leq (1 + 4^k\delta) \rho^{2k - 1}q_{1}^{2\ell} \sum_{\alpha \in \sB} \sum_{v = 0}^{k - 1} \frac{L_\star n^{k - 1 - v}}{(k - 1 - v)!} \binom{k - 1}{v} \lPa n^{1 - \eps_{0}} \rPa^{v} \\
    &= (1 + 4^k\delta) \rho^{2k - 1}q_{1}^{2\ell} \lPa \frac{L_\star n^{k - 1}}{(k - 1)!} \rPa \sum_{\alpha \in \sB} \sum_{v = 0}^{k - 1} \frac{(k - 1)!}{(k - 1 - v)!} \binom{k - 1}{v}n^{-\eps_{0}v} \\
    & \leq (1 + 4^k\delta) \rho^{2k - 1}q_{1}^{2\ell} \lPa \frac{L_\star n^{k - 1}}{(k - 1)!} \rPa \sum_{\alpha \in \sB} \sum_{v = 0}^{k - 1} k^{2v}n^{-\eps_{0}v}\,.
\end{aligned}
\]
The innermost sum is bounded by $\sum_{v = 0}^{k - 1} k^{2v}n^{-\eps_{0}v} \leq 1 + 2k^{2}n^{-\eps_{0}}$ for all $n$ sufficiently large. Taking summation over all $\alpha$ in the preceding display and using the asymptotic formula for $|\sB|$ then gives
\[
    \E[f(Y)^{2}] \leq (1 + 4^k\delta)(1 + 2k^{2}n^{-\eps_{0}}) \rho^{2k - 1}q_{1}^{2\ell}|\sB|^{2}\,.
\]
Putting everything together, 
\[
    \frac{\E[f(Y)x]^{2}}{\E[f(Y)^{2}] \E[x^{2}]} \geq \frac{1}{(1 + 4^k\delta)(1 + 2k^{2}n^{-\eps_{0}})} = 1 - o(1)\,,
\]
and the bound on the correlation follows.
\end{proof}

\appendix
\section{Equivalence of symmetric and asymmetric models}
Here we formalize the correspondence between the asymmetric and symmetric tensor PCA models.

Let $\theta \in \R^n$ be a random vector. For a parameter $\lambda \geq 0$ and $W \in (\R^n)^{\otimes r}$ a random tensor with independent $\cN(0, 1)$ entries, we define the \emph{asymmetric tensor PCA model} as $Y =\lambda \theta^{\otimes r} + W$. Let $P$ be the orthogonal projection onto the subspace of symmetric tensors: 
\[
    (PX)_{i_1,\hdots,i_r} = \frac{1}{r!} \sum_{\pi \in S_r} X_{i_{\pi(1)},\hdots,i_{\pi(r)}}, \quad \forall X \in (\R^{n})^{\otimes r}\,,
\]
where $S_r$ is the set of all permutations of $[r]$. For a given parameter $\lsy \geq 0$, we let $\Wsy = \sqrt{r!}PW$ and define the \emph{symmetric tensor PCA model} as $\Ysy = \lsy \theta^{\otimes r} + \Wsy$.

Let $x$ be a scalar estimand, which we assume is $\theta$-measurable. For the asymmetric and symmetric model, define accordingly the degree-$D$ maximum correlation
\[
    \Corr_{\leq D}(\lambda) := \sup_{f \in \R_{D}[Y]} \frac{\E[xf(Y)]}{\sqrt{\E[x^{2}] \E[f(Y)^{2}]}}\,, \quad \Corr_{\leq D}^{\mathsf{sy}}(\lsy) := \sup_{f \in \R_{D}[\Ysy]} \frac{\E[xf(\Ysy)]}{\sqrt{\E[x^2] \E[f(\Ysy)^2]}}\,,
\]
which we view as functions of $\lambda$ and $\lsy$, respectively.

\begin{lemma}\label{lem:equivalence:of:noise:models}
Under the correspondence $\lsy = \sqrt{r!} \lambda$, the degree-$D$ maximum correlation for the symmetric and asymmetric models is equal: 
\[
    \Corr_{\leq D}(\lambda) = \Corr_{\leq D}^{\mathsf{sy}}(\sqrt{r!}\lambda)\,.
\]
In particular, if $f$ is any polynomial in $Y$ depending only on the entries indexed by a set $\Lambda$, then there is a polynomial $g$ in $\Ysy$, of degree at most that of $f$, depending only on the entries indexed by $\Lambda$, such that
\[
    \frac{\E[xf(Y)]}{\sqrt{\E[x^{2}] \E[f(Y)^{2}]}} \leq \frac{\E[xg(\Ysy)]}{\sqrt{\E[x^2] \E[g(\Ysy)^2]}}\,.
\]
\end{lemma}

\begin{proof}
The constraint $\lsy = \sqrt{r!} \lambda$ implies that $\Ysy = \sqrt{r!}PY$. Since $P$ is linear, if $f$ is a polynomial of degree at most $D$, then $f \circ (\sqrt{r!}P)$ is also a polynomial of degree at most $D$. It follows that
\[
    \Corr_{\leq D}(\lambda) \geq \Corr_{\leq D}^{\mathsf{sy}}(\sqrt{r!} \lambda)\,.
\]
To prove the reverse inequality, it will suffice to prove the second claim of the lemma. Towards this end, let $f$ be a polynomial depending only on the entries indexed by $\Lambda$. Because $P$ is an orthogonal projection, the mean zero Gaussian tensor $G=(I-P)W$ is independent of $PW$. Moreover, $G$ is independent of $\theta$ whence $G$ is independent of $\Ysy$. Letting $\E_{\Ysy}$ resp. $\E_{(\Ysy, x)}$ denote the expectation conditional on $\Ysy$ resp. $(\Ysy, x)$, define
\[
    g(\Ysy) = \E\nolimits_{\Ysy}[f(\Ysy/\sqrt{r!} + G)]\,.
\]
Integrating out the independent randomness in $G$, we see that $g$ is a polynomial, whose degree is no larger than the degree of $f$, depending only on the entries indexed by $\Lambda$. Using the independence of $G$ and $(x, \Ysy)$ and also noting that $Y = \Ysy/\sqrt{r!} + G$, we obtain
\[
    \E[xg(\Ysy)] = \E[\E\nolimits_{(\Ysy, x)}[xf(\Ysy/\sqrt{r!} + G)]] = \E[\E\nolimits_{(\Ysy, x)}[xf(Y)]] = \E[xf(Y)]\,,
\]
and similarly by (conditional) Jensen's inequality, 
\[
    \E[g(\Ysy)^{2}] = \E[\E\nolimits_{\Ysy}[f(\Ysy/\sqrt{r!} + G)]^{2}] = \E[\E\nolimits_{\Ysy}[f(Y)]^2] \leq \E[f(Y)^{2}]\,.
\]
Combining the preceding two displays yields the desired bound. This concludes the proof.
\end{proof}

\section{Polynomial time algorithmic guarantee}\label{sec:polytime:algorithm}
\begin{algorithm}
\caption{Preprocessing step for symmetric sparse tensor PCA recovery.}
\label{alg:sparse:PCA:preprocessing}

\begin{algorithmic}[1]
\State \textbf{Input:} Symmetric tensor $Y$ and parameters $\rho, \lambda, \eps$.
\State Let $\lambda_\star = \sqrt{\frac{(r-2)!(1+\eps)}{en^{r-1}\rho^{2r-2}}}$ and $p=\frac{\lambda_\star}{\lambda}$. Note that $p \in [0, 1]$ under the hypothesis of Theorem~\ref{thm:sparse:PCA}-(b).
\State Independently sample a tensor $G$ with the same law as $\Wsy$. Let $Y_\star = pY + \sqrt{1-p^2}G$. Observe that $Y_\star$ has the same law as the symmetric sparse PCA model with signal parameter $\lambda_\star$.
\State \textbf{Output:} $\widetilde{Y} = Y_\star$.
\end{algorithmic}
\end{algorithm}

This section is devoted to the proof of algorithmic guarantee (Theorem~\ref{thm:algorithm}). Throughout, we shall assume that $Y$ is generated from the planted dense subhypergraph model; the same proof applies to the symmetric sparse tensor PCA model after substituting the \textbf{Preprocessing} step in Line 2 of Algorithm~\ref{alg:almost:exact:recovery} with Algorithm~\ref{alg:sparse:PCA:preprocessing}.

Recall that the estimator constructed in Section~\ref{subsec:constructing:tree:estimator} counted the number of trees of a certain structure. Let $\ell=\lceil \frac{4}{\eps}\log(1/\rho) \rceil=O_{\eps}(\log{n})$ be the same as defined in the proof of Theorem~\ref{thm:planted:hypergraph}-(b) (see~Section~\ref{subsec:constructing:tree:estimator}) and write $k=(r-1)(2\ell+2)+1$ resp. $\ell_k=(k-1)/(r-1)$ to denote the number of vertices resp. edges in any tree $\alpha \in \sT\equiv\sT_{\ell}$. Note that $k$ here differs from the notation used in Section~\ref{subsec:constructing:tree:estimator}.

We shall also assume that $\lambda = \frac{q_1-q_0}{\sqrt{q_0(1-q_0)}} = \sqrt{\frac{(r-2)!(1+\eps)}{en^{r-1}\rho^{2r-2}}}$ holds with equality, as was done in the proof of Theorems~\ref{thm:planted:hypergraph} and~\ref{thm:sparse:PCA}. We justify this assumption as follows. Let $p \in [0, 1]$ and suppose $Y$ is an instance of the planted dense subhypergraph model with parameters $\rho, q_0, q_1$. Let $Y_\star$ be the hypergraph obtained by iterating the following procedure: independently, for each $e = \{i_1,\hdots,i_r\}$, leave $Y_e$ unchanged with probability $p$ and with probability $1-p$ replace $Y_e$ with the outcome of an independent $\Ber(q_0)$ random variable. Observe that, conditional on $\theta$, 
\[
    \text{$(Y_\star)_e = 1$ with probability $p(q_0+(q_1-q_0)\theta^{V(e)}) + (1-p)q_0 = q_0 + p(q_1-q_0) \theta^{V(e)}$}
\]
In other words, $Y_\star$ has the law of an instance of the planted dense subhypergraph model with $q_1$ decreased to $q_1' = q_0+p(q_1-q_0) \in [q_0, q_1]$ and the remaining parameters unchanged. In particular, by preprocessing the edges in this way, we may assume that $q_1$, and thus $\lambda$, is as small as possible.

For each vertex $i \in [n]$, let $\sT^{(i)}$ be the set of trees obtained by relabelling $1$ as $i$. That is, letting $\pi_i$ denote the transposition exchanging vertex $1$ and $i$, we let $\sT^{(i)} = \{(\pi_i(e))_{e \in \alpha}: \alpha \in \sT\}$. Define
\[
    f^{(i)}(Y) = \sum_{\alpha \in \sT^{(i)}} \widetilde{Y}^{\alpha} \quad\textnormal{where}\quad\widetilde{Y}=\frac{Y-q_0}{\sqrt{q_0(1-q_0)}}\,,
\]
and observe that, because the distribution of the planted subhypergraph model is invariant under vertex relabeling, Theorem~\ref{thm:planted:hypergraph}-(b) implies $\Corr \lPa f^{(i)}(Y), \theta_{i} \rPa = \frac{|\E[\theta_i f^{(i)}(Y)]|}{\sqrt{\E[\theta_i^2] \E[f^{(i)}(Y)^2]}} = 1-o(1)$.

\begin{definition}
A coloring is any mapping $c: [n] \rightarrow [k]$.
\end{definition}

The idea is to replace the tree counts by a color-coded version, obtained by randomly coloring the vertices $[n]$ and counting only those trees $\alpha \in \sT^{(i)}$ whose vertices receive distinct colors. Let us fix an integer $t \geq 1$, whose exact value will be determined later (see the proof of Proposition~\ref{prop:single:vertex:recovery} below). Let $c_{1}, \hdots, c_{t}: [n] \rightarrow [k]$ be independent, uniformly random colorings, chosen independently of $(Y, \theta)$ and define $\xi_{s}(\alpha) = \In_{\text{$c_{s}$ is injective on $V(\alpha)$}}$. Compute, 
\[
	\E[\xi_{s}(\alpha)] = \Pb(\xi_{s}(\alpha) = 1) = \Pb(\text{$c_{s}$ is injective on $[k]$}) = \frac{k!}{k^{k}} =: q
\]
for all $s \in [t]$ and $\alpha \in \sT^{(i)}$. Next, for each vertex $i \in [n]$, we define the single color-coded estimator
\[
	f_{s}^{(i)}(Y) = \frac{1}{q}\sum_{\alpha \in \sT^{(i)}} \widetilde{Y}^{\alpha} \xi_{s}(\alpha)\,.
\]
Letting $\E_{Y}$ denote the expectation conditional on $Y$, we calculate
\begin{equation}
	\E\nolimits_{Y}[f_{s}^{(i)}(Y)] = \frac{1}{q}\sum_{\alpha \in \sT^{(i)}} \widetilde{Y}^{\alpha} \cdot \E\nolimits_{Y}[\xi_{s}(\alpha)] = \sum_{\alpha \in \sT^{(i)}} \widetilde{Y}^{\alpha} = f^{(i)}(Y)\,.
\label{eq:single:coloring:unbiased}
\end{equation}
Finally, we define our candidate estimator: 
\[
	g^{(i)}(Y) = \frac{1}{t} \sum_{s = 1}^{t} f_{s}^{(i)}(Y)\,.
\]
Taking summation over $s=1,\hdots,t$ in~\eqref{eq:single:coloring:unbiased} shows that $\E_{Y}[g^{(i)}(Y)] = f^{(i)}(Y)$. Furthermore, 

\begin{lemma}\label{lem:correlation:transfer}
If $tq \rightarrow \infty$, then $\Corr \lPa g^{(i)}(Y), \theta_i \rPa = \frac{|\E[\theta_i g^{(i)}(Y)]|}{\sqrt{\E[\theta_i^2] \E[g^{(i)}(Y)^2]}} = 1 - o(1)$.
\end{lemma}

\begin{proof}
By exchangeability, we may assume w.l.o.g. that $i=1$. Let $f \equiv f^{(1)}$ and $g \equiv g^{(1)}$. To reduce notation, we omit the dependence on $Y$ from the polynomials $f$ and $g$. It will suffice to show that
\begin{equation}
	\E[\Delta^{2}] = o(\E[f^{2}])
\label{eq:correlation:transfer:sufficient:condition}
\end{equation}
where $\Delta=f-g$. Indeed, observe that $|\E[\theta_1f]|=(1-o(1))\sqrt{\E[f^2] \E[\theta_1^2]}$, so if~\eqref{eq:correlation:transfer:sufficient:condition} holds, then the triangle inequality and Cauchy Schwarz yield, 
\[
	\frac{|\E[\theta_1g]|}{\sqrt{\E[g^{2}] \E[\theta_1^{2}]}} \geq \frac{|\E[\theta_1f]| - |\E[\theta_1\Delta]|}{\lPa \sqrt{\E[f^{2}]} + \sqrt{\E[\Delta^{2}]} \rPa \sqrt{\E[\theta_1^{2}]}} \geq \frac{(1 - o(1)) \sqrt{\E[f^{2}]} - \sqrt{\E[\Delta^{2}]}}{(1 + o(1)) \sqrt{\E[f^{2}]}} = 1 - o(1)\,.
\]
We thus turn our attention to proving~\eqref{eq:correlation:transfer:sufficient:condition}. Observe that, 
\[
    \E[g^{2}] = \E \lBr \frac{1}{t^2q^2} \sum_{s, s' = 1}^{t} \sum_{\alpha, \beta \in \sT} \widetilde{Y}^{\alpha + \beta} \xi_{s}(\alpha) \xi_{s'}(\beta) \rBr = \frac{1}{t^2q^2} \sum_{s, s' = 1}^{t} \sum_{\alpha, \beta \in \sT} \E[\widetilde{Y}^{\alpha + \beta}] \cdot \E[\xi_{s}(\alpha) \xi_{s'}(\beta)]\,,
\]
where the last equality used the independence of the colorings and the planted dense subhypergraph model. If $s \neq s'$, then $\E[\xi_{s}(\alpha) \xi_{s'}(\beta)] = q^{2}$; otherwise if $s=s'$, then $\E[\xi_{s}(\alpha) \xi_{s'}(\beta)] \leq \E[\xi_{s}(\alpha)] = q$ by non-negativity. Using that the mixed moments $\E[\widetilde{Y}^{\alpha + \beta}]$ are non-negative (refer to the calculations in Section~\ref{subsec:planted:hypergraph:upper:bound}), 
\[
    \E[g^{2}] \leq \frac{t(t - 1)q^{2} + tq}{t^2q^2} \sum_{\alpha, \beta \in \sT} \E[\widetilde{Y}^{\alpha + \beta}] = \lPa 1 + \frac{1 - q}{tq} \rPa \E[f^{2}]\,.
\]
Rearranging and using $\E[fg] = \E[f^2]$ since $\E_Y[g] = f$, we obtain
\[
    \E[\Delta^{2}] = \E[g^{2}] - 2\E[fg] + \E[f^{2}] = \E[g^{2}] - \E[f^{2}] \leq \lPa \frac{1 - q}{tq} \rPa \E[f^{2}] = o(\E[f^{2}])\,,
\]
where the last step holds provided $tq \rightarrow \infty$.
\end{proof}

\begin{proposition}\label{prop:single:coloring:polytime:computable}
For all $s \in [t]$, the vector $(f_s^{(i)}(Y))_{i \in [n]}$ is computable in time $n^{r}e^{O_r(\ell)}$.
\end{proposition}

\begin{proof}
Let $c: [n] \rightarrow [k]$ be any coloring. We say that two graphs $\alpha_{1}, \alpha_{2} \in \sT^{(i)}$ are isomorphic if there exists a (root-preserving) hypergraph isomorphism $\varphi: V(\alpha_{1}) \rightarrow V(\alpha_{2})$ with $\varphi(i) = i$. Let $\sH$ be the set of isomorphism classes of $\sT$. For every class $H \in \sH$, fix a representative and write $r_H$ for its root. Let $\Aut(H)$ be the group of root-preserving automorphisms of $H$ and define
\[
	X_{H}(i, Y) = \sum_{\substack{\varphi: V(H) \hookrightarrow [n] \\ \varphi(r_H) = i}} \widetilde{Y}^{\varphi(H)} \In_{\text{$c$ is injective on $\varphi(V(H))$}}\,.
\]
Expanding $\sT^{(i)}$ over its automorphism classes, 
\begin{equation}\label{eq:automorphism:class:expansion}
    \sum_{\alpha \in \sT^{(i)}} \widetilde{Y}^{\alpha} \In_{\text{$c$ is injective on $V(\alpha)$}} = \sum_{H \in \sH} \frac{X_{H}(i, Y)}{|\Aut(H)|}\,.
\end{equation}
We show that we can compute the RHS in polynomial time. We do this by bounding the size of $\sH$ and showing that the vector $(X_{H}(i, Y))_{i \in [n]}$ can be computed in polynomial time independent of the chosen representative $H \in \sH$.

{\bf{Total number of isomorphism classes.}}
For each representative hypertree $H \in \sH$, map it to its rooted incidence bipartite graph, viewed up to rooted isomorphism as a rooted tree with $k + \ell_k$ vertices. Note that this mapping is injective from $\sH$ to the set of isomorphism classes of rooted trees. Hence, the number of distinct classes in $\sH$ is upper bounded by the number of unlabeled rooted trees with $k+\ell_k$ vertices. The latter is bounded by the total number of plane trees with $k + \ell_k$ vertices. Now since the number of plane trees with $m$ vertices can be bounded by the $m$-th Catalan number, which is bounded by $4^m$, we conclude that
\begin{equation}\label{eq:isomorphism:class:counting}
	|\sH| \leq 4^{k + \ell_k} = e^{O_{r}(k)}\,.
\end{equation}

{\bf{Computation of $X_{H}$.}}
Fix any representative hypertree $H \in \sH$. By identifying $H$ with its bipartite incidence graph, we view $H$ as a rooted tree (in the usual sense with $r=2$) on the vertex set $V(H) \sqcup E(H)$. For a vertex-node $v \in V(H)$, let $H_{v}$ be the subtree rooted at $v$ consisting of $v$ and its descendants in the incidence tree. Similarly, for an edge-node $e \in E(H)$, write $p_{e}$ to denote its parent vertex and let $H_{e}$ be the subgraph consisting of $p_{e}$ along with $e$ and its descendants. For any node $a$ with subtree $H_a$, write $\vtx(H_a)$ to denote the set of vertex-nodes in $H_a$. If $a$ is an edge-node, this excludes its parent vertex-node.

For vertex labels $i \in [n]$ and colors $Q \subseteq [k]$, define the dynamic programming (DP) quantities
\[
	A_{v}(i, Q) = \sum_{\substack{\varphi: V(H_{v}) \hookrightarrow [n] \\ \varphi(v) = i}} \widetilde{Y}^{\varphi(H_{v})} \In_{\text{$c$ maps $\varphi(\vtx(H_{v}))$ bijectively to $Q$}}
\]
and
\[
	B_{e}(i, Q) = \sum_{\substack{\varphi: V(H_{e}) \hookrightarrow [n] \\ \varphi(p_{e}) = i}} \widetilde{Y}^{\varphi(H_{e})} \In_{c(i) \notin Q} \In_{\text{$c$ maps $\varphi(\vtx(H_{e}) \setminus \{p_e\})$ bijectively to $Q$}}\,.
\]
The dynamic programming equations are as follows: 
\begin{itemize}
	\item[1.] If $v$ is a leaf (has no children), then $A_{v}(i, Q) = \In_{Q = \{c(i)\}}$. Otherwise, for a vertex-node $v$ with child edge-nodes $e_{1}, \hdots, e_{d}$ the vertex DP equation is
	\begin{equation}
		A_{v}(i, Q) = \In_{c(i) \in Q} \sum_{\substack{Q_{1}, \hdots, Q_{d} \subseteq [k] \\ \bigsqcup_{j = 1}^{d} Q_{j} = Q \setminus \{c(i)\}}} \prod_{j = 1}^{d} B_{e_{j}}(i, Q_{j})\,.
    \label{eq:vertex:DP}
	\end{equation}
	To simplify the calculation of the computation cost, it will be useful to consider the function, 
	\[
		h_{j}(i, Q) = \sum_{Q' \subseteq Q} h_{j - 1}(i, Q') B_{e_{j}}(i, Q \setminus Q'), \quad \forall j \in [d]
	\]
	with base-case $h_{0}(i, Q) = \In_{Q = \emptyset}$. Then
	\begin{equation}
		A_{v}(i, Q) = \begin{cases}
			h_{d}(i, Q \setminus \{c(i)\}) & \text{if $c(i) \in Q$} \\
			0					   & \text{otherwise}
		\end{cases}\,.
    \label{eq:subset:convolution}
	\end{equation}
	The point here is that we have rewritten Eq.~\eqref{eq:vertex:DP} in terms of subset convolutions.
	\item[2.] For an edge-node $e$ with child vertex-nodes $v_{1}, \hdots, v_{r - 1}$ the edge DP equation is
	\begin{equation}
		B_{e}(i, Q) = \In_{c(i) \notin Q} \sum_{\substack{y_{1}, \hdots, y_{r - 1} \in [n] \setminus \{i\} \\ \text{distinct}}} \widetilde{Y}_{\{i, y_{1}, \hdots, y_{r - 1}\}} \sum_{\substack{Q_{1}, \hdots, Q_{r - 1} \subseteq [k] \\ \bigsqcup_{j = 1}^{r - 1} Q_{j} = Q}} \prod_{j = 1}^{r - 1} A_{v_{j}}(y_{j}, Q_{j})\,.
    \label{eq:edge:DP}
	\end{equation}
\end{itemize}
Finally, we observe that $X_{H}(i, Y) = A_{r_H}(i, [k])$. It remains to bound the runtime of computing the dynamic programming (DP) equations.
\begin{itemize}
	\item[1.] For a vertex-node $v \in V(H)$ with $d_{v}$ children edge-nodes and a vertex label $i \in [n]$, the cost of computing Eq.~\eqref{eq:subset:convolution} by summing over subsets of $[k]$ is at most, 
	\[
		\sum_{j = 1}^{d_{v}} \sum_{Q \subseteq [k]} 2^{|Q|} = d_{v}3^{k}\,.
	\]
	Summing over all indices and vertex-nodes, the cost of computing the vertex DP equations is no more than, 
	\[
		\sum_{v \in V(H)} \sum_{i \in [n]} d_{v}3^{k} = n3^{k} \sum_{v \in V(H)} d_{v} = n\ell_k 3^{k}\,,
	\]
	where we used that the sum of the number of children of all vertex-nodes is simply the number of edge-nodes.
	\item[2.] For an edge-node $e$, the cost of computing Eq.~\eqref{eq:edge:DP} for a fixed pair $(i, Q) \in [n] \times 2^{[k]}$ is at most
	\[
		n^{r - 1}(r - 1)^{|Q|}\,,
	\]
	where the first term counts the number of ways to choose the $y_{j}$'s and the second factor counts the ways to partition the color set $Q$. Summing over all edge-nodes and pairings yields, 
	\[
		\sum_{e \in E(H)} \sum_{(i, Q) \in [n] \times 2^{[k]}} n^{r - 1}(r - 1)^{|Q|} \leq \sum_{e \in E(H)} n^{r}(2r)^{k} = n^{r} \ell_k (2r)^{k}
	\]
	as an upper bound for the runtime of the edge DP equations.
\end{itemize}
Combining the preceding bounds shows that $(X_{H}(i, Y))_{i \in [n]}$ can be computed in time, 
\begin{equation}\label{eq:DP:time}
	n \ell_k 3^{k} + n^{r} \ell_k (2r)^{k} = n^{r}e^{O_r(k)}\,.
\end{equation}
To conclude the proof, observe that the LHS of~\eqref{eq:automorphism:class:expansion} is precisely equal to $qf_{s}^{(i)}(Y)$ for $c = c_s$, where $q$ is a multiplicative scaling, and the estimates~\eqref{eq:isomorphism:class:counting} and~\eqref{eq:DP:time} imply that the RHS of~\eqref{eq:automorphism:class:expansion} can be computed in time $n^{r}e^{O_r(\ell)}$, using also that $k=O_r(\ell)$.
\end{proof}

From the preceding, we deduce the following result.

\begin{proposition}\label{prop:single:vertex:recovery}
Consider the planted dense subhypergraph model. Suppose the hypotheses of Theorem~\ref{thm:planted:hypergraph}-(b) holds for a fixed $\eps>0$. There exists a randomized algorithm, whose runtime is at most $n^{r + o(1)}e^{O_r(\ell)}$, such that given input $Y$ and parameters $\rho, q_0,q_1$, outputs a vector $\hat{x}=(\hat{x}_i)_{i \in [n]}$ such that $\Corr(\hat{x}_i, \theta_{i}) = 1 - o(1)$ for all $i \in [n]$.
\end{proposition}

\begin{proof}
Let $\hat{x}_{i}=g^{(i)}(Y)$. Take $t = \lceil q^{-1} \log{n} \rceil$ and $\ell = \lceil \frac{4}{\eps} \log(1/\rho) \rceil$ so that $tq \geq \log{n} \rightarrow \infty$ and the result of Lemma~\ref{lem:correlation:transfer} applies. Then, 
\[
	t \leq q^{-1} \log{n} + 1 \leq e^{k} \log{n} + 1 = n^{o(1)}e^{O_r(\ell)}\,.
\]
We see that $(g^{(i)}(Y))_{i \in [n]}$ is a sum of $n^{o(1)}e^{O_r(\ell)}$-many vectors, each of which can be computed in time $n^{r}e^{O_r(\ell)}$ according to Proposition~\ref{prop:single:coloring:polytime:computable}. It follows that $\hat{x}$ can be computed within the stated runtime.
\end{proof}

\begin{proof}[Proof of Algorithmic Guarantee Theorem~\ref{thm:algorithm}]
As noted at the start of this section, we shall assume that $Y$ is generated from the planted dense subhypergraph model. For $i \in [n]$, take $\hat{x}_{i}$ to be the estimator of $\theta_{i}$ given by Proposition~\ref{prop:single:vertex:recovery}. Define
\[
    a_{n} := \frac{\E[\theta_{i}\hat{x}_i]}{\E[\theta_i^2]}\,.
\]
Note that this definition does not depend on the chosen vertex $i$ because the laws of the random variables $\{(\theta_{i}, \hat{x}_{i})\}_{i \in [n]}$ are exchangeable. In particular, we have $a_n>0$ since $\E[\theta_1\hat{x}_1] = \E[\theta_1f^{(1)}(Y)] = |\sT| \lambda^{\ell_k}\rho^{k} > 0$. Compute, 
\[
\begin{aligned}
    \E[(\hat{x}_i  - a_{n} \theta_{i})^{2}] &= \E[\hat{x}_i^{2}] \lPa 1 - \frac{2a_{n} \E[\theta_{i} \hat{x}_i] - a_{n}^{2} \E[\theta_{i}^{2}]}{\E[\hat{x}_i^{2}]} \rPa = \E[\hat{x}_i^{2}] \lPa 1 - \frac{\E[\theta_{i} \hat{x}_i]^{2}}{\E[\theta_{i}^{2}] \E[\hat{x}_i^{2}]} \rPa\,.
\end{aligned}
\]
Proposition~\ref{prop:single:vertex:recovery} shows that the RHS of the above display is $o(\E[\hat{x}_i^2])$. On the other hand, Proposition~\ref{prop:single:vertex:recovery} also implies, 
\[
    \E[\hat{x}_i^{2}] = (1 + o(1)) \frac{\E[\theta_{i} \hat{x}_i]^{2}}{\E[\theta_{i}^{2}]} = (1 + o(1))a_{n}^{2} \rho\,.
\]
Thus, letting $z_{i} := \hat{x}_i/a_{n}$ denote the normalized score, the preceding estimates imply that
\[
    \delta_n^2 := \E \lBr \frac{1}{n\rho} \sum_{i = 1}^{n}(z_i - \theta_{i})^{2} \rBr = \frac{\E[(z_1-\theta_1)^2]}{\rho} = \frac{\E[(\hat{x}_1 - a_n\theta_1)^2]}{a_n^2\rho} = o \lPa \frac{\E[\hat{x}_1^2]}{a_n^2\rho} \rPa = o(1)\,.
\]
Therefore, by Markov's inequality, 
\begin{equation}
    \frac{1}{n\rho} \sum_{i = 1}^{n}(z_i - \theta_{i})^{2} \leq \delta_n
\label{eq:alg:markov:bound}
\end{equation}
with probability at least $1 - \delta_n$. From here on out, we assume the high probability event on which the estimate in~\eqref{eq:alg:markov:bound} holds. For a vector $\gamma \in \R^{n}$, let $T_{s}(\gamma)$ denote the set of $s$ largest indices of $\gamma$, breaking ties by decreasing order of index. Take as our candidate estimator, 
\[
    \widehat{S} = T_{\lfloor n\rho \rfloor}(\hat{x})\,.
\]
Because $\hat{x}$ differs from the normalized score by a global scaling factor, we have $T_{s}(\hat{x}) = T_{s}(z)$ for all $s \in [n]$, where $z = (z_1, \hdots, z_n)$. We claim that $|\widehat{S} \triangle S|=o(n\rho)$ with probability $1-o(1)$. Towards this end, define $S' = T_{|S|}(z)$ and observe that $|S| = |S'|$ implies $|S \setminus S'| = |S' \setminus S|$. Take a matching between the sets $S \setminus S'$ and $S' \setminus S$. That is, pair up each false negative $i \in S \setminus S'$ with a false positive $j \in S' \setminus S$ such that every false negative (and false positive) is matched exactly once. Then
\[
    (z_{i} - \theta_{i})^{2} + (z_j - \theta_{j})^{2} = (z_{i} - 1)^{2} + z_{j}^{2} = \lPa z_{i} - 1/2 \rPa^{2} + \lPa z_{j} - 1/2 \rPa^{2} + (z_{j} - z_{i}) + 1/2 \geq 1/2\,,
\]
where the last inequality holds since $z_{j} \geq z_{i}$ by definition. Taking summation over all pairs gives
\[
    |S' \triangle S| = 2|S \setminus S'| \leq 4 \left \{\sum_{(i, j)} (z_{i} - \theta_{i})^{2} + (z_{j} - \theta_{j})^{2} \right \} \leq 4 \sum_{i = 1}^{n} (z_{i} - \theta_{i})^{2} \leq 4\delta_n n\rho\,,
\]
where the last inequality holds assuming the bound~\eqref{eq:alg:markov:bound}. Next, we argue $S'$ is close to $\widehat{S}$ with high probability. Because $|S| = \sum_{i = 1}^{n} \theta_{i}$ is a sum of independent $\Ber(\rho)$ random variables, Chernoff's inequality (cf.~\cite[Theorem 2.3.1 and Exercise 2.3.6]{vershynin-notes}) implies that
\[
    \P \lPa \frac{||S|- n\rho|}{n\rho} \leq \tau \rPa \geq 1 - 2e^{-\tau^2n\rho/C}\,,
\]
for any $\tau \in [0, 1]$ and where $C>0$ is an absolute constant. Taking $\tau\equiv\tau_n = \sqrt{\frac{C\log{n}}{n\rho}}$, which is $o(1)$ since $n\rho=\omega(\log{n}) \rightarrow \infty$ by assumption, we find that $||S|-\lfloor n\rho \rfloor| \leq \tau n\rho + 1$ with probability at least $1-2/n$. In particular, there exists a deterministic sequence $\eps_n \rightarrow 0$ such that
\begin{equation}
    \frac{||S| - \lfloor n\rho \rfloor|}{n\rho} \leq \eps_n
\label{eq:alg:chernoff:bound}
\end{equation}
with probability at least $1 - \eps_n$. Working on the intersection of the events for which~\eqref{eq:alg:chernoff:bound} and~\eqref{eq:alg:markov:bound} hold, we finally conclude that
\[
    |\widehat{S} \triangle S| \leq |\widehat{S} \triangle S'| + |S' \triangle S| = ||S| - \lfloor n\rho \rfloor| + |S' \triangle S| \leq (4\delta_n + \eps_n) n\rho
\]
with probability at least $1-\delta_n-\eps_n$, where we used the identity $|\widehat{S} \triangle S'| = ||S| - \lfloor n\rho \rfloor|$, which holds because either one of $S'$ or $\widehat{S}$ is a subset of the other by construction. This proves the almost-exact recovery guarantee.

It remains to bound the runtime. In this procedure, the runtime of computing $\hat{x}$ is $n^{r+o(1)}e^{O_{r}(\ell)}$ as given by Proposition~\ref{prop:single:vertex:recovery}. Because $r \geq 2$, the calculation of the score and sorting are negligible relative to this runtime. This completes the proof.
\end{proof}

\bibliographystyle{alpha}
\bibliography{main}

@article{chen2019phase,
  title={Phase transition in the spiked random tensor with {Rademacher} prior},
  author={Chen, Wei-Kuo},
  journal={The Annals of Statistics},
  volume={47},
  number={5},
  pages={2734--2756},
  year={2019},
  publisher={JSTOR}
}

@inproceedings{lesieur2017statistical,
  title={Statistical and computational phase transitions in spiked tensor estimation},
  author={Lesieur, Thibault and Miolane, L{\'e}o and Lelarge, Marc and Krzakala, Florent and Zdeborov{\'a}, Lenka},
  booktitle={2017 IEEE International Symposium on Information Theory (ISIT)},
  pages={511--515},
  year={2017},
  organization={IEEE}
}

@article{carpentier2025low,
  title={Low-degree lower bounds via almost orthonormal bases},
  author={Carpentier, Alexandra and Giancola, Simone Maria and Giraud, Christophe and Verzelen, Nicolas},
  journal={arXiv preprint arXiv:2509.09353},
  year={2025}
}

@article{carpentier2025phase,
  title={Phase transition for stochastic block model with more than $\sqrt{n}$ communities},
  author={Carpentier, Alexandra and Giraud, Christophe and Verzelen, Nicolas},
  journal = {arXiv preprint arXiv:2509.15822},
  year={2025}
}

@InProceedings{luo20a,
  title = 	 {Open Problem: Average-Case Hardness of Hypergraphic Planted Clique Detection},
  author =       {Luo, Yuetian and Zhang, Anru R},
  booktitle = 	 {Proceedings of Thirty Third Conference on Learning Theory},
  pages = 	 {3852--3856},
  year = 	 {2020},
  editor = 	 {Abernethy, Jacob and Agarwal, Shivani},
  volume = 	 {125},
  series = 	 {Proceedings of Machine Learning Research},
  month = 	 {09--12 Jul},
  publisher =    {PMLR}
}

@article{tang2026detection,
  title={Detection Is Harder Than Estimation in Certain Regimes: Inference for Moment and Cumulant Tensors},
  author={Tang, Runshi and Han, Yuefeng and Zhang, Anru R},
  journal={arXiv preprint arXiv:2603.26029},
  year={2026}
}

@article{choo2021complexity,
  title = {The complexity of sparse tensor {PCA}},
  author = {Choo, Davin and d'Orsi, Tommaso},
  journal = {Advances in Neural Information Processing Systems},
  volume = {34},
  pages = {7993--8005},
  year = {2021}
}

@article{mao2024testing,
  title = {Testing network correlation efficiently via counting trees},
  author = {Mao, Cheng and Wu, Yihong and Xu, Jiaming and Yu, Sophie H},
  journal = {The Annals of Statistics},
  volume = {52},
  number = {6},
  pages = {2483--2505},
  year = {2024},
  publisher = {Institute of Mathematical Statistics}
}

@article{macris2020all,
  title = {All-or-nothing statistical and computational phase transitions in sparse spiked matrix estimation},
  author = {Macris, Nicolas and Rush, Cynthia and others},
  journal = {Advances in Neural Information Processing Systems},
  volume = {33},
  pages = {14915--14926},
  year = {2020}
}

@article{niles2020all,
  title = {The all-or-nothing phenomenon in sparse tensor {PCA}},
  author = {Niles-Weed, Jonathan and Zadik, Ilias},
  journal = {Advances in Neural Information Processing Systems},
  volume = {33},
  pages = {17674--17684},
  year = {2020}
}

@inproceedings{mossel2023sharp,
  title = {Sharp thresholds in inference of planted subgraphs},
  author = {Mossel, Elchanan and Niles-Weed, Jonathan and Sohn, Youngtak and Sun, Nike and Zadik, Ilias},
  booktitle = {The Thirty Sixth Annual Conference on Learning Theory},
  pages = {5573--5577},
  year = {2023},
  organization = {PMLR}
}

@inproceedings{buhai25false,
  author = {Buhai, Rares-Darius and Hsieh, Jun-Ting and Jain, Aayush and Kothari, Pravesh K.},
  booktitle = {2025 IEEE 66th Annual Symposium on Foundations of Computer Science (FOCS)},
  title = {The quasi-polynomial low-degree conjecture is false}, 
  year = {2025},
  volume = {},
  number = {},
  pages = {2577-2590},
  doi = {10.1109/FOCS63196.2025.00134}}

@inproceedings{corinzia22statistical,
  title = {Statistical and computational thresholds for the planted $k$-densest sub-hypergraph problem },
  author = {Corinzia, Luca and Penna, Paolo and Szpankowski, Wojciech and Buhmann, Joachim},
  booktitle = {Proceedings of The 25th International Conference on Artificial Intelligence and Statistics},
  pages = {11615--11640},
  year = {2022},
  editor = {Camps-Valls, Gustau and Ruiz, Francisco J. R. and Valera, Isabel},
  volume = {151},
  series = {Proceedings of Machine Learning Research},
  month = {28--30 Mar},
  publisher = {PMLR},
  url = {https://proceedings.mlr.press/v151/corinzia22a.html}
}

@book{sze-book,
  title = {Orthogonal polynomials},
  author = {G\'{a}bor Szeg\"{o}},
  publisher = {American Mathematical Soc.},
  volume = {23},
  year = {1975}
}

@article{CMSW25,
  title = {Stochastic block models with many communities and the {Kesten--Stigum} bound},
  author = {Chin, Byron and Mossel, Elchanan and Sohn, Youngtak and Wein, Alexander S},
  journal = {arXiv preprint arXiv:2503.03047},
  year = {2025}
}

@article {KestenStigum:66,
  AUTHOR = {H. Kesten and B. P. Stigum},
  TITLE = {Additional limit theorems for indecomposable multidimensional
              {G}alton-{W}atson processes},
  JOURNAL = {Ann. Math. Statist.},
  VOLUME = {37},
  YEAR = {1966},
  PAGES = {1463--1481},
}

@inproceedings{HS-bayesian,
  title = {Efficient Bayesian estimation from few samples: community detection and related problems},
  author = {Hopkins, Samuel B and Steurer, David},
  booktitle = {58th Annual Symposium on Foundations of Computer Science (FOCS)},
  pages = {379--390},
  year = {2017},
  organization = {IEEE}
}

@inproceedings{BR-reduction,
  title = {Complexity theoretic lower bounds for sparse principal component detection},
  author = {Berthet, Quentin and Rigollet, Philippe},
  booktitle = {Conference on Learning Theory},
  pages = {1046--1066},
  year = {2013},
  organization = {PMLR}
}

@inproceedings{BBH-reduction,
  title = {Reducibility and computational lower bounds for problems with planted sparse structure},
  author = {Brennan, Matthew and Bresler, Guy and Huleihel, Wasim},
  booktitle = {Conference on Learning Theory},
  pages = {48--166},
  year = {2018},
  organization = {PMLR}
}

@article{jerrum,
  title = {Large cliques elude the Metropolis process},
  author = {Jerrum, Mark},
  journal = {Random Structures \& Algorithms},
  volume = {3},
  number = {4},
  pages = {347--359},
  year = {1992},
  publisher = {Wiley Online Library}
}

@article{decelle,
  title = {Asymptotic analysis of the stochastic block model for modular networks and its algorithmic applications},
  author = {Decelle, Aurelien and Krzakala, Florent and Moore, Cristopher and Zdeborov{\'a}, Lenka},
  journal = {Physical Review E},
  volume = {84},
  number = {6},
  pages = {066106},
  year = {2011},
  publisher = {APS}
}

@article{ogp-survey,
  title = {The overlap gap property: A topological barrier to optimizing over random structures},
  author = {Gamarnik, David},
  journal = {Proceedings of the National Academy of Sciences},
  volume = {118},
  number = {41},
  pages = {e2108492118},
  year = {2021},
  publisher = {National Acad Sciences}
}

@article{sq-clique,
  title = {Statistical algorithms and a lower bound for detecting planted cliques},
  author = {Feldman, Vitaly and Grigorescu, Elena and Reyzin, Lev and Vempala, Santosh S and Xiao, Ying},
  journal = {Journal of the ACM},
  volume = {64},
  number = {2},
  pages = {1--37},
  year = {2017},
  publisher = {ACM New York, NY, USA}
}

@article{sos-clique,
  title = {A nearly tight sum-of-squares lower bound for the planted clique problem},
  author = {Barak, Boaz and Hopkins, Samuel and Kelner, Jonathan and Kothari, Pravesh K and Moitra, Ankur and Potechin, Aaron},
  journal = {SIAM Journal on Computing},
  volume = {48},
  number = {2},
  pages = {687--735},
  year = {2019},
  publisher = {SIAM}
}

@inproceedings{sos-survey,
  title = {High dimensional estimation via sum-of-squares proofs},
  author = {Raghavendra, Prasad and Schramm, Tselil and Steurer, David},
  booktitle = {Proceedings of the International Congress of Mathematicians: Rio de Janeiro 2018},
  pages = {3389--3423},
  year = {2018},
  organization = {World Scientific}
}

@inproceedings{sos-detect,
  title = {The power of sum-of-squares for detecting hidden structures},
  author = {Hopkins, Samuel B and Kothari, Pravesh K and Potechin, Aaron and Raghavendra, Prasad and Schramm, Tselil and Steurer, David},
  booktitle = {58th Annual Symposium on Foundations of Computer Science (FOCS)},
  pages = {720--731},
  year = {2017},
  organization = {IEEE}
}

@phdthesis{hopkins-thesis,
  Author = {Hopkins, Samuel},
  School = {Cornell University},
  Title = {Statistical inference and the sum-of-squares method},
  Year = {2018}
}

@inproceedings{ld-notes,
  title = {Notes on computational hardness of hypothesis testing: Predictions using the low-degree likelihood ratio},
  author = {Kunisky, Dmitriy and Wein, Alexander S and Bandeira, Afonso S},
  booktitle = {ISAAC Congress (International Society for Analysis, its Applications and Computation)},
  pages = {1--50},
  year = {2019},
  organization = {Springer}
}

@inproceedings{sq-ld,
  title = {Statistical query algorithms and low degree tests are almost equivalent},
  author = {Brennan, Matthew S and Bresler, Guy and Hopkins, Sam and Li, Jerry and Schramm, Tselil},
  booktitle = {Conference on Learning Theory},
  pages = {774--774},
  year = {2021},
  organization = {PMLR}
}

@article{fp,
  title = {The {Franz-Parisi} criterion and computational trade-offs in high dimensional statistics},
  author = {Bandeira, Afonso S and {El Alaoui}, Ahmed and Hopkins, Samuel and Schramm, Tselil and Wein, Alexander S and Zadik, Ilias},
  journal = {Advances in Neural Information Processing Systems},
  volume = {35},
  pages = {33831--33844},
  year = {2022}
}

@article{RM-tensor-pca,
  title = {A statistical model for tensor {PCA}},
  author = {Richard, Emile and Montanari, Andrea},
  journal = {Advances in Neural Information Processing Systems},
  volume = {27},
  year = {2014}
}

@inproceedings{tensor-pca-sos,
  title = {Tensor principal component analysis via sum-of-square proofs},
  author = {Hopkins, Samuel B and Shi, Jonathan and Steurer, David},
  booktitle = {Conference on Learning Theory},
  pages = {956--1006},
  year = {2015},
  organization = {PMLR}
}

@article{subexp-sparse,
  title = {Subexponential-time algorithms for sparse {PCA}},
  author = {Ding, Yunzi and Kunisky, Dmitriy and Wein, Alexander S and Bandeira, Afonso S},
  journal = {Foundations of Computational Mathematics},
  pages = {1--50},
  year = {2023},
  publisher = {Springer}
}

@article{BGJ-tensor,
  title = {Algorithmic thresholds for tensor {PCA}},
  author = {{Ben Arous}, Gerard and Gheissari, Reza and Jagannath, Aukosh},
  journal = {The Annals of Probability},
  volume = {48},
  number = {4},
  pages = {2052--2087},
  year = {2020},
  publisher = {Institute of Mathematical Statistics}
}

@article{submatrix-ogp,
  title = {The overlap gap property in principal submatrix recovery},
  author = {Gamarnik, David and Jagannath, Aukosh and Sen, Subhabrata},
  journal = {Probability Theory and Related Fields},
  volume = {181},
  number = {4},
  pages = {757--814},
  year = {2021},
  publisher = {Springer}
}

@inproceedings{kikuchi,
  title = {The {Kikuchi} hierarchy and tensor {PCA}},
  author = {Wein, Alexander S and {El Alaoui}, Ahmed and Moore, Cristopher},
  booktitle = {2019 IEEE 60th Annual Symposium on Foundations of Computer Science (FOCS)},
  pages = {1446--1468},
  year = {2019},
  organization = {IEEE}
}

@article{GZ-clique,
  title = {The landscape of the planted clique problem: Dense subgraphs and the overlap gap property},
  author = {Gamarnik, David and Zadik, Ilias},
  journal = {arXiv preprint arXiv:1904.07174},
  year = {2019}
}

@inproceedings{grp-testing,
  title = {Statistical and computational phase transitions in group testing},
  author = {{Coja-Oghlan}, Amin and Gebhard, Oliver and {Hahn-Klimroth}, Max and Wein, Alexander S and Zadik, Ilias},
  booktitle = {Conference on Learning Theory},
  pages = {4764--4781},
  year = {2022},
  organization = {PMLR}
}

@article{LZ-tensor,
  title = {Tensor clustering with planted structures: Statistical optimality and computational limits},
  author = {Luo, Yuetian and Zhang, Anru R},
  journal = {The Annals of Statistics},
  volume = {50},
  number = {1},
  pages = {584--613},
  year = {2022},
  publisher = {Institute of Mathematical Statistics}
}

@article{BBP,
  title = {Phase transition of the largest eigenvalue for nonnull complex sample covariance matrices},
  author = {Baik, Jinho and {Ben Arous}, G{\'e}rard and P{\'e}ch{\'e}, Sandrine},
  journal = {The Annals of Probability},
  pages = {1643--1697},
  year = {2005}
}

@article{fund-limits-wigner,
  title = {Fundamental limits of detection in the spiked {Wigner} model},
  author = {{El Alaoui}, Ahmed and Krzakala, Florent and Jordan, Michael I},
  journal = {The Annals of Statistics},
  volume = {48},
  number = {2},
  pages = {863--885},
  year = {2020}
}

@article{amp,
  title = {Message-passing algorithms for compressed sensing},
  author = {Donoho, David L and Maleki, Arian and Montanari, Andrea},
  journal = {Proceedings of the National Academy of Sciences},
  volume = {106},
  number = {45},
  pages = {18914--18919},
  year = {2009},
  publisher = {National Acad Sciences}
}

@article{BM-amp,
  title = {The dynamics of message passing on dense graphs, with applications to compressed sensing},
  author = {Bayati, Mohsen and Montanari, Andrea},
  journal = {IEEE Transactions on Information Theory},
  volume = {57},
  number = {2},
  pages = {764--785},
  year = {2011},
  publisher = {IEEE}
}

@article{FR-amp,
  title = {Iterative reconstruction of rank-one matrices in noise},
  author = {Fletcher, Alyson K and Rangan, Sundeep},
  journal = {Information and Inference: A Journal of the IMA},
  volume = {7},
  number = {3},
  pages = {531--562},
  year = {2018},
  publisher = {Oxford University Press}
}

@article{MV-amp,
  title = {Estimation of low-rank matrices via approximate message passing},
  author = {Montanari, Andrea and Venkataramanan, Ramji},
  journal = {The Annals of Statistics},
  volume = {49},
  number = {1},
  pages = {321--345},
  year = {2021}
}

@article{miolane-survey,
  title = {Phase transitions in spiked matrix estimation: information-theoretic analysis},
  author = {Miolane, L{\'e}o},
  journal = {arXiv preprint arXiv:1806.04343},
  year = {2018}
}

@article{HM-tree,
  title = {Low degree hardness for broadcasting on trees},
  author = {Huang, Han and Mossel, Elchanan},
  journal = {arXiv preprint arXiv:2402.13359},
  year = {2024}
}

@article{submatrix-amp,
  title = {Submatrix localization via message passing},
  author = {Hajek, Bruce and Wu, Yihong and Xu, Jiaming},
  journal = {Journal of Machine Learning Research},
  volume = {18},
  number = {186},
  pages = {1--52},
  year = {2018}
}

@article{BI-info,
  title = {Detection of a sparse submatrix of a high-dimensional noisy matrix},
  author = {Butucea, C and Ingster, YI},
  journal = {Bernoulli: a Journal of Mathematical Statistics and Probability},
  volume = {19},
  number = {5 B},
  pages = {2652--2688},
  year = {2013},
  publisher = {International Statistical Institute}
}

@article{kolar-info,
  title = {Minimax localization of structural information in large noisy matrices},
  author = {Kolar, Mladen and Balakrishnan, Sivaraman and Rinaldo, Alessandro and Singh, Aarti},
  journal = {Advances in Neural Information Processing Systems},
  volume = {24},
  year = {2011}
}

@article{BIS-info,
  title = {Sharp variable selection of a sparse submatrix in a high-dimensional noisy matrix},
  author = {Butucea, Cristina and Ingster, Yuri I and Suslina, Irina A},
  journal = {ESAIM: Probability and Statistics},
  volume = {19},
  pages = {115--134},
  year = {2015},
  publisher = {EDP Sciences}
}

@article{clique-e,
  title = {Finding hidden cliques of size $\sqrt{N/e}$ in nearly linear time},
  author = {Deshpande, Yash and Montanari, Andrea},
  journal = {Foundations of Computational Mathematics},
  volume = {15},
  pages = {1069--1128},
  year = {2015},
  publisher = {Springer}
}

@article{alon-clique,
  title = {Finding a large hidden clique in a random graph},
  author = {Alon, Noga and Krivelevich, Michael and Sudakov, Benny},
  journal = {Random Structures \& Algorithms},
  volume = {13},
  number = {3-4},
  pages = {457--466},
  year = {1998},
  publisher = {Wiley Online Library}
}

@inproceedings{log-density,
  title = {Detecting high log-densities: an $O(n^{1/4})$ approximation for densest $k$-subgraph},
  author = {Bhaskara, Aditya and Charikar, Moses and Chlamtac, Eden and Feige, Uriel and Vijayaraghavan, Aravindan},
  booktitle = {Proceedings of the forty-second ACM Symposium on Theory of Computing},
  pages = {201--210},
  year = {2010}
}

@article{AV-info,
  title = {Community detection in dense random networks},
  author = {Arias-Castro, Ery and Verzelen, Nicolas},
  journal = {The Annals of Statistics},
  volume = {42},
  number = {3},
  pages = {940--969},
  year = {2014},
  publisher = {Institute of Mathematical Statistics},
  doi = {10.1214/14-AOS1208}
}

@article{VA-info,
  title = {Community detection in sparse random networks},
  author = {Verzelen, N and Arias-Castro, E},
  journal = {Annals of Applied Probability},
  volume = {25},
  number = {6},
  pages = {3465--3510},
  year = {2015},
  publisher = {Institute of Mathematical Statistics}
}

@article{ames-convex,
  title = {Guaranteed recovery of planted cliques and dense subgraphs by convex relaxation},
  author = {Ames, Brendan P. W.},
  journal = {Journal of Optimization Theory and Applications},
  volume = {167},
  number = {2},
  pages = {653--675},
  year = {2015},
  publisher = {Springer},
  doi = {10.1007/s10957-015-0777-x}
}

@article{CX-info,
  title = {Statistical-computational tradeoffs in planted problems and submatrix localization with a growing number of clusters and submatrices},
  author = {Chen, Yudong and Xu, Jiaming},
  journal = {Journal of Machine Learning Research},
  volume = {17},
  number = {27},
  pages = {1--57},
  year = {2016}
}

@article{LM-wigner,
  title = {Fundamental limits of symmetric low-rank matrix estimation},
  author = {Lelarge, Marc and Miolane, L{\'e}o},
  journal = {Probability Theory and Related Fields},
  volume = {173},
  pages = {859--929},
  year = {2019},
  publisher = {Springer}
}

@article{proof-replica,
  title = {Mutual information for symmetric rank-one matrix estimation: A proof of the replica formula},
  author = {Dia, Mohamad and Macris, Nicolas and Krzakala, Florent and Lesieur, Thibault and Zdeborov{\'a}, Lenka and others},
  journal = {Advances in Neural Information Processing Systems},
  volume = {29},
  year = {2016}
}

@article{abbe-survey-sbm,
  title = {Community detection and stochastic block models: recent developments},
  author = {Abbe, Emmanuel},
  journal = {Journal of Machine Learning Research},
  volume = {18},
  number = {177},
  pages = {1--86},
  year = {2018}
}

@article{AS-acyclic,
  title = {Proof of the achievability conjectures for the general stochastic block model},
  author = {Abbe, Emmanuel and Sandon, Colin},
  journal = {Communications on Pure and Applied Mathematics},
  volume = {71},
  number = {7},
  pages = {1334--1406},
  year = {2018},
  publisher = {Wiley Online Library}
}

@inproceedings{kunisky24tensor,
  author = {Kunisky, Dmitriy and Moore, Cristopher and Wein, Alexander S.},
  booktitle = {2024 IEEE 65th Annual Symposium on Foundations of Computer Science (FOCS)},
  title = {Tensor cumulants for statistical inference on invariant distributions},
  year = {2024},
  volume = {},
  number = {},
  pages = {1007-1026},
  doi = {10.1109/FOCS61266.2024.00067}, 
  note = {arXiv version available at arXiv:2404.18735.}
}

@article{opt-bot,
  title = {Optimal low degree hardness for broadcasting on trees},
  author = {Huang, Han and Mossel, Elchanan},
  journal = {arXiv preprint arXiv:2502.04861},
  year = {2025}
}

@article{one-community-sparse,
  title = {Finding one community in a sparse graph},
  author = {Montanari, Andrea},
  journal = {Journal of Statistical Physics},
  volume = {161},
  pages = {273--299},
  year = {2015},
  publisher = {Springer}
}

@article{color-coding,
  title = {Color-coding},
  author = {Alon, Noga and Yuster, Raphael and Zwick, Uri},
  journal = {Journal of the ACM (JACM)},
  volume = {42},
  number = {4},
  pages = {844--856},
  year = {1995},
  publisher = {ACM New York, NY, USA}
}

@article{ld-survey,
  title = {Computational complexity of statistics: new insights from low-degree polynomials},
  author = {Wein, Alexander S},
  journal = {arXiv preprint arXiv:2506.10748},
  year = {2025}
}

@article{BHP-19,
  author = {Mark Budden and Josh Hiller and Andrew Penland},
  title = {Minimally connected $r$-uniform hypergraphs},
  journal = {Australasian Journal of Combinatorics},
  volume = {82},
  number = {1},
  pages = {1--20},
  year = {2022}
}

@article{DMW-23,
  author = {Abhishek Dhawan and Cheng Mao and Alexander S. Wein},
  title = {Detection of dense subhypergraphs by low-degree polynomials},
  journal = {Random Structures \& Algorithms},
  volume = {66},
  number = {1},
  pages = {e21279},
  year = {2025},
  doi = {10.1002/rsa.21279}
}

@article{lavault2011note,
  title={A note on Pr\"{u}fer-like coding and counting forests of uniform hypertrees},
  author={Lavault, Christian},
  journal={arXiv preprint arXiv:1110.0204},
  year={2011}
}

@article{RV-88,
  author = {Andrzej Ruci{\'n}ski and Andrew Vince},
  title = {Balanced extensions of graphs and hypergraphs},
  journal = {Combinatorica},
  volume = {8},
  number = {3},
  pages = {279--291},
  year = {1988},
  doi = {10.1007/BF02126800}
}

@article{SW-22,
  author = {Tselil Schramm and Alexander S. Wein},
  title = {Computational barriers to estimation from low-degree polynomials},
  journal = {The Annals of Statistics},
  volume = {50},
  number = {3},
  pages = {1833--1858},
  year = {2022},
  doi = {10.1214/22-AOS2179}
}

@article{arxiv-version,
  title = {Sharp phase transitions in estimation with low-degree polynomials},
  author = {Sohn, Youngtak and Wein, Alexander S},
  journal = {arXiv preprint arXiv:2502.14407},
  year = {2025}, 
  note = {Conference version appeared in Proceedings of the 57th Annual ACM Symposium on Theory of Computing}
}

@book{moon1970counting,
  author = {J. W. Moon},
  title = {Counting labelled trees},
  series = {Canadian Mathematical Monographs},
  publisher = {Canadian Mathematical Congress},
  year = {1970}
}

@article{Bedini2008,
  doi = {10.1088/1751-8113/41/20/205003},
  url = {https://doi.org/10.1088/1751-8113/41/20/205003},
  year = {2008},
  month = {apr},
  publisher = {},
  volume = {41},
  number = {20},
  pages = {205003},
  author = {Bedini, Andrea and Caracciolo, Sergio and Sportiello, Andrea},
  title = {Hyperforests on the complete hypergraph by Grassmann integral representation},
  journal = {Journal of Physics A: Mathematical and Theoretical}}

@article{perry-2020,
  author = {Amelia Perry and Alexander S. Wein and Afonso S. Bandeira},
  title = {{Statistical limits of spiked tensor models}},
  volume = {56},
  journal = {Annales de l'Institut Henri Poincaré, Probabilités et Statistiques},
  number = {1},
  publisher = {Institut Henri Poincaré},
  pages = {230 -- 264},
  year = {2020},
  doi = {10.1214/19-AIHP960},
  URL = {https://doi.org/10.1214/19-AIHP960}
}

@article{lovig-2025,
  title = {Almost-optimal local-search methods for sparse tensor {PCA}},
  author = {Lovig, Max and Sheehan, Conor and Tsirkas, Konstantinos and Zadik, Ilias},
  journal = {arXiv preprint arXiv:2506.09959},
  year = {2025},
  archivePrefix = {arXiv},
  eprint = {2506.09959},
  primaryClass = {math.ST},
  doi = {10.48550/arXiv.2506.09959},
  note = {Version 1}
}

@article{chen-2024,
  title = {On the low-temperature MCMC threshold: the cases of sparse tensor {PCA}, sparse regression, and a geometric rule},
  author = {Chen, Zongchen and Sheehan, Conor and Zadik, Ilias},
  journal = {arXiv preprint arXiv:2408.00746},
  year = {2024},
  archivePrefix = {arXiv},
  eprint = {2408.00746},
  primaryClass = {math.ST},
  doi = {10.48550/arXiv.2408.00746}
}

@article{arash-2008,
  title = {High-dimensional analysis of semidefinite relaxations for sparse principal components},
  author = {Amini, Arash A. and Wainwright, Martin J.},
  journal = {The Annals of Statistics},
  volume = {37},
  number = {5B},
  pages = {2877--2921},
  year = {2009},
  publisher = {Institute of Mathematical Statistics},
  doi = {10.1214/08-AOS664}
}

@article{jagannath-2020,
  title   = {Statistical thresholds for tensor {PCA}},
  author  = {Jagannath, Aukosh and Lopatto, Patrick and Miolane, L{\'e}o},
  journal = {Annals of Applied Probability},
  volume  = {30},
  number  = {4},
  pages   = {1910--1933},
  year    = {2020},
  doi     = {10.1214/19-AAP1547}
}

@inproceedings{feldman-2023,
  title     = {Sharp recovery thresholds of tensor {PCA} spectral algorithms},
  author    = {Feldman, Michael Jacob and Donoho, David},
  booktitle = {Advances in Neural Information Processing Systems},
  year      = {2023},
  url       = {https://openreview.net/forum?id=b1BhHjBxsx}
}

@article{li-2025,
  title     = {A smooth computational transition in tensor PCA},
  author    = {Li, Zhangsong},
  journal   = {arXiv preprint arXiv:2509.09904},
  year      = {2025}
}

@article{yuan-2021,
  author  = {Yuan, Mingao and Shang, Zuofeng},
  title   = {Information limits for detecting a subhypergraph},
  journal = {Stat},
  volume  = {10},
  number  = {1},
  pages   = {e407},
  year    = {2021},
  doi     = {10.1002/sta4.407}
}

@book{vershynin-notes,
  author    = {Vershynin, Roman},
  title     = {High-dimensional probability: An introduction with applications in data science},
  edition   = {2},
  publisher = {Cambridge University Press},
  year      = {2026},
}

@article{tsirkas-2026,
  title         = {The monotonicity of the Franz-Parisi potential is equivalent with low-degree {MMSE} lower bounds},
  author        = {Tsirkas, Konstantinos and Wang, Leda and Zadik, Ilias},
  journal       = {arXiv preprint arXiv:2603.20070}, 
  year          = {2026},
  primaryClass  = {math.ST}
}

@article{hajek2017recovery,
  author  = {Hajek, Bruce and Wu, Yihong and Xu, Jiaming},
  title   = {Information limits for recovering a hidden community},
  journal = {IEEE Transactions on Information Theory},
  volume  = {63},
  number  = {8},
  pages   = {4729--4745},
  year    = {2017},
  doi     = {10.1109/TIT.2017.2653804}
}

@inproceedings{bresler-2023,
  author    = {Bresler, Guy and Jiang, Tianze},
  title     = {Detection-recovery and detection-refutation gaps via reductions from planted clique},
  booktitle = {Proceedings of the Thirty Sixth Conference on Learning Theory},
  series    = {Proceedings of Machine Learning Research},
  volume    = {195},
  pages     = {5850--5889},
  year      = {2023},
  publisher = {PMLR}
}

@inproceedings{bresler-2020,
  author    = {Brennan, Matthew and Bresler, Guy},
  title     = {Reducibility and statistical-computational gaps from secret leakage},
  booktitle = {Proceedings of the Thirty Third Conference on Learning Theory},
  series    = {Proceedings of Machine Learning Research},
  volume    = {125},
  pages     = {648--847},
  year      = {2020},
  publisher = {PMLR}
}

@article{FranzParisi95,
  author  = {Franz, Silvio and Parisi, Giorgio},
  title   = {Recipes for metastable states in spin glasses},
  journal = {Journal de Physique I},
  volume  = {5},
  number  = {11},
  pages   = {1401--1415},
  year    = {1995},
  doi     = {10.1051/jp1:1995201}
}

@article{FranzParisi97,
  author  = {Franz, Silvio and Parisi, Giorgio},
  title   = {Phase diagram of coupled glassy systems: A mean-field study},
  journal = {Physical Review Letters},
  volume  = {79},
  number  = {13},
  pages   = {2486--2489},
  year    = {1997},
  doi     = {10.1103/PhysRevLett.79.2486}
}

\end{document}